\documentclass[11pt]{article}

   \renewcommand{\footnote}[1]{
\textsuperscript{
\addtocounter{footnote}{1}
(\thefootnote)
}
\footnotetext{#1}
}

\usepackage{bm}
\usepackage{amsmath}
\usepackage{amsthm}
\usepackage{trfsigns}
\usepackage{wasysym}
\usepackage{amssymb}
\usepackage{amsfonts}
\usepackage{graphicx}  
\usepackage[frenchb,english]{babel}
\usepackage[utf8]{inputenc}
\usepackage[T1]{fontenc}
\usepackage{mathrsfs}
\usepackage{pdfsync}
\usepackage{enumerate}
\usepackage{version}
\usepackage{calc}
\usepackage{subfigure}

\usepackage{hyperref}
 \usepackage{pstricks,pstricks-add,pst-math,pst-xkey}
 \usepackage{float}
 \usepackage{indentfirst}

\numberwithin{equation}{section}

\newtheorem{prop}{Proposition}

\newtheorem{thm}{Theorem}
\newtheorem{lem}[prop]{Lemma}
\newtheorem{cor}[prop]{Corollary}

\theoremstyle{remark}
\newtheorem{remark}[prop]{Remark}
\newtheorem{notation}[prop]{Notation}
\newtheorem{construction}[prop]{Construction}

\def\a{\alpha}
\def\b{\beta}
\def\ve{\varepsilon}
\def\v{\varepsilon}

\def\O{\Omega}
\def\d{\partial}
\def\p{\partial}

\def\w{\omega}
\def\o{\omega}
\def\R{\mathbb{R}}
\def\C{\mathbb{C}}
\def\S{\mathbb{S}}
\def\Div{{\rm  div}}
\def\deg{ {\rm deg}}
\def\n{\nabla}
\def\dist{{\rm dist}}
\def\tr{{\rm tr}}
\def\f{\varphi}
\def\weak{\rightharpoonup}

\def\N{\mathbb{N}}
\def\Z{\mathbb{Z}}
\def\J{\mathcal{J}}
\def\I{\mathcal{I}}
\def\M{\mathcal{M}}
\def\xm{\xi_{\rm meso}}

\def\1{\textrm{1\kern-0.25emI}}
\def\num{P}
\def\Ring{\mathscr{R}}
\def\H{\mathscr{H}}
\def\RHO{S_\rho}
\def\WM{\widehat{\mathcal{M}}}
\title{The Ginzburg-Landau functional with a discontinuous and rapidly oscillating pinning term. Part II: the non-zero degree case}
\author{Mickaël {\sc Dos Santos}\footnote{Université Paris Est-Créteil, 61 avenue du Général de Gaulle, 94010 Créteil Cedex }\\\url{mickael.dos-santos@u-pec.fr}
}

\begin{document}
\maketitle
\begin{abstract}
We consider minimizers of a Ginzburg-Landau energy with a discontinuous and rapidly oscillating pinning term, subject to a  Dirichlet boundary condition of degree $d>0$. The pinning term models an unbounded number of small impurities in the domain. We prove that for strongly type II superconductor with impurities, minimizers have exactly $d$ isolated zeros (vortices). These vortices are of degree $1$ and  pinned by the impurities. As in the standard case studied by Bethuel, Brezis and Hélein, the macroscopic location of vortices is governed by vortex/vortex and vortex/ boundary repelling effects. In some special cases we prove that their macroscopic location tends to minimize the renormalized energy of Bethuel-Brezis-Hélein.  In addition, impurities affect the microscopic location of vortices. Our technics allows us to work with  impurities having different size. In this situation we prove that vortices are pinned by the largest impurities.
\end{abstract}
\tableofcontents
\section{Introduction}

In this article we let $\O\subset\R^2$ be a smooth simply connected domain and let $a_\v:\O\to\{b,1\},\,b\in(0,1)$ be a measurable function. We associate to $a_\v$ the pinned Ginzburg-Landau energy
\begin{equation}\label{8.PinnedGLFunctional}
E_\v(u)=\frac{1}{2}\int_\O\left\{|\n u(x)|^2+\frac{1}{2\v^2}\left(a_\v(x)^2-|u(x)|^2\right)^2\right\}\,{\rm d}x.
\end{equation}
Here, $u:\O\to\C$ is in the Sobolev space $H^1(\O,\C)$ and $\v>0$ is the inverse of the Ginzburg-Landau parameter. 

Our goal is to consider a discontinuous and rapidly oscillating pinning term (the pinning term is $a_\v:\O\to\{b,1\}$). Our pinning term is periodic with respect to a $\delta\times\delta$-grid with $\delta=\delta(\v)\to0$ as $\v\to0$ (in some cases we drop the periodic hypothesis). 

We are interested in the minimization of \eqref{8.PinnedGLFunctional} in $H^1(\O,\C)$ subject to a Dirichlet boundary condition: we fix $g\in C^\infty(\p\O,\S^1)$ and thus the set of the test functions is
\[
H^1_g:=\{u\in H^1(\O,\C)\,|\,\tr_{\p\O}u=g\}.
\]

The situation where $d=\deg_{\p\O}(g)=0$ was studied in detail in \cite{publi2}. The non zero degree case ($d=\deg_{\p\O}(g)>0$) is the purpose of the present article. Recall that for $\Gamma\subset\R^2$ a Jordan curve and $g\in H^{1/2}(\Gamma,\S^1)$, the degree (winding number) of $g$ is defined as
\[
\deg_{\Gamma}(g):=\frac{1}{2\pi}\int_{\Gamma}g\times\p_\tau g\,{\rm d}\tau.
\]
Here "$\times$'' stands for the vectorial product in $\mathbb{C}$, \emph{i.e.} $z_1 \times z_2= {\rm Im}(\overline{z_1}z_2)$, $z_1,z_2\in\mathbb{C}$, $\tau$ is the direct unit tangent vector of $\Gamma$ ($\tau=\nu^\bot$ where $\nu$ is the outward normal unit vector of ${\rm int}(\Gamma)$, the bounded open set whose boundary is $\Gamma$) and $\p_\tau$ is the tangential derivative on $\Gamma$.

This energy is a simplification of the full Ginzburg-Landau energy (see Eq. \eqref{FullGinzburgLandauFunctionalPinning} below) whose minimizers model the state of a Type II superconductor (the parameter $\v$ corresponds to a material parameter, this parameter is small for Type II superconductor) \cite{Tin1}, \cite{SS1}. The pinning term allows to model a heterogenous superconductor (see \cite{K1} or Introduction of \cite{TheseDosSantos}).

Physical informations which can be obtained with the simplification of the full Ginzburg-Landau energy are quantization and location of zeros of minimizers. Their zeros represent the centers of small areas where the superconductivity is destroyed. These areas are called vorticity defects. Here the superconductor is a cylinder whose cross section is $\O$ and the vorticity defects (under some special conditions) takes the form of small wires parallel to the superconductor \cite{Tin1}, \cite{SS1}.


Before going further, let us summarize two previous works in related directions \cite{LM1}, \cite{ASS1}. In these works, the role of the pinning term is identified: its points of {\it minimum}  attract the vorticity defects.

In \cite{LM1}, Lassoued and Mironescu considered the case where $a_\v\equiv a$. Here, the pinning term $a=\begin{cases}b&\text{in }\o\\1&\text{in }\O\setminus\o\end{cases}$, $0<b<1$, and $\o$ is a smooth inner domain of $\O$. These authors proved that the vorticity defects are quantified by $\deg_{\p\O}(g)$, localized in $\o$ and that their position is governed by a renormalized energy (in the spirit of \cite{BBH}). 


In \cite{ASS1}, Aftalion, Sandier and Serfaty considered a smooth and $\v$-dependent pinning term $a_\v$. Their study allows to consider the case where the pinning term has fast oscillations: it is a perturbation of a fixed smooth function $\tilde{b}:\O\to[b,1]$ s.t. $a_\v\geq \tilde{b}$. 

In contrast with \cite{LM1}, \cite{ASS1} is dedicated to the study of a full Ginzburg-Landau energy $GL_\v$ with the pinning term $a_\v$

\begin{equation}\label{FullGinzburgLandauFunctionalPinning}
GL_\v(u,A)= \frac{1}{2}\int_{\Omega } \left\{|{\rm curl} A -  h_{\rm ex}|^2+|(\nabla - i   A)  u|^2 + \frac{1}{2\v^2}(a_\v^2-| u|^2)^2\right\}.
\end{equation}


We denoted by $A\in\R^2$ the electromagnetic vector potential of the induced field and by $h_{\rm ex}\gg1$ the intensity of the applied magnetic field (see \cite{SS1} for more details). 

They considered the following hypotheses on $a_\v,\tilde{b}$:
\begin{enumerate}[$\bullet$]
\item $|\n a_\v|\leq C h_{\rm ex}$ 
\item  there is $\sigma_\v\in\R$ s.t. $\sigma_\v=o\left((\ln|\ln\v|)^{-1/2}\right)$ and  for all $x\in\O$, we have
\[
\min_{B(x,\sigma_\v)}\left\{a_\v-\tilde{b}\right\}=0.
\]
\end{enumerate}

In the study of the full Ginzburg-Landau functional without pinning term $GL^{0}_\v$ ($GL^0_\v$ is obtained from \eqref{FullGinzburgLandauFunctionalPinning} by taking $a_\v\equiv1$), the vorticity defects appear for large apply magnetic field. They are characterized by two facts: the presence of isolated zeros $x_i$ of a map $u$ with a non zero degree around small circles centered in $x_i$ and the existence of a magnetic field inside the domain (${\rm curl}(A)\simeq h_{\rm ex}$ inside small discs).  The nature of the superconductivity makes that both facts appear together. Assume that the intensity of the applied field $h_{\rm ex}$ depends on $0<\v<1$ and that $h_{\rm ex}/|\ln\v|\to\Lambda\in\R^*_+$. For the full Ginzburg-Landau energy without pinning term $GL^{0}_\v$, it is well known (see \emph{e.g.} \cite{SS1}) that there is an inner domain $\o_\Lambda$ (non decreasing w.r.t. $\Lambda$) s.t., when $\v\to 0$, the vorticity defects are "uniformly located" by $\o_\Lambda$ (in this situation the number of vortices is unbounded).

In \cite{ASS1} (study of a full Ginzburg-Landau functional with a pinning term), the authors proved the existence of $\o_\Lambda$, an inner set of $\O$, where the penetration of the magnetic field is located. In contrast with the situation without pinning term, the presence of $a_\v$ makes that, in general, the vortices are not uniformly located in $\o_\Lambda$. Although in the proofs of the main results of \cite{ASS1}, the minimal points of $\tilde{b}$ seem play the role of a pinning site, this fact is not proved. They expect that the most favorable pinning sites should be close to the \emph{minima} of $\tilde{b}$ : $\o_\Lambda$ should be located close to the points of \emph{minimum} of $\tilde{b}$.

One of our goals is to prove that the minimum points of a rapidly oscillating and discontinuous pinning term attract the vorticity defects.

Before going further, we construct our (periodic) pinning term $a_\v$. 

\begin{construction}\label{Constru.ConstrucPinnTerm}The periodic pinning term\\
Consider
\begin{enumerate}[$\bullet$]\item $\delta=\delta(\v)\in(0,1),\,\lambda=\lambda(\v)\in(0,1]$;
 \item $\o\subset Y=(-1/2,1/2)^2$ be a smooth bounded and simply connected open set s.t. $(0,0)\in\o$ and $\overline{\o}\subset Y$ (here $Y$ is the unit cell). 
\end{enumerate} For $k,l\in\Z$ we denote 
\begin{center}
\hfill$Y_{k,l}^\delta:=\delta\cdot Y+(\delta k,\delta l)$,\hfill$\O^{\rm incl}_\delta=\bigcup_{Y_{k,l}^\delta\subset\O}\overline{Y_{k,l}^\delta}$,\hfill$\o^\lambda=\lambda\cdot\o$, \hfill\phantom{a}
\end{center}
\[
\o^\lambda_{\rm per}=\bigcup_{(k,l)\in\Z^2}\left\{\w^\lambda+(k,l)\right\}\text{ and }\o_\v=\bigcup_{\substack{(k,l)\in\Z^2\text{ s.t.}\\Y_{k,l}^\delta\subset\O}}\left\{\delta\cdot\w^\lambda+(\delta k,\delta l)\right\}.
\]

For $b\in(0,1)$, we define 
\[
\begin{array}{cccc}
a^\lambda:&\R^2&\to&\{b,1\}\\&x&\mapsto&\begin{cases}b&\text{if }x\in\o^\lambda_{\rm per}\\1&\text{otherwise}\end{cases}
\end{array}
\text{ and }
\begin{array}{cccc}
a_\v:&\R^2&\to&\{b,1\}\\&x&\mapsto&\begin{cases}b&\text{if }x\in\o_\v\\1&\text{otherwise}\end{cases}
\end{array}.
\]
The values of the periodic pinning term are represented Figure \ref{Intro.FigureTermeChevillage}. The connected components of $\{a_\v=b\}=\o_\v$ are called inclusions or impurities.
\end{construction}
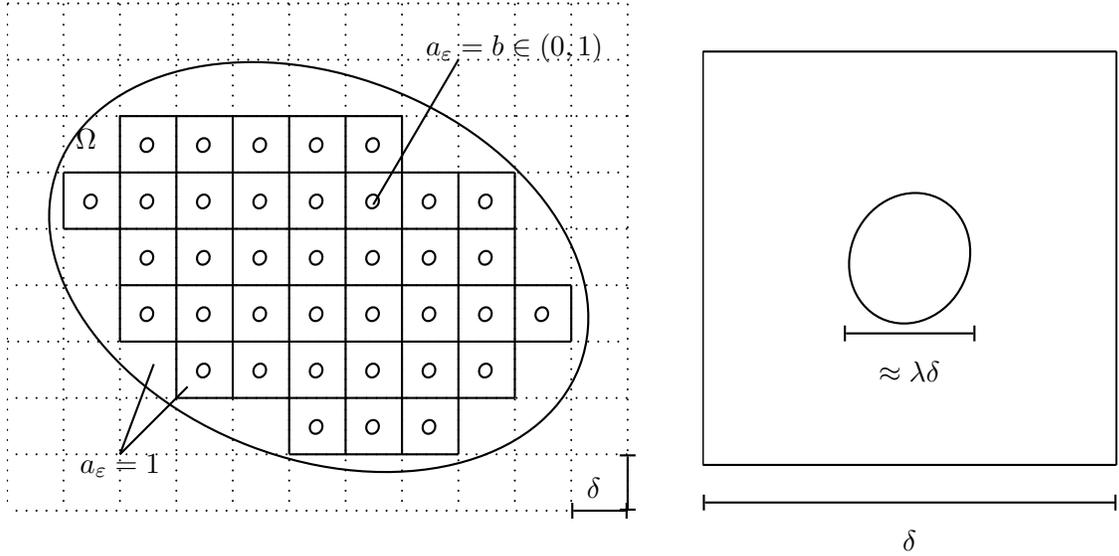
\begin{figure}[h]
\subfigure[The pining term is periodic on a $\delta\times\delta$-grid]{
\psset{xunit=0.15cm,yunit=0.15cm,algebraic=true,dotstyle=*,dotsize=3pt 0,linewidth=0.8pt,arrowsize=3pt 2,arrowinset=0.25}
\begin{pspicture*}(-20,-14.3)(36,40)
\rput{64.54}(2.39,2.43){\psellipse(0,0)(0.72,0.64)}
\psline(12.45,17.15)(20,30)
\rput(25,31){$a_\v=b\in(0,1)$}
\psline(-7,3)(-10,-5)
\rput(-10,-6){$a_\v=1$}
\psline(-4,1)(-10,-5)

\rput{64.54}(7.39,2.43){\psellipse(0,0)(0.72,0.64)}
\rput{64.54}(12.39,2.43){\psellipse(0,0)(0.72,0.64)}
\rput{64.54}(2.39,7.43){\psellipse(0,0)(0.72,0.64)}
\rput{64.54}(2.39,12.43){\psellipse(0,0)(0.72,0.64)}
\rput{64.54}(2.39,17.43){\psellipse(0,0)(0.72,0.64)}
\rput{64.54}(7.39,17.43){\psellipse(0,0)(0.72,0.64)}
\rput{64.54}(12.39,17.43){\psellipse(0,0)(0.72,0.64)}
\rput{64.54}(17.39,17.43){\psellipse(0,0)(0.72,0.64)}
\rput{64.54}(17.39,12.43){\psellipse(0,0)(0.72,0.64)}
\rput{64.54}(7.39,12.43){\psellipse(0,0)(0.72,0.64)}
\rput{64.54}(7.39,7.43){\psellipse(0,0)(0.72,0.64)}
\rput{64.54}(12.39,12.43){\psellipse(0,0)(0.72,0.64)}
\rput{64.54}(12.39,7.43){\psellipse(0,0)(0.72,0.64)}
\rput{64.54}(17.39,7.43){\psellipse(0,0)(0.72,0.64)}
\rput{64.54}(17.39,2.43){\psellipse(0,0)(0.72,0.64)}
\rput{-20.94}(7.62,11.61){\psellipse(0,0)(24.83,17.1)}

\rput{64.54}(-2.61,7.43){\psellipse(0,0)(0.72,0.64)}
\rput{64.54}(-7.61,7.43){\psellipse(0,0)(0.72,0.64)}
\rput{64.54}(-2.61,2.43){\psellipse(0,0)(0.72,0.64)}
\rput{64.54}(-2.61,12.43){\psellipse(0,0)(0.72,0.64)}
\rput{64.54}(-2.61,17.43){\psellipse(0,0)(0.72,0.64)}
\rput{64.54}(-2.61,22.43){\psellipse(0,0)(0.72,0.64)}
\rput{64.54}(-7.61,22.43){\psellipse(0,0)(0.72,0.64)}
\rput{64.54}(-7.61,17.43){\psellipse(0,0)(0.72,0.64)}
\rput{64.54}(-7.61,12.43){\psellipse(0,0)(0.72,0.64)}
\rput{64.54}(-12.61,17.43){\psellipse(0,0)(0.72,0.64)}
\rput{64.54}(2.39,22.43){\psellipse(0,0)(0.72,0.64)}
\rput{64.54}(7.39,22.43){\psellipse(0,0)(0.72,0.64)}
\rput{64.54}(12.39,22.43){\psellipse(0,0)(0.72,0.64)}
\rput{64.54}(22.39,12.43){\psellipse(0,0)(0.72,0.64)}
\rput{64.54}(22.39,7.43){\psellipse(0,0)(0.72,0.64)}
\rput{64.54}(27.39,7.43){\psellipse(0,0)(0.72,0.64)}
\rput{64.54}(22.39,2.43){\psellipse(0,0)(0.72,0.64)}
\rput{64.54}(7.39,-2.57){\psellipse(0,0)(0.72,0.64)}
\rput{64.54}(12.39,-2.57){\psellipse(0,0)(0.72,0.64)}
\rput{64.54}(17.39,-2.57){\psellipse(0,0)(0.72,0.64)}
\psline[linestyle=dotted,dash=18pt 18pt](-20,30)(35,30)
\psline[linestyle=dotted,dash=18pt 18pt](35,25)(-20,25)
\psline[linestyle=dotted,dash=18pt 18pt](-20,20)(35,20)
\psline[linestyle=dotted,dash=18pt 18pt](35,15)(-20,15)
\psline[linestyle=dotted,dash=18pt 18pt](35,10)(-20,10)
\psline[linestyle=dotted,dash=18pt 18pt](35,0)(-20,0)
\psline[linestyle=dotted,dash=18pt 18pt](35,-5)(-20,-5)
\psline[linestyle=dotted,dash=18pt 18pt](35,-10)(-20,-10)
\psline[linestyle=dotted,dash=18pt 18pt](35,35)(35,-10)
\psline[linestyle=dotted,dash=18pt 18pt](30,35)(30,-10)
\psline[linestyle=dotted,dash=18pt 18pt](25,35)(25,-10)
\psline[linestyle=dotted,dash=18pt 18pt](20,35)(20,-10)
\psline[linestyle=dotted,dash=18pt 18pt](15,35)(15,-10)
\psline[linestyle=dotted,dash=18pt 18pt](10,35)(10,-10)
\psline[linestyle=dotted,dash=18pt 18pt](5,35)(5,-10)
\psline[linestyle=dotted,dash=18pt 18pt](0,35)(0,-10)
\psline[linestyle=dotted,dash=18pt 18pt](-5,35)(-5,-10)
\psline[linestyle=dotted,dash=18pt 18pt](-10,35)(-10,-10)
\psline[linestyle=dotted,dash=18pt 18pt](-15,35)(-15,-10)
\psline[linestyle=dotted,dash=18pt 18pt](-20,35)(-20,-10)
\psline[linestyle=dotted,dash=18pt 18pt](-20,35)(35,35)
\psline[linestyle=dotted,dash=18pt 18pt](-20,5)(35,5)
\psline(5,-5)(20,-5)
\psline(25,0)(-5,0)
\psline(-10,5)(30,5)
\psline(-10,10)(30,10)
\psline(25,15)(-15,15)
\psline(-15,20)(25,20)
\psline(-10,25)(15,25)

\psline(-15,20)(-15,15)
\psline(-10,5)(-10,25)
\psline(-5,0)(-5,25)
\psline(0,0)(0,25)
\psline(5,-5)(5,25)
\psline(10,25)(10,-5)
\psline(15,25)(15,-5)
\psline(20,-5)(20,20)
\psline(25,0)(25,20)
\psline(30,5)(30,10)

\rput{64.54}(22.39,17.43){\psellipse(0,0)(0.72,0.64)}
\rput(32,-8){$\delta$}
\psline{|-|}(30,-10)(35,-10)
\psline{|-|}(35,-10)(35,-5)
\rput(-13,23){$\O$}
\end{pspicture*}
}\hfill
 \subfigure[The parameter $\lambda$ controls the size of an inclusion in the cell]
 {
\psset{xunit=0.25cm,yunit=0.25cm,algebraic=true,dotstyle=*,dotsize=3pt 0,linewidth=0.8pt,arrowsize=3pt 2,arrowinset=0.25}
\begin{pspicture*}(38,4)(60,35)
\psline(38,31)(60,31)
\psline(38,31)(38,9)
\psline(38,9)(60,9)
\psline(60,9)(60,31)
\rput{64.54}(49,20){\psellipse(0,0)(3.6,3.2)}
\psline{|-|}(38,7)(60,7)
\rput(49,5){$\delta$}
\psline{|-|}(45.5,16)(52.5,16)
\rput(49,14){$\approx\lambda\delta$}
\end{pspicture*}
}
\caption{The periodic pinning term}\label{Intro.FigureTermeChevillage}
\end{figure} 


In the rest of this article $\lambda=\lambda(\v)$ and $\delta=\delta(\v)$ are functions of $\v$. We assume that $\delta\to0$ as $\v\to0$. In addition, we assume that either $\lambda\equiv1$, or $\lambda\to0$ as $\v\to0$. The case $\lambda\to0$ is the \emph{diluted case}.


We make the (technical) assumption
\begin{equation}\label{8.MainHyp}
\lim_\v\frac{|\ln(\lambda\delta)|^3}{|\ln \v|}=0.
\end{equation}
\begin{remark}
\begin{enumerate}[$\bullet$]
\item This is slightly more restrictive than asking that $\lambda\delta\gg\v^\alpha$ for all $\alpha\in(0,1)$.
\item Hypothesis \eqref{8.MainHyp} is technical, a more natural hypothesis should be $\lambda\delta\gg\v$ or $\lambda\delta\gg\v^\alpha$ for some $\alpha\in(0,1)$.
\item In \cite{ASS1} and in the situation where we have a bounded number of zeros (the applied magnetic field is not too large), the smooth pinning term $a^{0}_\v$ satisfies the condition $|\n a^{0}_\v|\leq C|\ln\v|$. In order to compare this assumption with \eqref{8.MainHyp}, we may consider a regularization of our pinning term by a mollifier $\rho_t(x)=t^{-2}\rho(x/t)$. A suitable scale $t$ to have a complete view of the variations of $a_\v$ is $t=\lambda\delta$. Thus, $|\n (\rho_{\lambda\delta}\ast a_\v)|$ is of order $\dfrac{1}{\lambda\delta}$. Consequently, the condition \eqref{8.MainHyp} allows to consider a more rapidly oscillating  than the condition in \cite{ASS1}. Indeed, we have $\ln|\n a^{0}_\v|\apprle \ln|\ln\v|$ and on the other hand \eqref{8.MainHyp} is equivalent to $\ln|\n (\rho_{\lambda\delta}\ast a_\v)|\apprle|\ln(\lambda\delta)|=o(|\ln\v|^{1/3})$.
\end{enumerate}
\end{remark}

The goal of this article is to study the minimizers of
\[
E_\v(u)=\frac{1}{2}\int_\O\left\{|\n u|^2+\frac{1}{2\v^2}\left(a_\v^2-|u|^2\right)^2\right\},\,u\in H^1_g
\]
in the asymptotic $\v\to0$. A standard method (initiated in \cite{LM1}) consists in decoupling $E_\v$ into a sum of two functionals.   The key tool in this method is $U_{\v}$ {\bf the} unique global minimizer of $E_\v$ in $H^1_1$ (see \cite{LM1}). Clearly, $U_\v$ satisfies
\begin{equation}\label{8.EquationforU}
\begin{cases}-\Delta U_\v=\displaystyle\frac{1}{\v^2}U_\v(a^2_\v-U_\v^2)&\text{in }\O\\U_\v=1&\text{on }\p\O\end{cases}.
\end{equation}
From the uniqueness of $U_\v$, by construction of a test function, it is easy to get that $b\leq U_\v\leq1$.

This special solution may be seen as a regularization of $a_\v$. For example, one may easily prove that $U_\v$ is exponentially close to $a_\v$ far away from $\p\o_\v$ (a more complete description of $U_\v$ is done Appendix \ref{S8.DescrSpecSol}). Namely, we have 
\begin{prop}\label{P8.UepsCloseToaeps}
There are $C,\alpha>0$ independent of $\v,R>0$ s.t.
\begin{equation}\label{8.UepsCloseToaeps}
|a_\v-U_\v|\leq C{\rm e}^{-\frac{{\alpha}R}{\v}}\text{ in }V_{R}:=\{x\in\O\,|\,\dist(x,\p\o_\v)\geq R\},
\end{equation}
\begin{equation}\label{7.GradUepsCloseToaeps}
|\n U_\v|\leq \frac{C{\rm e}^{-\frac{ \alpha R}{\v}}}{\v}\text{ in }W_{R}:=\{x\in\O\,|\,\dist(x,\p\o_\v),\dist(x,\p\O)\geq R\}.
\end{equation}
\end{prop}
A similar result was proved in \cite{publi3} (Proposition 2). The above proposition yields by the same arguments.

As in \cite{LM1}, we define
\[
F_\ve(v) = \frac{1}{2} \int_{\O} \left\{U_{\ve}^2 |\n v|^2 + \frac{1}{2 \ve^2} U_{\ve}^4(1-|v|^2)^2\right\}.
\]

Then we have for all $v\in H^1_g$, (see \cite{LM1}) 
\[
E_\v(U_\v v)=E_\v(U_\v)+F_\v(v).
\]
Therefore, $u_\v$ is a minimizer of $E_\v$ if and only if $u_\v=U_\v v_\v$ where $v_\v$ is a minimizer of $F_\v$ in $H^1_g$. Consequently, the study of a minimizer $u_\v=U_\v v_\v$ of $E_\v$ in $H^1_g$ (location of zeros and asymptotics) can be performed by combining the asymptotic of $U_\v$ with one of $v_\v$.

Our main result is the following
\begin{thm}\label{THMMAIN}
Assume that $\lambda,\delta$ satisfy \eqref{8.MainHyp} and that $\lambda\to0$. \vspace{2mm}
\\{\rm Quantization.} There are $\v_0>0$, $c>0$ and $\eta_0>0$ s.t. for $0<\v<\v_0$:
\begin{enumerate}
\item $v_\v$ has exactly $d$ zeros $x_1^\v,...,x_d^\v$, 
\item $B(x_i^\v,c\lambda\delta)\subset\o_\v$, 
\item for $\rho=\rho(\v)\downarrow0$ s.t. $|\ln\rho|/|\ln\v|\to0$, there is $C>0$ independent of $\v$ satisfying
\[
\text{$\displaystyle|v_\v|\geq1-C\sqrt{\frac{|\ln\rho|}{|\ln\v|}}$ in $\O\setminus\cup_{}\overline{B_{}(x^\v_i,\rho)}$,}
\]
\item for $\v<\v_0$\begin{itemize}\item[$\bullet$] There are two repulsive effects: $|x^\v_i-x^\v_j|\geq\eta_0$ for $i\neq j$ and $\dist(x^\v_i,\p\O)\geq\eta_0$;\item[$\bullet$]$\deg_{\p B(x_i^\v,\delta)}(v_\v)=1$.\end{itemize}
\end{enumerate}
{\rm Location.} \begin{itemize}
\item[$\bullet$]The macroscopic location of the zeros tends to minimize the renormalized energy of Bethuel-Brezis-Hélein $W_g:\{\{x_1,...,x_d\}\subset\O\,|\,x_i\neq x_j\text{ for }i\neq j\}\to\R$ (defined in \cite{BBH}, Chapter I Eq. (47)):
\[
\limsup W_g(x^\v_1,...,x^\v_d)=\min_{\substack{a_1,...,a_d\in\O\\a_i\neq a_j}} W_g(a_1,...,a_d)
\] 
\item[$\bullet$]The microscopic location of the zeros inside $\o_\v$ tends to depend {\bf only} on $\o$ and $b$: 
\begin{itemize}
\item since $x_i^\v\in\o_\v$, we have $x_i^\v=(k_\v\delta,l_\v\delta)+\lambda\delta y_i^\v$ with $k_\v,l_\v\in\Z$ and $y_i^\v\in \w$; 
\item for $\v_n\downarrow0$ s.t. $y_i^{\v_n}\to \hat{\hat{a}}_i$, we have $\hat{\hat{a}}_i\in\o$ which minimizes a renormalized energy $\tilde{W}_1:\o\to\R$ (given in \cite{publi3} Eq. (90)) which depends only on $\o$ and $b\in(0,1)$.
\end{itemize}
\end{itemize}
\end{thm}
\begin{remark}\begin{enumerate}
\item The renormalized energy defined in \cite{BBH}
\[
W_g:\{\{x_1,...,x_d\}\subset\O\,|\,x_i\neq x_j\text{ for }i\neq j\}\to\R
\]
governs the location of the zeros in the situation where $a_\v\equiv1$ (homogenous case): the zeros tend to minimize $W_g$. In \cite{BBH} (Chapter 1), the authors defined a renormalized energy in a more general setting 
\[
W_g^{\rm BBH}:\left\{\{(x_1,d_1),...,(x_N,d_N)\}\,\left|\,\begin{array}{c}x_i\in\O,\,x_i\neq x_j\text{ for }i\neq j\\d_i\in\Z\text{ is s.t. }\sum_{i=1}^N d_i=d\end{array}\right.\right\}\to\R.\]
Here $W_g(x_1,...,x_d)=W_g^{\rm BBH}(\{(x_1,1),...,(x_d,1)\})$, \emph{i.e.}, in this article we will consider only the renormalized energy with the degrees equal $1$ and thus we do not specify the degrees in its notation.
\item From smoothness of $W_g$ (see \cite{BBH} and \cite{CM1}), \emph{Location part} of Theorem \ref{THMMAIN} implies that up to pass to a subsequence, the zeros converge to a minimizer of $W_g$. 
\item This macroscopic location is strongly correlated with the Dirichlet boundary condition $g\in C^\infty(\p\O,\S^1)$. 
\item The result about the macroscopic position of the periodic and diluted pinning term may be sum up as: \emph{the macroscopic position of the zeros tends to be the same than in the homogenous case ($a_\v\equiv1$)}.
\item The microscopic location of the zeros (position inside an inclusion) is independent of the boundary condition. For example, in the situation $\o=B(0,r_0)$, \emph{i.e.}, the inclusions are discs, this location should be the center of the inclusion. This fact is not proved yet.
\item In Assertion 4. of  {\it Quantization part}, $\deg_{\p B(x_i^\v,\delta)}(v_\v)=\deg_{\p B(x_i^\v,\delta)}(v_\v/|v_\v|)$.
\end{enumerate}
\end{remark}
\section{Main results}
We present in this section several extensions of the above result dropping either the dilution of the inclusion ($\lambda\equiv1$ instead of $\lambda\to0$) or the periodic structure. The main results of this section are obtained under the condition: $\lambda\delta$ satisfies \eqref{8.MainHyp}. 

Our sharper results are shared into four theorems:
\begin{enumerate}[$\bullet$]
\item The first theorem (Theorem \ref{T8.MainThm1}) gives informations on the zeros of minimizers $u_\v,v_\v$ (quantization and location).
\item The second theorem (Theorem \ref{T8.MainThm2}) establishes the asymptotic behavior of $v_\v$. 
\item The  third theorem (Theorem \ref{T8.Localisationdanslesinclusions}) establishes, under the additional hypothesis $\lambda\to0$, that the microscopic position of the zeros is independent of the boundary condition $g$.
\item The last theorem (Theorem \ref{T8.MainThm3}) gives an expansion of $F_\v(v_\v)$.
\end{enumerate}

The technics developed in this paper allows to consider either the case $\lambda\to0$ or $\lambda\equiv1$. The results in the diluted case are more precise. One may drop the periodic structure for the pinning term and consider impurities (the connected components of $\o_\v=\{a_\v=b\}$) with different sizes (adding the hypothesis $\lambda\to0$). 

More precisely we may consider the pinning term defined as follow: 
\begin{construction}The general diluted pinning term
\begin{enumerate}[$\bullet$]
\item Fix $\num\in\N^*$, $j\in\{1,...,\num\}$ and $1>\v>0$. We consider $M_j^\v\in\N$ and 
\begin{equation*}
\mathcal{M}_j^\v=\begin{cases}\emptyset&\text{if }M_j^\v=0\\\{1,...,{M_j^\v}\}&\text{if }M_j^\v\in\N^*\end{cases}.
\end{equation*}
\item The sets $\M_j^\v$'s are s.t. (for sufficiently small $\v$) one may fix $y_{i,j}^\v\in\O$ s.t. for $(i,j)\neq(i',j')$, $i\in\M_j^\v,\,i'\in\M_{j'}^\v$ we have
\begin{equation}\label{SepHyp}
|y_{i,j}^\v-y_{i',j'}^\v|\geq \delta^j+\delta^{j'}\text{ and }\dist(y^\v_{i,j},\p\O)\geq\delta^j.
\end{equation}
We denote $\WM^\v_j:=\{y_{i,j}^\v\,|\,i\in\mathcal{M}^\v_j\}$.
 
For sake of simplicity, we assume that there is $\eta>0$ s.t. for small $\v$, we have $M_1^\v\geq d=\deg_{\p\O}(g)$ and 
\begin{equation}\label{Sizehyp}
\min\left\{\min_{i=1,...,d}\dist(y_{i,1}^\v,\p\O),\min_{\substack{i,i'=1,...,d\\i\neq i'}}|y_{i,1}^\v- y_{i',1}^\v|\right\}\geq\eta.
\end{equation}
\item We now define the domain which models the impurities:
\[
\o_\v=\bigcup_{j=1}^\num\bigcup_{i\in\M_j^\v}\left\{y_{i,j}^\v+\delta^j\cdot\o^\lambda\right\},\,\o^\lambda=\lambda\cdot\o.
\]
\end{enumerate}
The pinning term is
\[
\begin{array}{cccc}a_\v:&\R^2&\to&\{b,1\}\\&x&\mapsto&\begin{cases}1&\text{if }x\notin\o_\v\\b&\text{if }x\in\o_\v\end{cases}\end{array}
\]
The values of the pinning term are represented Figure \ref{RepresentationGenTC}.
\end{construction}
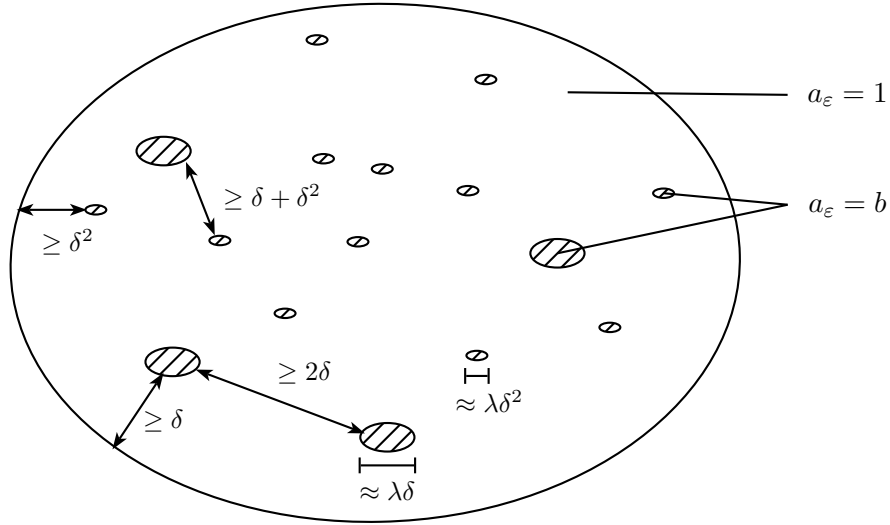
\begin{figure}[h]
\begin{center}
\psset{xunit=.850cm,yunit=.850cm,algebraic=true,dotstyle=o,dotsize=3pt 0,linewidth=0.8pt,arrowsize=3pt 2,arrowinset=0.25}
\begin{pspicture*}(-14.5,-2.2)(-0.225,6.2)
\rput{-178.73}(-8.39,1.97){\psellipse(0,0)(5.72,4.07)}
\rput{0}(-11.7,3.72){\psellipse[fillstyle=hlines](0,0)(0.44,0.24)}
\rput{0}(-11.56,0.42){\psellipse[fillstyle=hlines](0,0)(0.44,0.24)}
\rput{0}(-8.2,-0.76){\psellipse[fillstyle=hlines](0,0)(0.44,0.24)}
\rput{0}(-5.54,2.12){\psellipse[fillstyle=hlines](0,0)(0.44,0.24)}
\rput{0}(-8.66,2.3){\psellipse[fillstyle=hlines](0,0)(0.18,0.09)}
\rput{0}(-6.66,4.84){\psellipse[fillstyle=hlines](0,0)(0.18,0.09)}
\rput{0}(-6.8,0.52){\psellipse[fillstyle=hlines](0,0)(0.18,0.09)}
\rput{0}(-10.82,2.32){\psellipse[fillstyle=hlines](0,0)(0.18,0.09)}
\rput{0}(-12.76,2.8){\psellipse[fillstyle=hlines](0,0)(0.18,0.09)}
\rput{0}(-9.2,3.6){\psellipse[fillstyle=hlines](0,0)(0.18,0.09)}
\rput{0}(-9.3,5.46){\psellipse[fillstyle=hlines](0,0)(0.18,0.09)}
\rput{0}(-6.94,3.1){\psellipse[fillstyle=hlines](0,0)(0.18,0.09)}
\rput{0}(-8.28,3.44){\psellipse[fillstyle=hlines](0,0)(0.18,0.09)}
\rput{0}(-3.88,3.06){\psellipse[fillstyle=hlines](0,0)(0.18,0.09)}
\rput{0}(-4.72,0.96){\psellipse[fillstyle=hlines](0,0)(0.18,0.09)}
\rput{0}(-9.8,1.18){\psellipse[fillstyle=hlines](0,0)(0.18,0.09)}
\psline(-1.94,2.88)(-3.88,3.06)
\psline(-5.54,2.12)(-1.94,2.88)
\psline(-5.38,4.64)(-1.94,4.6)
\rput(-1,2.88){$a_\v=b$}
\rput(-1,4.6){$a_\v=1$}

\psline{|-|}(-7.75,-1.2)(-8.65,-1.2)
\psline{|-|}(-6.6,.2)(-7,.2)

\psline{<->}(-8.55,-.7)(-11.2,0.3)
\rput(-9.5,.25){\small$\geq2\delta$}

\psline{<->}(-10.9,2.4)(-11.35,3.55)
\rput(-10,3){\small$\geq\delta+\delta^2$}

\rput(-8.2,-1.6){\small$\approx\lambda\delta$}
\rput(-6.6,-.2){\small$\approx\lambda\delta^2$}

\psline{<->}(-12.5,-.95)(-11.7,0.25)
\rput(-11.70,-.5){\small$\geq\delta$}
\rput(-1.5,2.8){\psline{<->}(-12.5,0)(-11.4,0)
\rput(-11.70,-.5){\small$\geq\delta^2$}}
\end{pspicture*}
\caption{Representation of the general diluted pinning term with $\num=2$}\label{RepresentationGenTC}
\end{center}
\end{figure}

Our main results are
\begin{thm}\label{T8.MainThm1}
Assume that $\lambda,\delta$ satisfy \eqref{8.MainHyp} and if the pinning term is not periodic (represented Figure \ref{RepresentationGenTC}) then we assume also that $\lambda\to0$. 

There is $\v_0>0$ s.t.:
\begin{enumerate}
\item for $0<\v<\v_0$, $v_\v$ has exactly $d$ zeros $x_1^\v,...,x_d^\v$, 
\item there are $c>0$ and $\eta_0>0$ s.t. for $\v<\v_0$, $B(x_i^\v,c\lambda\delta)\subset\o_\v$ and 
\[
\min_i\left\{\min_{j\neq i}|x^\v_i-x^\v_j|,\dist(x^\v_i,\p\O)\right\}\geq\eta_0.
\]
In particular, if the pinning term is not periodic, then the zeros are trapped by the largest inclusions (those of size $\lambda\delta$).
\item for $\rho=\rho(\v)\downarrow0$ s.t. $|\ln\rho|/|\ln\v|\to0$, we have for $\v<\v_0$, 
\[\text{$\displaystyle|v_\v|\geq1-C\sqrt{\frac{|\ln\rho|}{|\ln\v|}}$ in $\O\setminus\cup_{}\overline{B_{}(x^\v_i,\rho)}$.}
\]
Here $C$  is independent of $\v$.
\item for $\v<\v_0$, $\deg_{\p B(x_i^\v,\delta)}(v_\v)=1$.
\end{enumerate}
\end{thm}
\begin{remark}
Hypothesis \eqref{Sizehyp} is used to simplify the statements. Without this hypothesis, some of the results are subject to technical considerations on $\delta,\lambda,b...$ For example if we consider the pinning term $a_\v$ defined in $\O=B(0,2)$ by
\[
\begin{array}{cccl}
a_\v:&B(0,2)&\to&\{b,1\}\\&x&\mapsto&\begin{cases}b&\text{if } x\in B(0,\lambda\delta)\cup B(1,\lambda\delta^2)\\1&\text{otherwise}\end{cases}
\end{array},
\]
and $g\in C^\infty(\p\O,\S^1)$ s.t. $\deg_{\p\O}(g)=2$, then Hypothesis \eqref{Sizehyp} is not satisfied. In this situation, we may prove that, for sufficiently small $\v$, $v_\v$ has exactly two zeros and if $2(1-2b^2)|\ln\lambda|+(1-3b^2)|\ln\delta|\to+\infty$ (resp. $-\infty$), then the zeroes are in $B(0,\lambda\delta)$ (resp. there is one zero inside $B(0,\lambda\delta)$ and one zero inside $B(1,\lambda\delta^2)$).
 
\end{remark}
\begin{thm}\label{T8.MainThm2}
Assume that $\lambda,\delta$ satisfy \eqref{8.MainHyp} and if the pinning term is not periodic (represented Figure \ref{RepresentationGenTC}) then we assume also that $\lambda\to0$.

Let $\v_n\downarrow0$, up to a subsequence, we have the existence of $a_1,...,a_d\in\O$, $d$ distinct points s.t. $x_i^{\v_n}\to a_i$ and
\[
|v_{\v_n}|\to1\text{ and }v_{\v_n}\weak v_*\text{ in }H^1_{\rm loc}(\overline{\O}\setminus\{a_1,...,a_d\},\S^1)
\]
where $v_*$ solves
\begin{equation}\nonumber
\begin{cases}
-\Div(\mathcal{A}\n v_*)=(\mathcal{A}\n v_*\cdot\n v_*)v_*&\text{in }\O\setminus\{a_1,...,a_d\}
\\
v_*=g&\text{on }\p\O
\end{cases}.
\end{equation}
Here $\mathcal{A}$ is the homogenized matrix of $a^2\left(\displaystyle\frac{\cdot}{\delta}\right){\rm Id}_{\R^2}$ if $\lambda\equiv1$ and $\mathcal{A}={\rm Id}_{\R^2}$ if $\lambda\to0$.

In addition, for each $M>0$, $v'_{\v,i}(\cdot)=v_\v\left(x_i^\v+\displaystyle\frac{\v }{b}\cdot\right)$ converges, up to a subsequence, in $C^1(B(0,M))$ to $f(|x|)\displaystyle\frac{x}{|x|}{\rm e}^{\imath\theta_i}$ where $f:\R^+\to\R^+$ is the universal function defined in \cite{M2} and $\theta_i\in\R$.
\end{thm}
\begin{thm}\label{T8.Localisationdanslesinclusions}
Assume, in addition to the hypotheses of Theorem \ref{T8.MainThm2}, that $\lambda\to0$. 

Let $[x]=[(x_1,x_2)]=([x_1],[x_2])\in\Z^2$ be the vectorial integer part of the point $x\in\R^2$.

For $x_i^{\v}$ a zero of $v_\v$, let 
\[
y_i^{\v}=\displaystyle \frac{\frac{x_i^{\v}}{\delta}-[\frac{x_i^{\v}}{\delta}]}{\lambda}\in \o.
\]
Then, as $\v\to0$, up to pass to a subsequence, we have $y_i^{\v}\to \hat{\hat{a}}_i\in\o$. Here, $\hat{\hat{a}}_i$ minimizes a renormalized energy $\tilde{W}_1:\o\to\R$ (given in \cite{publi3} Eq. (90)) which depends only on $\o$ and $b$. In particular, $\hat{\hat{a}}_i$ is independent of the boundary condition $g$. 
\end{thm}

\begin{thm}\label{T8.MainThm3}
Assume that $\lambda,\delta$ satisfy \eqref{8.MainHyp} and if the pinning term is not periodic (represented Figure \ref{RepresentationGenTC}) then we assume also that $\lambda\to0$. 

We have the following expansion
\[
F_\v(v_\v)= J_{\v,\v}+ db^2(\pi\ln b+\gamma)+o_\v(1)
\]
where $J_{\v,\v}$ is defined in \eqref{8.StatementSecondAuxProblemDir} and $\gamma>0$ is the universal constant defined in \cite{BBH} Lemma IX.1.
\end{thm}
This article is divided in two parts:
\begin{enumerate}[$\bullet$]
\item In the first one (Section \ref{S8.FirstSection}) we consider two auxiliary minimization problems for weighted Dirichlet functionals associated to $\S^1$-valued maps. 
\item The second part (Section \ref{S8.GLStudy}) is devoted to the proofs of Theorems \ref{THMMAIN}, \ref{T8.MainThm1}, \ref{T8.MainThm2}, \ref{T8.Localisationdanslesinclusions}, \ref{T8.MainThm3}. The main tool is an $\eta$-ellipticity result (Lemma \ref{L8.BadDiscsLemma}).  This lemma reduces (under the assumption that  $\lambda,\delta$ satisfy \eqref{8.MainHyp}) the study of $F_\v$ to the one of the auxiliary problems considered  Section \ref{S8.FirstSection}.
\end{enumerate}
\section{Shrinking holes for weighted Dirichlet functionals}\label{S8.FirstSection}
This section is devoted to the study of two minimization problems and it is divided in three subsections.  

The first and the second subsections are related with minimizations of  weighted Dirichlet functionals among $\S^1$-valued maps. In both subsections, the considered weights are the more general one: $\alpha\in L^\infty(\R^2,[b^2,1])$. The third subsection deals the weight $\alpha=U_\v^2$ in the situation where $U_\v$ is the minimizer of $E_\v$ in $H^1_1$ with $a_\v$ represented Figure \ref{Intro.FigureTermeChevillage} (the periodic case with or without dilution) or Figure \ref{RepresentationGenTC} (the general diluted case). 

\begin{notation}
In Section \ref{S8.FirstSection} we fix :
\begin{enumerate}[$\bullet$]
\item a smooth simply connected domain $\O\subset\R^2$; 
\item a boundary condition $g\in C^\infty(\p\O,\S^1)$ s.t. $d:=\deg_{\p\O}(g)>0$;
\item a smooth and bounded open subset $\O'\prec\R^2$ s.t. $\overline{\O}\subset\O'$;
\item an extension of $g$ which is in $ C^\infty(\overline{\O'}\setminus{\O},\S^1)$ (this extension is also denoted by $g$).
\end{enumerate}
We will also consider (uniformly bounded) families of points/degres $\{(x_1,d_1),...,(x_N,d_N)\}=\{{\bf x},{\bf d}\}$ s.t. 
\begin{enumerate}[$\bullet$]
\item $x_i\in\O$, $x_i\neq x_{i'}$ for $i\neq i'$ ;
\item $d_i$ are s.t. $d_i\in\N^*$ and $\sum_i d_i=d$ (thus $N\leq d$). 
\end{enumerate}
According to the considered problems, for $0<\rho\leq8^{-1}\min_{i\neq i'}|x_i-x_{i'}|$ we will use the following perforated domains
\begin{enumerate}[$\bullet$]
\item $\O_\rho:=\O_\rho({\bf x})=\O\setminus\cup_i\overline{B(x_i,\rho)}$ ;
\item $\O'_\rho:=\O'_\rho({\bf x})=\O'\setminus\cup_i\overline{B(x_i,\rho)}$.
\end{enumerate}
\end{notation}
\subsection{Existence results}\label{Sec.ExistenceResults}
In this subsection we prove the existence of solutions of two minimization problems whose studies will be the purpose of the rest of Section \ref{S8.FirstSection} (Subsections \ref{S8.DegCondDirCond} $\&$ \ref{Sec.SecondProbl}).
\subsubsection{Existence of minimal maps defined in a perforated domain}

Let ${\bf x}=(x_1,...,x_N)$ be $1\leq N\leq d$ distinct points of $\O$ and let ${\bf d}=(d_1,...,d_N)\in(\N^*)^N$ s.t. $\sum_i d_i=d$.

For $0<\rho<8^{-1}\min_{i\neq j}|x_i-x_j|$, we denote $\O_{\rho}=\O\setminus\cup \overline{B(x_i,{\rho})}$. 

We define
\[
\I_{\rho}({\bf x},{\bf d})=\I_{\rho}:=\left\{w\in H^1(\O_{\rho},\S^1)\,|\,w=g\text{ on }\p\O\text{ and }\deg_{\p B(x_i,{\rho})}(w)=d_i\right\}
\]
and for $0<\rho<8^{-1}\min\left\{\min_{i\neq j}|x_i-x_j|,\min_i \dist(x_i,\p\O)\right\}$
\[
\J_{\rho}({\bf x},{\bf d})=\J_{\rho}:=\left\{w\in H^1(\O_{\rho},\S^1)\,|\,w=g\text{ on }\p\O\text{ and }w(x_i+\rho{\rm e}^{\imath\theta})={\rm e}^{\imath(d_i\theta+\theta_i)},\,\theta_i\in\R\right\}.
\]
From the compatibility condition $\deg_{\p\O}(g)=d=\sum d_i$, we have $\I_{\rho}({\bf x},{\bf d}),\J_{\rho}({\bf x},{\bf d})\neq\emptyset$ and it is clear that $\J_{\rho}({\bf x},{\bf d})\subset\I_{\rho}({\bf x},{\bf d})$.

In Subsection \ref{S8.DegCondDirCond}, we compare the minimal energies corresponding to a weighted Dirichlet functional in the above sets. Here, we just state existence results.
\begin{prop}\label{P8.ExistenceOfminInIJ}
Let $\alpha\in L^\infty(\O)$ be s.t. $b^2\leq\alpha\leq1$. Consider the minimization problems
\[
\widehat{\I}_{\rho,\alpha}({\bf x},{\bf d})=\inf_{w\in \I_{\rho}}{\frac{1}{2}\int_{\O_{\rho}}\alpha|\n w|^2}
\]
and 
\[
\widehat{\J}_{\rho,\alpha}({\bf x},{\bf d})=\inf_{w\in \J_{\rho}}{\frac{1}{2}\int_{\O_{\rho}}\alpha|\n w|^2}.
\]
In both minimization problems the \emph{infima} are attained.

Moreover, if $\alpha\in W^{1,\infty}(\O)$, then, denoting $w_{\rho,\alpha}^{\rm deg}$ (resp. $w_{\rho,\alpha}^{\rm Dir}$) a global minimizer of $\displaystyle\frac{1}{2}\int_{\O_\rho}\alpha|\n\cdot|^2$ in $\I_{\rho}({\bf x},{\bf d})$ (resp.  in $\J_{\rho}({\bf x},{\bf d})$) we have $w_{\rho,\alpha}^{\rm deg}\in H^2(\O_\rho,\S^1)$ (resp. $w_{\rho,\alpha}^{\rm Dir}\in H^2(\O_\rho,\S^1)$) and
\begin{equation}\label{8.EquationForCriticalPointDir+Deg}
\begin{cases}
-\Div(\alpha\n w^{\rm deg}_{\rho,\alpha})=\alpha|\n w^{\rm deg}_{\rho,\alpha}|^2w^{\rm deg}_{\rho,\alpha}\text{ in }\O_\rho
\\
w^{\rm deg}_{\rho,\alpha}\in\I_{\rho}\,{\rm and}\, w^{\rm deg}_{\rho,\alpha}\times\p_\nu w^{\rm deg}_{\rho,\alpha}=0\text{ on }\p B(x_i,\rho),\,i=1,...,N
\end{cases},
\end{equation}
 \begin{equation}\label{8.EquationForCriticalPointDir+AlmostDir}
\begin{cases}
-\Div(\alpha\n w^{\rm Dir}_{\rho,\alpha})=\alpha|\n w^{\rm Dir}_{\rho,\alpha}|^2w^{\rm Dir}_{\rho,\alpha}\text{ in }\O_\rho
\\
w^{\rm Dir}_{\rho,\alpha}\in\J_{\rho}\,{\rm and}\,\int_{\p B(x_i,\rho)}\alpha w^{\rm Dir}_{\rho,\alpha}\times\p_\nu w^{\rm Dir}_{\rho,\alpha}=0,\,i=1,...,N
\end{cases}.
\end{equation}
\end{prop}
The proof of this standard result is postponed to Appendix \ref{S8.ExistenceMiniMizingMap}.

In the special case $\alpha=U_\v^2$, we denote
\[
\widehat{\I}_{\rho,\v}({\bf x},{\bf d})=\inf_{w\in \I_{\rho}}{\frac{1}{2}\int_{\O_{\rho}}U_\v^2|\n w|^2}\text{ and }\widehat{\J}_{\rho,\v}({\bf x},{\bf d})=\inf_{w\in \J_{\rho}}{\frac{1}{2}\int_{\O_{\rho}}U_\v^2|\n w|^2}.
\]
\subsubsection{Existence of an optimal perforated domain}

For $\alpha\in L^\infty(\R^2,[b^2,1])$ we define
\begin{equation}\label{8.StatementSecondAuxProblemBis}
I_{\rho,\alpha}:=\inf_{\substack{x_1,...,x_N\in{\O}\\|x_i-x_j|\geq8\rho\\d_1,...,d_N>0,\sum d_i=d}}\inf_{\substack{w\in H_g^1(\O'_{\rho},\S^1)\\\deg_{\p B(x_i,{\rho})}(w)=d_i}}{\frac{1}{2}\int_{\O'_{\rho}}{\alpha|\n w|^2}}
\end{equation}
and
\begin{equation}\label{8.StatementSecondAuxProblemDirBis}
J_{\rho,\alpha}:=\inf_{\substack{x_1,...,x_d\in{\O}\\|x_i-x_j|\geq8\rho\\\dist(x_i,\p\O)\geq8\rho}}\inf_{\substack{w\in H_g^1(\O_{\rho},\S^1)\\w(x_i+\rho{\rm e}^{\imath\theta})={\rm e}^{\imath(\theta+\theta_i)},\theta_i\in\R}}{\frac{1}{2}\int_{\O_{\rho}}{\alpha|\n w|^2}}.
\end{equation}
Here $\O'_{\rho}=\O'\setminus\cup \overline{B(x_i,{\rho})}$.

In the special case $\alpha=U_\v^2$, we denote 
 \begin{equation}\label{8.StatementSecondAuxProblem}
I_{\rho,\v}:=\inf_{\substack{x_1,...,x_N\in\O\\|x_i-x_j|\geq8\rho\\d_1,...,d_N>0,\sum d_i=d}}\inf_{\substack{w\in H_g^1(\O'_{\rho},\S^1)\\\deg_{\p B(x_i,{\rho})}(w)=d_i}}{\frac{1}{2}\int_{\O'_{\rho}}{U_\v^2|\n w|^2}}
\end{equation}
and
\begin{equation}\label{8.StatementSecondAuxProblemDir}
J_{\rho,\v}:=\inf_{\substack{x_1,...,x_d\in\O\\|x_i-x_j|\geq8\rho\\\dist(x_i,\p\O)\geq8\rho}}\inf_{\substack{w\in H_g^1(\O_{\rho},\S^1)\\w(x_i+\rho{\rm e}^{\imath\theta})={\rm e}^{\imath(\theta+\theta_i)},\theta_i\in\R}}{\frac{1}{2}\int_{\O_{\rho}}{U_\v^2|\n w|^2}}.
\end{equation}
We have the following result
\begin{prop}\label{P8.MainAuxPb:DirVsDeg}
For $\alpha\in L^\infty(\R^2,[b^2,1])$, there are ${\bf x}_{\rho,\alpha}^{\rm deg},{\bf x}_{\rho,\alpha}^{\rm Dir}\in\O^d$ and ${\bf d}_{\rho,\alpha}\in(\N^*)^N$ (with ${\bf d}_{\rho,\alpha}=(d_1,...,d_N),\sum d_i=d$) s.t. $\{{\bf x}_{\rho,\alpha}^{\rm deg},{\bf d}_{\rho,\alpha}\}$ minimizes $I_{\rho,\alpha}$ and ${\bf x}_{\rho,\alpha}^{\rm Dir}$ minimizes $J_{\rho,\alpha}$.

\end{prop}
The proof of this result is in Appendix \ref{S8.ProofThirdAuxPb}.

 \subsection[Dirichlet Vs Degree Conditions in a fixed perforated domain]{Dirichlet Vs Degree Conditions in a fixed perforated domain}\label{S8.DegCondDirCond}
Let $\eta_{\rm stop}>0$ be s.t. $\eta_{\rm stop}<10^{-5}\cdot9^{-d^2}{\rm diam}(\O)$ and let $N\in\{1,...,d\}$.

Consider  $x_1,...,x_N\in\O$, $N$ distinct points of $\O$ satisfying the condition $\eta_{\rm stop}<10^{-3}\cdot9^{-d^2}\min\dist(x_i,\p\O)$, and let $\rho>0$ be s.t. $\min\left\{\eta_{\rm stop},\min_{i\neq j}|x_i-x_j|\right\}\geq8\rho$. Roughly speaking $\eta_{\rm stop}$ controls the distance between the points and $\p\O$.

The main result of this section is
\begin{prop}\label{P8.AuxResult1}
There is $C_0>0$ depending only on $g,\O,\eta_{\rm stop}$ and $b$ s.t.  for $\alpha\in L^{\infty}(\O,[b^2,1])$ we have
\[
\widehat{\I}_{\rho,\alpha}({\bf x},{\bf d})\leq\widehat{\J}_{\rho,\alpha}({\bf x},{\bf d})\leq\widehat{\I}_{\rho,\alpha}({\bf x},{\bf d})+C_0.
\]
Here, $\widehat{\I}_{\rho,\alpha}$ and $\widehat{\J}_{\rho,\alpha}$ are defined Proposition \ref{P8.ExistenceOfminInIJ}.
\end{prop}

The rigorous proof of Proposition \ref{P8.AuxResult1} is presented in Appendix \ref{S8.ProofOFfirstAuxiliaryProblem}. Here, we simply present the main lines of the proof.

Two situations are possible:
\begin{enumerate}
\item $N=1$ or the points $x_1,...,x_N$ are well separated: $\frac{1}{4}\min_{i\neq j}|x_i-x_j|>\eta_{\rm stop}$,
\item The points $x_1,...,x_N$ are not well separated: $\frac{1}{4}\min_{i\neq j}|x_i-x_j|\leq\eta_{\rm stop}$.
\end{enumerate}
If the points are well separated (or $N=1$), Proposition \ref{P8.AuxResult1} can be easily proved: it is a direct consequence of Proposition \ref{P7.MyrtoRingDegDir} and Lemma \ref{L8.UpperBoundS1ValuedMapInGoodCondition} in Appendix \ref{S8.ProofOFfirstAuxiliaryProblem}. These results, whose statements and proofs are postponed in Appendix \ref{S8.ProofOFfirstAuxiliaryProblem}, give essentially the existence of test functions into two kinds of domains.

The domains are
\begin{itemize}
\item[$\bullet$] the thin domain $\O_{10^{-1}\eta_{\rm stop}}({\bf x})=\O\setminus \cup \overline{B(x_i,10^{-1}\eta_{\rm stop})}$ obtained by perforating $\O$ by "large", "well separated" and "far from $\p\O$" discs,
\item[$\bullet$] the thick annulars $B(x_i,10^{-1}\eta_{\rm stop})\setminus\overline{B(x_i,\rho)}$.
\end{itemize}
The proof is made in three steps: 
\begin{enumerate}[Step 1:]
\item Using  Lemma \ref{L8.UpperBoundS1ValuedMapInGoodCondition}, we obtain a constant $C_1$ (depending only on $g,\O,\eta_{\rm stop}$) s.t. 
\[
\widehat{\J}_{10^{-1}\eta_{\rm stop},\alpha}({\bf x},{\bf d})\leq C_1.
\]
\item With the help of Proposition \ref{P7.MyrtoRingDegDir}, we obtain the existence of a constant $C_b$ (depending only on $b$) s.t. for $\tilde{d}\in\N$, denoting $A_\rho^i=B(x_i, 10^{-1}\eta_{\rm stop})\setminus{\overline{B(x_i,\rho)}}$, we have
\[
\inf_{\substack{w\in H^1(A_\rho^i,\S^1)\\w(x_1+10^{-1}\eta_{\rm stop}{\rm e}^{\imath\theta})={\rm Cst}_1{\rm e}^{\imath\tilde{d}\theta}\\w(x_1+\rho{\rm e}^{\imath\theta})={\rm Cst}_2{\rm e}^{\imath\tilde{d}\theta}}}\frac{1}{2}\int_{A_\rho^i}\alpha|\n w|^2\leq\inf_{\substack{w\in H^1(A_\rho^i,\S^1)\\\deg_{\p B(x_i,\rho)}=\tilde{d}}}\frac{1}{2}\int_{A_\rho^i}\alpha|\n w|^2+C_b\tilde{d}^2.
\]
\item By extending a minimizer of $\widehat{\J}_{10^{-1}\eta_{\rm stop},\alpha}({\bf x},{\bf d})$ by the ones of $\displaystyle\frac{1}{2}\int_{A_\rho^i}\alpha|\n \cdot|^2$ with Dirichlet conditions, we can construct a map which proves the result taking $C_0=C_1+d^3 C_b$.
\end{enumerate}

\subsection[Optimal perforated domains for the degree conditions]{Optimal perforated domains for the degree conditions} \label{Sec.SecondProbl}
Recall that we fixed $\O'\supset\O$ a smooth bounded domain s.t. $\dist(\p\O',\O)>0$ and a smooth $\S^1$-valued extension of $g$ to $\O'\setminus\overline{\O}$ (still denoted by $g$). 

 In this section, we study the minimization problem
\begin{equation}
I_{\rho,\v}:=\inf_{\substack{x_1,...,x_N\in{\O}\\|x_i-x_j|\geq8\rho\\d_1,...,d_N>0,\sum d_i=d}}\inf_{\substack{w\in H_g^1(\O'_{\rho},\S^1)\\\deg_{\p B(x_i,{\rho})}(w)=d_i}}{\frac{1}{2}\int_{\O'_{\rho}}{U_\v^2|\n w|^2}}
\end{equation}
where 
\[
\O'_{\rho}=\O'\setminus\cup \overline{B(x_i,{\rho})}
\]
and
\[
H_g^1(\O'_{\rho},\S^1)=\left\{w\in H^1(\O'_{\rho},\S^1)\,|\,w=g\text{ in }\O'\setminus\overline{\O\cup B(x_i,{\rho})}\right\};
\]
here, we extended $U_\v$ with the value $1$ outside $\O$. We recall that we denoted by $U_{\v}$ {\bf the} unique global minimizer of $E_\v$ in $H^1_1$.
\vspace{2mm}

In this subsection we assume that Hypothesis  \eqref{8.MainHyp} holds ($|\ln(\lambda\delta)|^3/|\ln \v|\to0$). This is not optimal for the statements but it makes the proofs simpler (this hypothesis may be relaxed, but it appears as a crucial and technical hypothesis for the methods developed Section \ref{S8.GLStudy}).

\vspace{2mm}

A first purpose of this section is the study of the behavior of $I_{\rho,\v}$ when $\rho=\rho(\v)\to0$ as $\v\to0$. In view of the application we have in mind we suppose that $\lambda\delta^{\num+1}\gg\rho(\v)\geq\v$ but this is not crucial for our arguments (here $\num=1$ if $U_\v$ is associated associated with the periodic pinning term) .

A second objective of our study is to exhibit the behavior of almost minimal configurations $\{(x^n_1,...,x^n_N),(d^n_1,...,d^n_N)\}$.



For fixed $\rho,\v$, the existence of a minimal configuration of points ${\bf x }_{\rho,\v}$ is the purpose of Proposition \ref{P8.MainAuxPb:DirVsDeg}. In this section we consider only almost minimal configurations. 

\begin{notation}\label{NOTATIONQASIMINDEFINTION}
For $\v_n\downarrow0$, we say that $\{(x^n_1,...,x^n_N),(d^n_1,...,d^n_N)\}$ is an almost minimal configuration for $\rho=\rho(\v_n)\downarrow0$ when $x^n_1,...,x^n_N\in\O$, $|x^n_i-x^n_j|\geq8\rho$, $d^n_1,...,d^n_N>0,\sum d^n_i=d$ and there is $C>0$ (independent of $n$) s.t.
\[
\inf_{\substack{w\in H_g^1(\O'_{\rho},\S^1)\\\deg_{\p B(x^n_i,{\rho})}(w)=d^n_i}}{\frac{1}{2}\int_{\O'_{\rho}}{U_{\v_n}^2|\n w|^2}}-I_{\rho,\v_n}\leq C.
\]
\end{notation}
Roughly speaking, we establish in this section two repelling effects for the points: point/point and point/$\p\O$ ; and an attractive effect for the points by the inclusions $\o_\v$.  

\subsubsection{The case of the periodic pinning term}

The main result of this section establishes that when $\v_n,\rho\downarrow0$, an almost  minimal configuration $\left\{(x^n_1,...,x^n_N),(d^n_1,...,d^n_N)\right\}$ is s.t. (for sufficiently large $n$)
\begin{enumerate}[$\bullet$]
\item the points $x^n_i$'s cannot be mutually close, 
\item the degrees $d_i^n$'s are necessarily all equal to $1$,
\item the points $x^n_i$'s cannot approach $\p\O$,
\item there is $c>0$ s.t. $\displaystyle B(x_i^n,c\lambda\delta)\subset\o_\v$ for all $i$.
\end{enumerate}
These facts are expressed in the following proposition (whose proof is postponed to Appendix \ref{S8.ProofOfSecondAuxPb}).
\begin{prop}\label{P8.ToMinimizeSecondPbThePointAreFarFromBoundAndHaveDegree1}[The case of a periodic pinning term]

Assume that $\lambda,\delta$ satisfy \eqref{8.MainHyp} and let $a_\v$ be the periodic the pinning term (represented Figure \ref{Intro.FigureTermeChevillage}).

Let $\v_n\downarrow0$, $\rho=\rho(\v_n)\downarrow0$, $x^n_1,...,x^n_N\in{\O}$ be s.t. $|x^n_i-x^n_j|\geq8\rho$, $\rho\geq\v_n$ and let $d^n_1,...,d^n_N\in\N^*$ be s.t. $\sum d^n_i=d$. 
\begin{enumerate}
\item Assume that there is $i_0\in\{1,...,N\}$ s.t. $d^n_{i_0}\neq1$ or that there are $i_0\neq j_0$ s.t. $|x_{i_0}^n-x_{j_0}^n|\to0$. Then
\[
\inf_{\substack{w\in H_g^1(\O'_{\rho},\S^1)\\\deg_{\p B(x^n_i,{\rho})}(w)=d^n_i}}\left\{\frac{1}{2}\int_{\O'_{\rho}}{U_{\v_n}^2|\n w|^2}-I_{\rho,\v_n}\right\}\to\infty.
\]
\item Assume that there is $i_0\in\{1,...,N\}$ s.t. $\dist(x_{i_0}^n,\p\O)\to0$. Then
\[
\inf_{\substack{w\in H_g^1(\O'_{\rho},\S^1)\\\deg_{\p B(x^n_i,{\rho})}(w)=d^n_i}}\left\{\frac{1}{2}\int_{\O'_{\rho}}{U_{\v_n}^2|\n w|^2}-I_{\rho,\v_n}\right\}\to\infty.
\]
\item Assume that  $\dfrac{\rho}{\lambda\delta}\to0$  and that there is $i_0\in\{1,...,N\}$ s.t. $x_{i_0}^n\notin\o_\v$ or s.t. $x_{i_0}^n\in\o_\v$ and $\displaystyle\frac{\dist(x_{i_0}^n,\p\o_\v)}{\lambda\delta}\to0$. Then
\[
\inf_{\substack{w\in H_g^1(\O'_{\rho},\S^1)\\\deg_{\p B(x^n_i,{\rho})}(w)=d^n_i}}\left\{\frac{1}{2}\int_{\O'_{\rho}}{U_{\v_n}^2|\n w|^2}-I_{\rho,\v_n}\right\}\to\infty.
\]
\end{enumerate}
\end{prop}
A straightforward consequence of Proposition \ref{P8.ToMinimizeSecondPbThePointAreFarFromBoundAndHaveDegree1} is the following
\begin{cor}\label{C8.CorolSecondAuxPb}
\begin{enumerate}
\item Consider an almost minimal configuration $\{{\bf x}_{\rho,\v},{\bf d}_{\rho,\v}\}\in{\O}^N\times {\N^*}^N$, 
\emph{i.e.}, assume that there is $w_{\rho,\v}\in H^1_g(\O'\setminus\cup\overline{B(x_i^{\rho,\v},\rho)},\S^1)$ verifying 
\[
\deg_{\p B(x_i^{\rho,\v},\rho)}(w)=d^{\rho,\v}_i\text{ and }\frac{1}{2}\int_{\O'\setminus\cup\overline{B(x_i^{\rho,\v},\rho)}}U_\v^2|\n w |^2\leq I_{\rho,\v}+C.
\]
(Here, $C$ is independent of $\v$.)

Then, there is some $\eta_0$ independent of $\v$ s.t., for small $\v$, we have
\[
\text{$|x_i^{\rho,\v}-x_j^{\rho,\v}|,\dist(x_i^{\rho,\v},\p\O)\geq\eta_0$ and $d_i=1$ for all $i\neq j$, $i,j\in\{1,...,N\}$.}
\]
In particular, we have $N=d$.
\item If, in addition,  $\rho=\rho(\v)$ is s.t. $\rho\geq\v$ and $\dfrac{\rho}{\lambda\delta}\to0$, then there is $c>0$ (independent of $\v$) s.t., for small $\v$, we have $B(x_i^{\rho,\v},c\lambda\delta)\subset\o_\v$.
\end{enumerate}
\end{cor}
\begin{proof}[Proof of Corollary \ref{C8.CorolSecondAuxPb}]
We prove the first part. Let $C>0$. We argue by contradiction and we assume that for all $n\in\N^*$ there are $0<\v_n\leq\rho=\rho(\v_n)\leq1/n$, ${\bf x }_n={\bf x }_{\rho,\v_n}$, $(d_1,...,d_N)$ and $w_n=w_{\rho,\v_n}$ satisfying the hypotheses of Corollary \ref{C8.CorolSecondAuxPb} and s.t. 
\[
\min\left\{|x_i^n-x_j^n|,\dist(x_i^n,\p\O)\right\}\to0\text{ or s.t. there is $i\in\{1,...,N\}$ for which we have $d_i\neq1$}.
\]
By construction we have that $\{{\bf x}_{\rho,\v_n},{\bf d}\}$ is an almost minimal configuration for $I_{\rho,\v_n}$ with $\rho=\rho(\v_n)\geq\v_n$. Clearly from Proposition \ref{P8.ToMinimizeSecondPbThePointAreFarFromBoundAndHaveDegree1} we find a contradiction. 

The proof of the second part is similar.
\end{proof}
We end this subsection by the following direct consequence of Corollary \ref{C8.CorolSecondAuxPb}
\begin{cor}\label{C8.AnAlmostMinConfigIsAnAlmostMinConf}
For sufficiently small $\v,\rho$, an almost minimal configuration $(x_1,...,x_d)$ for $J_{\rho,\v}$ is an almost minimal configuration for $I_{\rho,\v}$.

Moreover, there is $C_0>0$ s.t. $J_{\rho,\v}\leq I_{\rho,\v}+C_0$, $C_0$ is independent of small $\v$, $\rho$.
\end{cor}
\begin{proof}
Let $C\geq0$ and let $(x_1,...,x_d),(x'_1,...,x'_d)\in\O^d$ be s.t.
\[
\hat{\J}_{\rho,\v}(x_1,...,x_d)\leq J_{\rho,\v}+C
\]
and 
\[
\hat{\I}_{\rho,\v}(x'_1,...,x'_d)\leq I_{\rho,\v}+C.
\]
From Corollary  \ref{C8.CorolSecondAuxPb}, there is $\eta_0=\eta_0(C)>0$ s.t. for $\v\leq\rho\leq\eta_0$, $\min_{i}\dist(x_i',\p\O)\geq\eta_0$. Using Proposition \ref{P8.AuxResult1} we have the existence of $C_0$ s.t.
\begin{eqnarray*}
\hat{\I}_{\rho,\v}(x_1,...,x_d)\leq\hat{\J}_{\rho,\v}(x_1,...,x_d)&\leq& J_{\rho,\v}+C\leq\hat{\J}_{\rho,\v}(x'_1,...,x'_d)+C
\\&\leq&\hat{\I}_{\rho,\v}(x'_1,...,x'_d)+C+C_0\\&\leq& I_{\rho,\v}+2C+C_0.
\end{eqnarray*}
\end{proof}
\subsubsection[Sharper result in the periodic case with dilution]{A more precise result for the case of the periodic pinning term with dilution}
In this section we focus on the periodic pinning term (represented Figure \ref{Intro.FigureTermeChevillage}) with dilution: $\lambda\to0$.  

\begin{notation}\label{NOTATIONQUASIIN}We define two kinds of configuration of distinct points of $\O$:
\begin{itemize}
\item[$\bullet$] We say that for $\v_n\downarrow0$ and $\rho=\rho(\v_n)\to0$, $d$ distinct points of $\O$, ${\bf x}_n=(x^n_1,...,x^n_d)$  form a quasi-minimizer of $J_{\rho,\v_n}$ when $\J_{\rho,\v_n}({\bf x}_n)-J_{\rho,\v_n}\to0$.
\item[$\bullet$] We say that for $\v_n\downarrow0$ and $\rho=\rho(\v_n)\to0$, $d$ distinct points of $\O$, ${\bf x}_n=(x^n_1,...,x^n_d)$  form a quasi-minimizer of $W_g$, the {\it renormalized energy} of Bethuel-Brezis-Hélein (see \cite{BBH}) when $W_g({\bf x}_n)\to\min W_g$.
\end{itemize}
\end{notation}
\begin{prop}\label{P.RenormalizedBBHEnergy}[Asymptotic location of optimal perforations]

Assume that $\lambda,\delta$ satisfy \eqref{8.MainHyp} and that $\lambda\to0$.

Let $\v_n\downarrow0$, $\rho=\rho(\v_n)\to0$, $\rho\geq\v_n$ and ${\bf x}_n=(x^n_1,...,x^n_d)$ be $d$ distinct points of $\O$.

If the points ${\bf x}_n$ form a quasi-minimizer of $J_{\rho,\v_n}$, then ${\bf x}_n=(x^n_1,...,x^n_d)$  form a quasi-minimizer of $W_g$.
\end{prop}
This proposition is proved Appendix \ref{S.ProofOfRenEnergyBBH}.
\subsubsection{The case of a general pinning term with variable sizes of inclusions}

We assume that $a_\v$ is the general pinning term represented Figure \ref{RepresentationGenTC} with the hypothesis on the dilution: $\lambda\to0$.

\begin{prop}\label{P8.CorolSecondAuxPb}[The case of a non-periodic pinning term]

Assume that $\lambda,\delta$ satisfy \eqref{8.MainHyp} and $\lambda\to0$.

Let $\rho=\rho(\v)$ s.t. $\rho\geq\v$ and $\dfrac{\rho}{\lambda\delta^{3/2}}\to0$. If $\{{\bf x}_{\rho,\v},{\bf d}_{\rho,\v}\}$ is an almost minimal configuration for $I_{\rho,\v}$, then $N=d$ (thus $d_i=1$ for all $i$) and there are $c,\eta_0>0$ (independent of $\v$) s.t. for sufficiently small $\v$:
\begin{enumerate}
\item $|x_i^{\rho,\v}-x_j^{\rho,\v}|,\dist(x_i^{\rho,\v},\p\O)\geq\eta_0$ for all $i\neq j$, $i,j\in\{1,...,N\}$.
\item $B(x_i^{\rho,\v},c\lambda\delta)\subset\o_\v$ (the centers of the holes are included in the largest inclusions).
\end{enumerate}
Moreover, there is $C_0>0$ s.t. $J_{\rho,\v}\leq I_{\rho,\v}+C_0$, $C_0$ is independent of small $\v$, $\rho$. And thus an almost minimal configuration ${\bf x}_{\rho,\v}$ for $J_{\rho,\v}$ is an almost minimal configuration for $I_{\rho,\v}$
\end{prop}
This proposition is proved Appendix \ref{S.ProofOfRenEnergyBBH} (Subsection \ref{S.PorrofGenDiltoihefkqsdjf}).


\section{The pinned Ginzburg-Landau functional}\label{S8.GLStudy}
In this section, we turn to  the main purpose of this article: the study of minimizers of $E_\v$ (defined in \eqref{8.PinnedGLFunctional}) in $H^1_g$. 

The pinning term is the periodic one (represented Figure \ref{Intro.FigureTermeChevillage}) or the non periodic one (represented Figure \ref{RepresentationGenTC}).

Recall that we fix $\delta=\delta(\v),\,\delta\to0$, $\lambda=\lambda(\v),\,\lambda\equiv1$ or $\lambda\to0$ satisfying \eqref{8.MainHyp}. If the pinning term is not periodic then we add the hypothesis $\lambda\to0$.
\subsection{Sharp Upper Bound, $\eta$-ellipticity and Uniform Convergence}
\subsubsection{Sharp Upper Bound and an $\eta$-ellipticity result}

We may easily prove the following upper bound.  
\begin{lem}\label{L8.UpperboundAuxPb}
Assume that $\dfrac{\rho}{\lambda\delta}\to0$ (or $\dfrac{\rho}{\lambda\delta^{3/2}}\to0$ if the pinning term is not periodic), then we have 
\begin{equation}\label{8.UpperboundAuxPb}
\inf_{v\in H^1_{g}(\O,\C)}F_\v(v)\leq db^2(\pi \ln \frac{b\rho}{\v}+\gamma)+J_{\rho,\v}+o_\v(1),
\end{equation}
where $\gamma>0$ is a universal constant defined in \cite{BBH}, Lemma IX.1.
\end{lem}
\begin{proof}


We construct a suitable test function $\tilde{w}_\v\in H^1_g$ (for sufficiently small $\v$).

From Proposition \ref{P8.MainAuxPb:DirVsDeg}, one may consider  $(x^\v_1,...,x^\v_{d})={\bf x}^\v\in\O^d$, a minimal configuration for $J_{\rho,\v}$.

Note that since $\dfrac{\rho}{\lambda\delta}\to0$ (or $\dfrac{\rho}{\lambda\delta^{3/2}}\to0$ if the pinning term is not periodic), from  
Corollaries \ref{C8.CorolSecondAuxPb} $\&$ \ref{C8.AnAlmostMinConfigIsAnAlmostMinConf} (or Proposition \ref{P8.CorolSecondAuxPb} if the pinning term is not periodic), there are $\eta>0$ and $c>0$ s.t. for small $\v$ we have $B(x_i^\v,c\lambda\delta)\subset\o_\v$ and $\min_i\left\{\min_{i\neq j}|x_i-x_j|,\dist(x_i,\p\O)\right\}\geq\eta$.

Let $w_\v$ be a minimal map in $\J_{\rho,\v}({\bf x}^\v,{\bf 1})$ (Proposition \ref{P8.ExistenceOfminInIJ}). We denote ${\bf 1}:=(1,...,1)\in\N^d$

Let $u_{\v/(b\rho)}\in H^1(B(0,1),\C)$ be a global minimizer of 
\[
E^{0}_{\v/(b\rho)}(u)=\frac{1}{2}\int_{B(0,1)}\left\{|\n u|^2+\frac{b^2\rho^2}{2\v^2}(1-|u|^2)^2\right\},\,u\in H^1_{x/|x|}(B(0,1),\C).
\]

We consider the test function
\[
\tilde{w}_\v(x)=\begin{cases}w_\v&\text{in }\O_\rho\\\displaystyle\alpha^\v_iu_{\v/(b\rho)}\left(\frac{x-x_i^\v}{\rho}\right)&\text{in }B(x_i^\v,\rho)\end{cases}.
\]
Here the constants $\alpha_i^\v\in\S^1$ are s.t. $w_\v(x_i^\v+\rho{\rm e}^{\imath\theta})=\alpha_i^\v{\rm e}^{\imath\theta}$.

Estimate \eqref{8.UpperboundAuxPb} is obtained by using the fact that $E^{0}_\v(u_\v)=\pi|\ln\v|+\gamma+o_\v(1)$ as $\v\to0$ (see \cite{BBH} Lemma IX.1) and Proposition \ref{P8.UepsCloseToaeps}.

\end{proof}
Note that
\begin{equation}\label{8.CoarseEstimate}
I_{\rho,\v}\leq J_{\rho,\v}\leq \pi d |\ln\rho|+C.
\end{equation}
We now turn to the $\eta$-ellipticity.

We denote by $v_\v$ a global minimizer of $F_\v$ in $H^1_g$. We extend $|v_\v|$ with the value $1$ outside $\O$.

One of the main ingredients in this work is the following result. 
\begin{lem}\label{L8.BadDiscsLemma}[$\eta$-ellipticity lemma]

Let $0<\alpha<1/2$. Then the following results hold:
\begin{enumerate}
\item If for $\v<\v_0$
\[
F_\v(v_\v,B_{}(x,\v^{\alpha})\cap\O)\leq \chi^2|\ln\v|-C_1,
\]
then we have
\[
|v_\v|\geq1-C\chi\text{ in }B(x,{\v^{{2\alpha}}}).
\]
Here, $\chi_\v\in(0,1)$ is s.t. $\chi_\v\to0$ and $\v_0>0$, $C>0$, $C_1>0$ depend only on $b,\alpha,\chi,\O,\|g\|_{C^1(\p\O)}$.
\item If for $\v<\v_0$
\[
F_\v(v_\v,B_{}(x,\v^{\alpha})\cap\O)\leq C|\ln\v|,
\]
then we have
\[
|v_\v|\geq\mu\text{ in }B(x,\v^{2\alpha}).
\]
Here, $\mu\in(0,1)$ and $\v_0,C>0$ depend only on $b,\alpha,\mu,\O,\|g\|_{C^1(\p\O)}$.

\end{enumerate}
\end{lem}
This result is a direct consequence of Lemma 1 in \cite{publi3}.
\subsubsection{Uniform convergence of $|v_\v|$ outside $\o_\v$}
With the help of  Lemma \ref{L8.BadDiscsLemma}, we are in position to establish uniform convergence of $|v_\v|$ to $1$ far away from $\overline{\o_\v}$.

\begin{prop}\label{P8.ConvOfBadDiscsTo0}
Let $10^{-2}\cdot\dist(\o,\p Y)>\mu>0$ and $K_\v^\mu=\{x\in\O\,|\,\dist(x,\o_\v)\geq\mu\lambda\delta\}$. 
Then, for sufficiently small $\v$, we have
\[
|v_\v|\geq1-C\sqrt\frac{|\ln(\lambda\delta)|}{|\ln\v|}\text{ in } K_\v^\mu.
\] 
Here $C$ is independent of $\v$ and $\mu$.

Furthermore, if for some small $\v$, we have $\displaystyle| v_\v(x)|<1-C\sqrt\frac{|\ln(\lambda\delta)|}{|\ln\v|}$,  then 
\[
F_\v(v_\v,B_{}(x,\v^{1/4}))\geq\displaystyle\frac{2(\pi d+1)}{b^2(1-b^2)}|\ln(\lambda\delta)|.
\]
\end{prop}
\begin{proof}
Using Lemma \ref{L8.BadDiscsLemma}.1 with $\alpha=1/4$ and $\displaystyle\chi=\sqrt{\frac{2(\pi d+1)}{b^2(1-b^2)}\frac{|\ln(\lambda\delta)|}{|\ln\v|}}$, we obtain the existence of $C>0$ s.t. for $\v>0$ sufficiently small: 
\begin{center}
if $\displaystyle F_\v(v_\v,B_{}(x,\v^{1/4}))<\frac{2(\pi d+1)}{b^2(1-b^2)}|\ln(\lambda\delta)|$, then we have $|v_\v|\geq 1-C\chi$ in $B_{}(x,\v^{1/2})$.
\end{center}
In order to prove Proposition \ref{P8.ConvOfBadDiscsTo0}, we argue by contradiction. There are $\v_n\downarrow0$, $\mu>0$ and  $x_n\in  K_{\v_n}^\mu$ s.t.
\[
|v_{\v_n}(x_n)|<1-C\chi.
\] 
From \eqref{8.UepsCloseToaeps}, we find
\begin{equation}\label{7.HatUepsCloseToaeps}
| U_{\v_n}-1|\leq C_0{\rm e}^{-\frac{\alpha\mu }{2\xi}}\text{ in } K_{\v_n}^{\mu/2},\,\xi=\frac{\v_n}{\lambda\delta}.
\end{equation}
Consequently, Lemma \ref{L8.BadDiscsLemma}, the definition of $C$ and \eqref{7.HatUepsCloseToaeps} imply that for large $n$, 
\begin{equation}\label{8.ContradictionForUnifConv}
\frac{1}{2}\int_{B_{}(x_n,\v_n^{1/4})}{\left\{|\n v_{\v_n}|^2+\frac{1}{2\v_n^2}(1-|v_{\v_n}|^2)^2\right\}}\geq\frac{2(\pi d+1)}{b^2(1-b^2)}|\ln(\lambda\delta)|+o_\v(1).
\end{equation}


We extend $v_{\v}$ to $\O':=\O+B_{}(0,1)$ with the help of a fixed smooth $\S^1$-valued map $v$ s.t. $v=g$ on $\p\O$. We also extend $U_{\v}$ and $a_{\v}$ with the value $1$ outside $\O$.


For $n$  sufficiently large, we have
\[
\frac{1}{2}\int_{\O'}{\left\{|\n v_{\v_n}|^2+\frac{1}{2\v_n^2}(1-|v_{\v_n}|^2)^2\right\}}\leq C|\ln\v_n|.
\]
Theorem 4.1 in \cite{SS1} applied with $r=10^{-2}\cdot\lambda\delta\mu$ and for large $n$, implies the existence of  $\mathcal{B}^n=\{B^n_j\}$ a finite disjoint covering by balls of 
\[
\left\{x\in \O'\,\left|\,\dist(x,\p\O')>\dfrac{\v_n}{b}\text{ and }1-|v_{\v_n}(x)|\geq\left(\dfrac{\v_n}{b}\right)^{1/8}\right.\right\}
\]
s.t.
\[
{\rm rad}\,( \mathcal{B}^n)\leq10^{-2}\cdot\lambda\delta\mu
\]
satisfying
\begin{eqnarray}\nonumber
\frac{1}{2}\int_{\cup B^n_j}{\left\{|\n v_{\v_n}|^2+\frac{b^2}{2\v_n^2}(1-|v_{\v_n}|^2)^2\right\}}&\geq&\pi \sum_jd^n_j(|\ln\v_n|-|\ln(\lambda\delta)|)-C
\\\nonumber
&=&\pi \sum_jd_j^n|\ln\xi|-C.
\end{eqnarray}
Here, ${\rm rad}\,( \mathcal{B}^n)=\sum_j{\rm rad}(B^n_j)$, ${\rm rad}(B)$ stands for the radius of the ball $B$, $\xi=\v_n/(\lambda\delta)$ and the integers $d^n_j$ are defined by
\[
d^n_j=\begin{cases}|\deg_{\p B^n_j}(v_{\v_n})|&\text{if }B_j^n\subset\{x\in\O'\,|\,\dist(x,\p\O')>\dfrac{\v_n}{b}\}\\0&\text{otherwise}\end{cases}.
\]
Since 
 $B_j\subset\O+B_{1/2}\subset \{x\in\O'\,|\,\dist(x,\p\O')>\dfrac{\v_n}{b}\}$, we obtain
 \begin{equation}\label{8.LowerBoundOnSandierSerfatyCovering}
 \frac{1}{2}\int_{\cup B^n_j}{\left\{|\n v_{\v_n}|^2+\frac{b^2}{2\v_n^2}(1-|v_{\v_n}|^2)^2\right\}}\geq \pi d|\ln\xi|-C.
 \end{equation}

From \eqref{7.HatUepsCloseToaeps} and \eqref{8.MainHyp} 
we have 
\begin{eqnarray}\nonumber
{F}_{\xi}( v_{\v_n},\cup_j B_j\cup B(x_n,\v_n^{1/4}))&\geq&\frac{b^2(1-b^2)}{2}\int_{B(x_n,\v_n^{1/4})}{\left\{|\n v_{\v_n}|^2+\frac{1}{2\v_n^2}(1-|v_{\v_n}|^2)^2\right\}}+\\\label{7.HatUepsCloseToaepsBISBIS}&&+\frac{b^2}{2}\int_{\cup_j B_j}{\left\{|\n v_{\v_n}|^2+\frac{b^2}{2\v_n^2}(1-| v_{\v_n}|^2)^2\right\}}+o_n(1).
\end{eqnarray}



By combining \eqref{8.UpperboundAuxPb} (with $\rho=\lambda^2\delta^2$), \eqref{8.CoarseEstimate}, \eqref{8.ContradictionForUnifConv}, \eqref{8.LowerBoundOnSandierSerfatyCovering} and \eqref{7.HatUepsCloseToaepsBISBIS}, we find that
\begin{eqnarray}\nonumber
\pi db^2\ln[(\lambda\delta)/\xi]+\pi d|\ln[(\lambda\delta)^2]|&\geq&  F_{\v_n}( v_{\v_n},\O')-\mathcal{O}_{n}\left(1\right)
\\\nonumber&\geq&  F_{\v_n}( v_{\v_n},\cup_j B_j\cup B(x_n,\v_n^{1/4}))-\mathcal{O}_{n}\left(1\right)
\\\nonumber
&\geq& \pi db^2|\ln\xi|+2(\pi d+1)|\ln(\lambda\delta)|-\mathcal{O}_{n}(1)
\end{eqnarray}


which is a contradiction. 
This completes the proof of Proposition \ref{P8.ConvOfBadDiscsTo0}.
\end{proof}
\subsection{Bad discs}
\subsubsection{Construction and first properties of bad discs}
A fundamental tool in this article is the use of \emph{ad-hoc} coverings of $\{|v_\v|\leq7/8\}$ by small discs. The best radius for a covering of $\{|v_\v|\leq7/8\}$ should be of the order $\v$.  But the construction of such covering need some preliminary results.

Roughly speaking, the way to get a "sharp" covering is to consider a trivial covering and to "clean" it by dropping some discs with the help an "energetic test" ($\eta$-ellipticity result). 

Here, we used two kinds of energetic tests: Lemma \ref{L8.BadDiscsLemma} and Theorem III.3 in \cite{BBH}. Theorem III.3 in \cite{BBH} gives the most precise results (it allows to deal with discs with radius $\mathcal{O}(\v)$) but it needs a bound on the potential part $\v^{-2}\int_\O(1-|v_\v|^2)^2$ which is the purpose of Proposition \ref{P8.BoundAnnulusArroundBadDiscs}. In order to prove this bound (Proposition \ref{P8.BoundAnnulusArroundBadDiscs}), we first use larger discs (discs with radius $\rho$, $\v\ll\rho\ll\lambda\delta^{P+1}$). The construction of intermediate coverings is done via Lemma \ref{L8.BadDiscsLemma}.

We first consider 
\begin{notation}\label{NOTATIONEPS1/4Disc}A trivial covering of $\O$ by discs

For $\v>0$, we fix a family of discs $\left(B(x_i,\v^{{1/4}})\right)_{i\in I}$ s.t
\[
x_i\in\O,\,\forall\,i\in I,
\]
\[
B(x_i,\v^{{1/4}}/4)\cap B(x_i,\v^{{1/4}}/4)=\emptyset\text{ if }i\neq j,
\]
\[
\cup_{i\in I}B(x_i,\v^{{1/4}})\supset\O.
\]
\end{notation}
Then we select discs (using Lemma \ref{L8.BadDiscsLemma}) and we define
\begin{notation}The initial good/bad discs
\begin{enumerate}[$\bullet$]\item Let $C_0=C_0(1/4,7/8),\,\v_0=\v_0({1/4},7/8)$ be defined by  Lemma \ref{L8.BadDiscsLemma}.2. For $\v<\v_0$, we say that $B(x_i,\v^{{1/4}})$ is an {\it initial good disc} if
\[
F_\v(v_\v,B({x_i,\v^{{1/4}}})\cap\O)\leq C_0|\ln\v|
\]
and $B(x_i,\v^{{1/4}})$ is an {\it initial bad disc} if
\begin{equation}\label{8.BadBallEnergyBoundWithoutZero}
F_\v(v_\v,B({x_i,\v^{{1/4}}})\cap\O)> C_0|\ln\v|.
\end{equation}
\item We let $J=J(\v):=\{i\in I\,|\,B({x_i,\v^{{1/4}}})\text{ is an initial  bad disc}\}$.
\end{enumerate}
\end{notation}
An easy consequence of Lemma \ref{L8.UpperboundAuxPb} is
\begin{lem}\label{L8.finitenumberbaddiscs}The number of initial bad discs is bounded

There is an integer $N$ which depends only on $g$ and $\O$ s.t.
\[
{\rm Card}\,J\leq N.
\]
\end{lem}
\begin{proof}
Since each point of $\O$ is covered by at most $C>0$ (universal constant) discs $B({x_i,\v^{{1/4}}})$, we have
\[
\sum_{i\in J}F_\v(v_\v,B({x_i,\v^{{1/4}}})\cap\O)\leq CF_\v(v_\v,\O).
\]
The previous assertion implies that $\text{Card }J\leq\dfrac{C\pi d}{C_0}+1$.
\end{proof}
Let $\rho(\v)=\rho\downarrow0$ be s.t.
\begin{equation}\label{8.ConditionOnTheMicroRadius}
\frac{\rho}{\lambda\delta^{\num+1}}\to0
\text{ and }\frac{|\ln\rho|^3}{|\ln\v|}\to0.
\end{equation}
Note that from Assumption \eqref{8.MainHyp}, such a $\rho$ exists, \emph{e.g.}, $\rho=(\lambda\delta)^{\num+2}$. (Recall that if the pinning term is periodic then $\num=1$)

The following result is a straightforward variant of Theorem IV.1 in \cite{BBH}.
\begin{lem}\label{L8.SeparationBadDiscs}Separation of the initial bad discs

Let $\v_n\downarrow0$. Then (possibly after passing to a subsequence and relabeling the indices), we may choose $J'\subset J$ and a constant $\kappa$ independent of $n$ s.t.
\[
J'=\{1,...,N'\},\,N'={\rm Cst},
\]
\[
|x_i-x_j|\geq16\kappa\rho\text{ for }i,j\in J',\,i\neq j,
\]
and 
\[
\cup_{i\in J}{B(x_i,{\v_n^{1/4}})}\subset \cup_{i\in J'}{B(x_i,\kappa\rho)}.
\]
\end{lem}
\begin{notation}\label{NOTATIONRHOBAD}The $\rho$-bad disc

For $i\in J'$, we say that $B_{}(x_i, 2\kappa\rho)$ is a $\rho$-bad disc.
\end{notation}
\begin{prop}\label{P8.BadDiscSeparation}
We have
\begin{enumerate}
\item $\displaystyle\frac{\rho}{\dist(B_{}(x_i, 2\kappa\rho),\p\O)}\to0$,
\item $\deg_{\p B_{}(x_i,2\kappa\rho)}(v_{\v_n})>0$,
\item $\displaystyle F_{\v_n}(v_{\v_n},B_{}(x_i,2\kappa\rho))\geq \pi b^2 \deg_{\p B_{}(x_i,2\kappa\rho)}(v_{\v_n})\ln\frac{\rho}{\v_n}-\mathcal{O}(1)$,
\item $\displaystyle|v_{\v_n}|\geq1-C\sqrt{\frac{|\ln\rho|}{|\ln\v_n|}}$ in $\O\setminus\cup_{i\in J'}\overline{B_{}(x_i,2\kappa\rho)}$.
\end{enumerate}
\end{prop}
\begin{proof}
We prove Assertions 1., 2. and 3.. Set 
\[
J_0':=\{i\in J'\,|\,\deg_{\p (B(x_i,2\kappa\rho)\cap\O)}(v_{\v_n})>0\}.
\]
Since $|v_{\v_n}|\geq \frac{7}{8}$ in $\O\setminus\cup_{i\in J'}\overline{B(x_i,2\kappa\rho)}$, we have  
\begin{equation}\label{8.GoodConfiDegree}
0<d=\sum_{I\in J'}\deg_{\p (B(x_i,2\kappa\rho)\cap\O)}(v_{\v_n})\leq\sum_{I\in J'_0}\deg_{\p (B(x_i,2\kappa\rho)\cap\O)}(v_{\v_n}).
\end{equation}
Consequently $J'_0\neq\emptyset$.
 
Up to a subsequence, we may assume that $J_0'$ is independent of $n$.

From Proposition \ref{P8.ConvOfBadDiscsTo0}, for all $i\in J'_0$ , we have $\dist(B(x_i,\v^{1/4}),\p\O)\gtrsim\delta$ (or $\delta^\num$ if the pinning term is not periodic).  Consequently, for $i\in J'_0$ we find
\begin{equation}\label{8.CalculIntermediaairedemandé}
\frac{\dist(B_{}(x_i,2\kappa\rho),\p\O)}{\rho}\to+\infty
\end{equation}
since $\dfrac{\rho}{\lambda\delta^{\num+1}}\to0$.

Assertions 1., 2. and 3. will follow from the estimate
\begin{equation}\label{8.BadDiscDegLowBound}
F_{\v_n}(v_{\v_n},B(x_i,2\kappa\rho))\geq  b^2\pi \deg_{\p B(x_i,2\kappa\rho)}(v_{\v_n})\ln\frac{\rho}{\v_n}-\mathcal{O}(1),
\end{equation}
valid for $i\in J_0'$. Indeed, assume for the moment that \eqref{8.BadDiscDegLowBound} holds for $i\in J_0'$.

Then, by combining \eqref{8.UpperboundAuxPb}, \eqref{8.CoarseEstimate}, \eqref{8.BadBallEnergyBoundWithoutZero}, \eqref{8.ConditionOnTheMicroRadius}, \eqref{8.GoodConfiDegree} and \eqref{8.BadDiscDegLowBound},  we find that $J_0'=J'$, \emph{i.e.}, 2. holds. 
Consequently, by combining Assertion 2. with \eqref{8.CalculIntermediaairedemandé}, Assertion 1. yields and from Assertion 2. and \eqref{8.BadDiscDegLowBound}, Assertion 3. holds.

We now turn to the proof of \eqref{8.BadDiscDegLowBound}, which relies on Proposition 4.1 in \cite{SS1}.  We apply this proposition in the domain $B=B(0,2\kappa)$, to the function $v'(x)=v_{\v_n}\left[\rho(x-x_i)\right]$ and with the rescaled parameter $\displaystyle\xm=\frac{\v}{\rho}$.

Note that, from \eqref{8.ConditionOnTheMicroRadius}, $\v\ll\xm\ll\rho\ll\lambda\delta^{\num+1}$ and $|\ln\v|\sim|\ln\xm|\gg|\ln(\lambda\delta)|$.

Clearly, $v'$ satisfies  
\begin{eqnarray*}
 \int_B{\left\{|\n v'|^2+\frac{1}{\xm^2}(1-|v'|^2)^2\right\}}&=& \int_{B(x_i,2\kappa\rho)}{\left\{|\n v_{\v_n}|^2+\frac{1}{\v^2}(1-|v_{\v_n}|^2)^2\right\}}
 \\&=&\mathcal{O}(|\ln\v|)=\mathcal{O}(|\ln\xm|).
\end{eqnarray*}
Hence, one may apply the following result of Sandier and Serfaty: there is $(B_j)_{j\in I}$, a finite covering of 
\[
\{x\in B(0,2\kappa-\xm/b)\,|\,|v'(x)|\leq1-(\xm/b)^{1/8}\} 
\]
with disjoint balls $B_j$ of radius $r_j<10^{-3}$ s.t. 
\[
\frac{1}{2}\int_{\displaystyle B\cap \cup B_j}{\left\{|\n v'|^2+\frac{b^2}{\xm^2}(1-|v'|^2)^2\right\}}\geq\pi \sum_jd_j|\ln\xm|-\mathcal{O}(1);
\]
here 
$
d_j=\begin{cases}|\deg_{\p B_j}(v')|&\text{if }B_j\subset B(0,2\kappa-\xm/b)\\0&\text{otherwise}\end{cases}.
$

Note that from construction, $\{|v_{\v_n}|\leq7/8\}\subset\cup_{J}B(x_i,\v_n^{1/4})\subset\cup_{J'} B(x_i,\kappa\rho)$. Consequently: 
\[
\text{if }\deg_{\p (B_j\cap B(0,2\kappa-\xm/b))}(v')\neq 0\text{, then we have }B_j\subset B(0,\frac{3}{2}\kappa).
\]
Therefore, $\sum d_j=\deg_{\p B(0,2\kappa)}(v')=\deg_{\p B(x_i,2\kappa\rho)}(v_{\v_n})$ and
\begin{eqnarray}\nonumber
 \frac{1}{2}\int_{B(x_i,2\kappa\rho)}{\left\{|\n v_{\v_n}|^2+\frac{1}{2\v^2}(1-|v_{\v_n}|^2)^2\right\}}&\geq&\pi\deg_{\p B(x_i,2\kappa\rho)}(v_{\v_n})|\ln\xm|-\mathcal{O}(1)
\\\nonumber&=&\pi\deg_{\p B(x_i,2\kappa\rho)}(v_{\v_n})\ln\frac{\rho}{\v}-\mathcal{O}(1).
\end{eqnarray}
Thus \eqref{8.BadDiscDegLowBound} holds. 


The last assertion is obtained using Lemmas \ref{L8.UpperboundAuxPb} $\&$ \ref{L8.BadDiscsLemma}. Indeed, note that the proof of  \eqref{8.BadDiscDegLowBound} gives a more precise result
\[
F_{\v_n}(v_{\v_n},B(x_i,\frac{3}{2}\kappa\rho))\geq  b^2\pi \deg_{\p B(x_i,2\kappa\rho)}(v_{\v_n})\ln\frac{\rho}{\v_n}-\mathcal{O}(1).
\]
Let $x\in\O\setminus\cup_{J'} B(x_i,2\kappa\rho)$ then $B(x,\v_n^{1/4})\cap B(x_i,\frac{3}{2}\kappa\rho)=\emptyset$. Consequently, using  Lemma \ref{L8.UpperboundAuxPb} and the previous lower bound, we obtain:
\[
F_{\v_n}(v_{\v_n},B(x,\v_n^{1/4}))\leq I_{2\kappa\rho,\v_n}+C_0\leq \pi d|\ln\rho|+C_0.
\] 
Therefore, from Lemma \ref{L8.BadDiscsLemma}, there is $C>0$, independent of $x$ s.t. $|v_{\v_n}(x)|\geq\displaystyle1-C\sqrt{\frac{|\ln\rho|}{|\ln\v_n|}}$.
\end{proof}
\subsubsection{Location and degree of bad discs}
Let $\displaystyle w_n=\frac{v_{\v_n}}{|v_{\v_n}|}\in H^1(\O\setminus\cup_{J'} \overline{B(x_i,2\kappa\rho)},\S^1)$.
\begin{prop}\label{P8.StudyOfS1Part}
The map $w_n$ is an almost minimal function for $I_{2\kappa\rho,\v_n}$.
\end{prop}
\begin{proof}
Indeed, denote $\displaystyle K_n=\frac{1}{2}\int_{\O\setminus\cup_{J'} \overline{B(x_i,2\kappa\rho)}}{U_{\v_n}^2|\n w_n|^2}$, then we have


\begin{eqnarray*}
K_n\!\!\!\!\!\!\!\!&\leq&\!\!\!\!\!\!\!\!\!\!\!\!\!\!\!\!\!\!F_{\v_n}(v_{\v_n},\O\setminus\cup_{J'} \overline{B(x_i,2\kappa\rho)})
+\int_{\O\setminus\cup_{J'} \overline{B(x_i,2\kappa\rho)}}{U_{\v_n}^2(1-|v_{\v_n}|^2)|\n w_n|^2}
\\&=&\!\!\!\!\!\!\!\!\!\!\!\!\!\!\!\!\!\!\!\!F_{\v_n}(v_{\v_n},\O)-F_{\v_n}(v_{\v_n},\cup_{J'} B(x_i,2\kappa\rho))+
\int_{\O\setminus\cup_{J'} \overline{B(x_i,2\kappa\rho)}}{U_{\v_n}^2(1-|v_{\v_n}|^2)|\n w_n|^2}
\\&\leq \text{\eqref{8.UpperboundAuxPb}, Prop \ref{P8.BadDiscSeparation}}\leq&I_{2\kappa\rho,\v_n}+C\sqrt{\frac{|\ln\rho|}{|\ln\v_n|}}\int_{\O\setminus\cup_{J'} \overline{B(x_i,2\kappa\rho)}}{U_{\v_n}^2|\n w_n|^2}+\mathcal{O}(1)
\\&\leq \text{\eqref{8.UpperboundAuxPb}, Prop \ref{P8.BadDiscSeparation}}\leq&I_{2\kappa\rho,\v_n}+C\sqrt{\frac{|\ln\rho|}{|\ln\v_n|}}F_{\v_n}(v_{\v_n},\O\setminus\cup_{J'} \overline{B(x_i,2\kappa\rho)})+\mathcal{O}(1)
\\&\leq\begin{array}{c}\eqref{8.UpperboundAuxPb},\eqref{8.CoarseEstimate}\\ \text{Prop \ref{P8.BadDiscSeparation}}\end{array}\leq&I_{2\kappa\rho,\v_n}+C\sqrt{\frac{|\ln\rho|^3}{|\ln\v_n|}}+\mathcal{O}(1)
\\&\leq\eqref{8.ConditionOnTheMicroRadius}\leq&I_{2\kappa\rho,\v_n}+\mathcal{O}(1).
\end{eqnarray*}
\end{proof}
\begin{remark}
Note that the penultimate line in the proof of Proposition \ref{P8.StudyOfS1Part} is the main use of \eqref{8.MainHyp} (which is express in \eqref{8.ConditionOnTheMicroRadius}).
\end{remark}
By combining Proposition \ref{P8.ToMinimizeSecondPbThePointAreFarFromBoundAndHaveDegree1} with Proposition \ref{P8.StudyOfS1Part} in the periodic case or Proposition \ref{P8.CorolSecondAuxPb} if the pinning term is not periodic, we obtain the following
\begin{cor}\label{C8.TheCenterAreInGoodPosition}
The configuration $\{(x_1,...,x_{N'}),(\deg_{\p B(x_1,2\kappa\rho)}(v_{\v_n}),...,\deg_{\p B(x_{N'},2\kappa\rho)}(v_{\v_n}))\}$ is an almost minimal configuration of $I_{2\kappa\rho,\v_n}$ and consequently, $N'=d$, $\deg_{\p B(x_{i},2\kappa\rho)}(v_{\v_n})=1$ for all $i$ and there is $\eta_0>0$ independent of large $n$ s.t.
\[
\min\left\{\min_{i\neq j}|x_i-x_j|,\min_i\dist(x_i,\p\O)\right\}>2\eta_0,
\]
\[
B(x_i,2\eta_0\lambda\delta)\subset\o_\v.
\]
\end{cor}

\subsection{$H_{\rm loc}^1$-weak convergence}

In order to keep notations simple, we replace from now on, $2\kappa\rho$ by $\rho/2$. 


Using Corollary \ref{C8.TheCenterAreInGoodPosition}, there is $\{a_1,...,a_{d}\}\subset\O$ s.t. possibly after passing to a subsequence, we have $x^{n}_i=x_i\to a_i$.
 
Let $\rho_0>0$ be defined as
\[
\rho_0=10^{-2}\cdot\min_{k\neq l}\left\{\dist(a_k,\p\O),|a_k-a_l|\right\}.
\]
\subsubsection{The contribution of the modulus is bounded in the whole domain}
We are now in position to bound the potential part of $F_\v(v_\v)$. More precisely we have

\begin{prop}\label{P8.BoundAnnulusArroundBadDiscs}
We have $\displaystyle\int_{\O}{\left\{|\n |v_{\v_n}||^2+\frac{1}{\v_n^2}(1-|v_{\v_n}|^2)^2\right\}}=\mathcal{O}(1)$.
\end{prop}
\begin{proof}
From \eqref{8.UpperboundAuxPb}, Proposition \ref{P8.BadDiscSeparation} (Assertion 1., 2. and 3.) and Proposition \ref{P8.StudyOfS1Part}, we infer that 
\[
\displaystyle\int_{\O\setminus \cup_i \overline{B(x_i,\rho/2)}}{\left\{|\n |v_{\v_n}||^2+\frac{1}{\v_n^2}(1-|v_{\v_n}|^2)^2\right\}}=\mathcal{O}(1).
\]
Consequently it suffices to obtain a similar estimate in $B(x_i,\rho/2)$. Note that $B(x_i,\rho)\subset \o_\v$. Thus, if we set 
\[
u'(x)=\frac{u_{\v_n}(x_i+\rho x)}{b}:B(0,1)\to\C,
\]
then $u'$ solves
\[
-\Delta u'=\frac{1}{\left[{\v_n}/({b\rho})\right]^2}u'(1-|u'|^2)\text{ in }B(0,1).
\]
From \cite{BOS1}, we obtain
\[
\frac{1}{2}\int_{B(0,1/2)}{\left\{\left|\n |u'|\right|^2+\frac{b^2\rho^2}{2\v_n^2}(1-|u'|^2)^2\right\}}=\mathcal{O}(1).
\]
 This estimate is the subject of Theorem 1  for the potential part and Proposition 1 in \cite{BOS1} for the gradient of the modulus (see also Corollary 1 in \cite{BOS1}).
 
 Set $\displaystyle K_n=\frac{1}{2}\int_{B(0,1/2)}{\left\{\left|\n |u'|\right|^2+\frac{b^2\rho^2}{2\v_n^2}(1-|u'|^2)^2\right\}}$. Using Proposition \ref{P8.UepsCloseToaeps}, we obtain
 \begin{eqnarray*}
K_n=\mathcal{O}(1)&=&\frac{1}{2b^2}\int_{B(x_i,\rho/2)}{\left\{\left|\n |U_{\v_n}v_{\v_n}|\right|^2+\frac{b^4}{2\v_n^2}\left(1-\frac{|U_{\v_n}v_{\v_n}|^2}{b^2}\right)^2\right\}}
 \\&=&\frac{1}{2}\int_{B(x_i,\rho/2)}{\left\{\left|\n |v_{\v_n}|\right|^2+\frac{b^2}{2\v_n^2}\left(1-|v_{\v_n}|^2\right)^2\right\}}+o_n(1).
 \end{eqnarray*}
 Consequently, Proposition \ref{P8.BoundAnnulusArroundBadDiscs} holds.
\end{proof}
\subsubsection{We bound the energy in a fixed perforated domain}
 \begin{prop}\label{P7.LocalizationOfTheEnergyPart1}
For $0<\eta\leq\rho_0$, there is $C(\eta)>0$ independent of $n$ s.t. we have
\begin{equation}\label{8.EnergyOutsideTheHole}
\frac{1}{2}\int_{\O\setminus\cup \overline{B(a_i,\eta)}}{|\n v_{\v_n}|^2}\leq C(\eta).
\end{equation}
\end{prop}
\begin{proof}

We argue by contradiction and we assume that there is $\eta>0$ s.t., up to pass to a subsequence, we have $\int_{\O\setminus\cup \overline{B(a_i,\eta)}}{|\n v_{\v_n}|^2}\to\infty$. 

Because $\int_{\O\setminus\cup \overline{B(a_i,\eta)}}{|\n v_{\v_n}|^2}=\int_{\O\setminus\cup \overline{B(a_i,\eta)}}{|v_{\v_n}|^2|\n w_n|^2+|\n (|v_{\v_n}|)|^2}$, from Propositions \ref{P8.BadDiscSeparation} $\&$ \ref{P8.BoundAnnulusArroundBadDiscs} we get $\int_{\O\setminus\cup \overline{B(a_i,\eta)}}{|\n w_n|^2}\to\infty$. Therefore, we have $\int_{\O\setminus\cup \overline{B(x_i,10^{-1}\eta)}}{|\n w_n|^2}\to\infty$.

It is clear that we may get a map $\tilde{w}_n\in{\J}_{10^{-1}\eta,\v_n}({\bf x}_{\v_n},{\bf 1})$ s.t. $\int_{\O\setminus\cup \overline{B(x_i,10^{-1}\eta)}}{|\n \tilde{w}_n|^2}\leq C(\eta)$. 

For $i=1,...,d$, using Proposition \eqref{P7.MyrtoRingDegDir} (Appendix \ref{S8.ProofOFfirstAuxiliaryProblem}, Section \ref{Sectionkjsdhsdjfghpp}, Page \pageref{P7.MyrtoRingDegDir}), we get the existence of a map $\tilde{w}_{i,n}\in H^1(B(x_i,10^{-1}\eta)\setminus\overline{B(x_i,\rho/2)},\S^1)$ s.t. $\tilde{w}_{i,n}(x_i+10^{-1}\eta{\rm e}^{\imath \theta})=\tilde{w}_n(x_i+10^{-1}\eta{\rm e}^{\imath \theta})$ and
\[
\int_{B(x_i,10^{-1}\eta)\setminus\overline{B(x_i,\rho/2)}}U_{\v_n}^2|\n \tilde{w}_{i,n}|^2\leq \int_{B(x_i,10^{-1}\eta)\setminus\overline{B(x_i,\rho/2)}}U_{\v_n}^2|\n {w}_{n}|^2+\mathcal{O}(1)
\]
Therefore by extending $\tilde{w}_n$ with $\tilde{w}_{i,n}$ in  $B(x_i,10^{-1}\eta)\setminus\overline{B(x_i,\rho/2)}$ we get a map still denoted $\tilde{w}_n\in H^1_g(\O\setminus \cup \overline{B(x_i,\rho/2)},\S^1)$ s.t.
\[
\frac{1}{2}\int_{\O\setminus \cup \overline{B(x_i,\rho/2)}}|\n {w}_n|^2-\frac{1}{2}\int_{\O\setminus \cup \overline{B(x_i,\rho/2)}}|\n \tilde{w}_n|^2\to\infty
\]
which is in contradiction with Proposition \ref{P8.StudyOfS1Part}.

\end{proof}
Consequently, there is $v_*\in H^1_{\rm loc}(\overline{\O}\setminus\{a_1,...,a_d\},\S^1)$ s.t., up to pass to a subsequence, $v_{\v_n}\weak v_*$ in  $H^1_{\rm loc}(\overline{\O}\setminus\{a_1,...,a_d\})$. Next section is dedicate to the limiting equation of $v_*$.

\subsubsection{We establish the limiting equation}\label{SubSubSubSectionLimEq}

In order to obtain the expression of the homogenized problem, we use the {\it{unfolding operator}} (see \cite{CDG1}, definition 2.1). 

The use of the unfolding operator needs a slightly modification of the cell period. More precisely, instead of considering the $\delta\times\delta$-grid whose vertices-grid are the points $\{\delta(k,l)+(1/2,1/2)\,|\,k,l\in\Z\}$, we consider the one whose vertices-grid are $\{\delta(k,l)\,|\,k,l\in\Z\}$. 

Thus instead of having cells which contain one inclusion at their center we have cells with quarters of inclusion at their vertices. (See Figure \ref{Figure.Modifififif})
\begin{figure}[h]
\subfigure[Four period-cells which are obtained from $Y=(-1/2,1/2)^2$ and the new period-cell in dash obtained from $\tilde{Y}=(0,1)^2$]{
\psset{xunit=1.3cm,yunit=1.3cm}
\begin{pspicture*}(-6.05,-3.2)(0.05,3.05)
\psset{xunit=.65cm,yunit=.65cm,algebraic=true,dotstyle=o,dotsize=3pt 0,linewidth=0.8pt,arrowsize=3pt 2,arrowinset=0.25}
\psline[linewidth=1.6pt,linestyle=dashed,dash=4pt 4pt](-9,3)(-9,-3)
\psline[linewidth=1.6pt,linestyle=dashed,dash=4pt 4pt](-9,-3)(-3,-3)
\psline[linewidth=1.6pt,linestyle=dashed,dash=4pt 4pt](-3,-3)(-3,3)
\psline[linewidth=1.6pt,linestyle=dashed,dash=4pt 4pt](-3,3)(-9,3)
\psline(-6,0)(0,0)
\psline(0,0)(0,6)
\psline(0,6)(-6,6)
\psline(-6,6)(-6,0)
\psline(-6,0)(-6,-6)
\psline(-6,-6)(0,-6)
\psline(0,-6)(0,0)
\psline(0,0)(-6,0)
\psline(-6,0)(-12,0)
\psline(-12,0)(-12,-6)
\psline(-12,-6)(-6,-6)
\psline(-6,-6)(-6,0)
\psline(-12,0)(-6,0)
\psline(-6,0)(-6,6)
\psline(-6,6)(-12,6)
\psline(-12,6)(-12,0)
\rput(-12,-6){$\bullet$}
\rput(-10.8,-5.5){{$(-\frac{\delta}{2}, -\frac{\delta}{2})$}}
\rput(-6,-6){$\bullet$}
\rput(-4.9,-5.5){{$(\frac{\delta}{2}, -\frac{\delta}{2})$}}
\rput(-12,0){$\bullet$}
\rput(-11,0.5){{$(-\frac{\delta}{2}, \frac{\delta}{2})$}}
\rput(-8.97,-3){$\bullet$}
\rput(-9.5,-3.5){{$(0, 0)$}}
\rput{-21.7}(-3,3){\scalebox{1}{\psellipse(0,0)(2.56,1.37)}}
\rput{-21.7}(-9,3){\scalebox{1}{\psellipse(0,0)(2.56,1.37)}}
\rput{-21.7}(-9,-3){\scalebox{1}{\psellipse(0,0)(2.56,1.37)}}
\rput{-21.7}(-3,-3){\scalebox{1}{\psellipse(0,0)(2.56,1.37)}}
\end{pspicture*}
}\hfill
\subfigure[The new unit-cell $\tilde{Y}$ with four quarters of an inclusion and the values of $\tilde{a}=a^\lambda_{|\tilde{Y}}$]{
\begin{pspicture*}(-9,-4)(-3,3)
\psset{xunit=1cm,yunit=1.0cm,algebraic=true,dotstyle=o,dotsize=3pt 0,linewidth=0.8pt,arrowsize=3pt 2,arrowinset=0.25}
\psline(-9,3)(-9,-3)
\psline(-9,-3)(-3,-3)
\psline(-3,-3)(-3,3)
\psline(-3,3)(-9,3)
\rput(-6,0){\scalebox{3}{\rput{-21.7}(1,1){\scalebox{.3}{\psellipse(0,0)(2.56,1.37)}}
\rput{-21.7}(-1,1){\scalebox{.3}{\psellipse(0,0)(2.56,1.37)}}
\rput{-21.7}(1,-1){\scalebox{.3}{\psellipse(0,0)(2.56,1.37)}}
\rput{-21.7}(-1,-1){\scalebox{.3}{\psellipse(0,0)(2.56,1.37)}}}}
\rput(-8.1,-2.5){{$\tilde{a}=b$}}
\rput(-3.9,-2.5){{$\tilde{a}=b$}}
\rput(-8.1,2.5){{$\tilde{a}=b$}}
\rput(-3.9,2.5){{$\tilde{a}=b$}}
\rput(-6,-0){{$\tilde{a}=1$}}
\psframe[fillstyle=solid,fillcolor=white,linecolor=black](-10,-3)(15,-15)
\end{pspicture*}
}\caption{The modification of the reference cell}\label{Figure.Modifififif}
\end{figure}
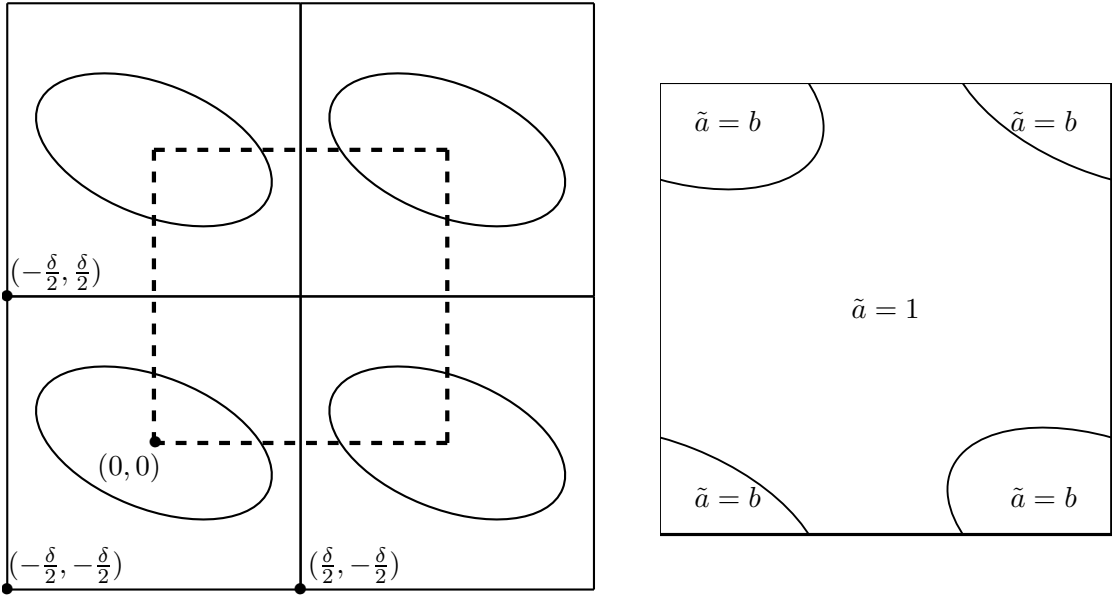

More specifically, we define, for $\O_0\subset\R^2$ an open set, $p\in(1,\infty)$ and $\delta>0$,
\[
\begin{array}{cccc}
\mathcal{T}_\delta:&L^p(\O_0)&\to&L^p(\O_0\times \tilde{Y})
\\
&\phi&\mapsto&\mathcal{T}_\delta(\phi)(x,y)=\left\{\begin{array}{cl}\phi\left(\delta\left[\displaystyle\frac{x}{\delta}\right]+\delta y\right)&\text{for }(x,y)\in\tilde{\O}^{\rm incl}_\delta\times \tilde{Y}\\0&\text{for }(x,y)\in\Lambda_\delta\times \tilde{Y}\end{array}\right.
\end{array}.
\]
Here, $\tilde{Y}=(0,1)^2$, $[s]$ is the integer part of $s\in\R$ and 
\[
\tilde{\O}^{\rm incl}_\delta:=\bigcup_{\substack{\tilde{Y}^K_\delta\subset\O_0,\,K\in\Z^2\\\tilde{Y}^K_\delta=\delta\cdot(K+\tilde{Y})}}\overline{\tilde{Y}^K_\delta},\,\Lambda_\delta:=\O_0\setminus\tilde{\O}^{\rm incl}_\delta\text{ and }\left[\frac{x}{\delta}\right]:=\left(\left[\frac{x_1}{\delta}\right],\left[\frac{x_2}{\delta}\right]\right).
\]

An adaptation of a result of Sauvageot (\cite{MS1}, Theorem 4) gives the following 
\begin{prop}\label{P8.asymptotiquedegnonzero}
Let $\O_0\subset\R^2$ be a smooth bounded open set. Let $v_n\in H^2(\O_0,\C)$ be s.t.
\begin{enumerate}
\item $|v_n|\leq1$ and $\displaystyle\int_{\O_0}(1-|v_n|^2)^2\to0$,
\item $v_n\weak v_*$ in $H^1(\O_0)$ for some $v_*\in H^1(\O_0,\S^1)$,
\item there are $H_n \in W^{1,\infty}(\O_0,[b^2,1])$ and $\delta=\delta_n\downarrow0$ s.t. $\mathcal{T}_\delta(H_n)(x,y)\to H_0(y)$ in $L^2(\O_0\times \tilde{Y})$,
\item $-{\rm div}(H_n\n v_n)=v_nf_n(x)$, $f_n\in L^\infty(\O_0,\R)$.
\end{enumerate}
Then $v_*$ is a solution of
\begin{equation}\nonumber
-\Div(\mathcal{A}\n v_*)=(\mathcal{A}\n v_*\cdot\n v_*)v_*
\end{equation}
where $\mathcal{A}$ is the homogenized matrix of $H_0(\frac{\cdot}{\delta}){\rm Id}_{\R^2}$. (See Appendix \ref{S8.ProofOfResultMyrto} to have more details about $\mathcal{A}$)

\end{prop}

The proof of Proposition \ref{P8.asymptotiquedegnonzero} is postponed to Appendix \ref{S8.ProofOfResultMyrto}.

We apply the above proposition with $\O_0=\O\setminus\cup\overline{B(a_i,\eta)}$, $\delta=\delta_n\downarrow0$ the sequence which defines $a_{\v_n}$ and $H_n=U_{\v_n}^2$. By application of Proposition \ref{P8.UepsCloseToaeps}, we obtain 
\[\mathcal{T}_{\delta}(U^2_{\v_n})(x,y)\stackrel{L^2(\O_0\times \tilde{Y})}{\to}
\begin{cases}
\tilde{a}^2(y)&\text{if }\lambda\equiv1\\1&\text{if }\lambda\to0
\end{cases}.
\]
Note that the $\tilde{Y}$-periodic extension of $\tilde{a}$ in $\R^2$ is equal to the $Y$-periodic extension of $1-(1-b^2)\1_{\o^\lambda}$ which is $a^\lambda$ (defined Construction \ref{Constru.ConstrucPinnTerm}).

We find that $v_*$ solves
\[
\begin{array}{cl}
-\Div(\mathcal{A}\n v_*)=(\mathcal{A}\n v_*\cdot\n v_*)v_*&\text{if }\lambda\equiv1\\-\Delta v_*=|\n v_*|^2v_*&\text{if }\lambda\to0
\end{array}.
\]
Here $\mathcal{A}$ is the homogenized matrix of $[{a^\lambda}(\frac{\cdot}{\delta})]^2{\rm Id}_{\R^2}$.
\subsection{The small bad discs}
\subsubsection{Definition}
With the help the bound on the potential part of the minimizers $\frac{1}{\v^2}\int_\O(1-|v_\v|^2)^2\leq C$ (Proposition \ref{P8.BoundAnnulusArroundBadDiscs}), in the spirit of \cite{BBH} (Theorem III.3), we may detect the vorticity defects (the connected components of $\{|v_\v|\leq7/8\}$) by smaller discs (discs with radius of order $\v$) than the $\rho$-bad discs (Notation \ref{NOTATIONRHOBAD}).
\begin{notation}\label{TheSmallBadDiscTGV}The small bad discs

The construction is done as follows:
\begin{enumerate}[$\bullet$]
\item We consider a covering of $\O$ as in Notation \ref{NOTATIONEPS1/4Disc} (page \pageref{NOTATIONEPS1/4Disc}). We fix $\rho=\rho(\v)\downarrow0$ s.t.  Assumption \eqref{8.ConditionOnTheMicroRadius} holds. For sufficiently small $\v$, we denote 
\[
\RHO(\v)=\{B(x,R/2)\,|\,B(x,R)\text{ given by Notation \ref{NOTATIONRHOBAD}, page \pageref{NOTATIONRHOBAD}}\}.
\]
\item Following \cite{BBH} (Theorem III.3), for $l\geq2$, there are $\kappa_l,\mu_l>0$ (depending only on $\O,g$ and $l$) s.t. for $x\in\O$,  if
\[
\frac{1}{\v^2}\int_{B(x,2\kappa_l\v)}(1-|v_\v|^2)^2\leq\mu_l
\]
then
\[
|v_\v|\geq1-\frac{1}{l^2}\text{ in }B(x,\kappa_l\v).
\]
We fix $l\geq2$ and we drop the subscript $l$. We now consider a covering of $\cup_{B\in\RHO(\v)}B$ by discs $(B(x^\v_i,\kappa\v))_{i\in I}$ s.t
\[
x^\v_i\in\cup_{B\in\RHO(\v)}B,\,\forall\,i\in I,
\]
\[
B(x^\v_i,\kappa\v/4)\cap B(x^\v_j,\kappa\v/4)=\emptyset\text{ if }i\neq j,
\]
\[
\cup_{i\in I}B(x^\v_i,\kappa\v)\supset\cup_{B\in\RHO(\v)}B.
\]
We say that $B(x^\v_i,\kappa\v)$ is a {\it small good disc} if
\[
\frac{1}{\v^2}\int_{B(x^\v_i,2\kappa\v)}(1-|v_\v|^2)^2<\mu.
\]
\item If $B(x^\v_i,\kappa\v)$ is not a small good disc, then we call it a {\it small bad disc}. We denote $J\subset I$ the set of indices of small bad discs.
\end{enumerate}
\end{notation}
Following \cite{BBH}, using Proposition \ref{P8.BoundAnnulusArroundBadDiscs}, there is $N_l=N>0$ (depending only on $\O$, $g$ and $l$) s.t. ${\rm Card}(J)\leq N$.

Using a standard separation process (Lemma \ref{L8.SeparationLemma}),  for $\v_n\downarrow0$, possibly after passing to a subsequence and relabeling the discs, there are $J'\subset J$ and $\kappa'\in\{\kappa,...,9^{N-1}\kappa\}$ s.t.
\begin{equation}\nonumber
\{|v_{\v_n}|<1-1/l^2\}\subset\cup_{i\in J}B(x^{\v_n}_i,\kappa\v_n)\subset\cup_{i\in J'}B(x^{\v_n}_i,\kappa'\v_n)
\end{equation}
and
\[
\displaystyle\frac{|x^{\v_n}_i-x^{\v_n}_j|}{\v_n}\geq 8\kappa'\text{ if }i,j\in J',\,i\neq j.
\]
By a standard iterative procedure, we may assume that the small bad discs are mutually far away in the $\v$-scale.
\begin{prop}\label{P8.BBHBadBallsProperty} Possibly after passing to a subsequence, we have, for large $R$ and $J''\subset J'$, 
\begin{equation}\nonumber
\{|v_{\v_n}|<1-{1}/{l^2}\}\subset\cup_{i\in J''} B(x_{i}^{\v_n},R\v_{n}),
\end{equation}
where, for $i\neq j$, 
\begin{equation}\nonumber
\frac{|x_{i}^{\v_n}-x_{j}^{\v_n}|}{\v_{n}}\to\infty\text{ as }n\to\infty.
\end{equation}
\end{prop}
\begin{notation}\label{Not.DefOfSepSmallBadDi}The small and separated bad discs

The discs $\{B(x_i^{\v_n},R\v_n)\,|\,i\in J''\}$ obtained in Proposition \ref{P8.BBHBadBallsProperty} are the small and separated bad discs.
\end{notation}
\subsubsection[There are exactly $d$ small bad discs]{Each $\rho$-bad disc contains exactly one small bad disc}

By construction, we know that the small and separated bad discs (defined Notation \ref{Not.DefOfSepSmallBadDi}) are covered by the $\rho$-bad discs defined Notation \ref{NOTATIONRHOBAD} (page \pageref{NOTATIONRHOBAD}). We next prove that there are exactly $d$ small bad discs and consequently, there is exactly one small bad discs per $\rho$-bad discs.
\begin{prop}\label{P8.BBHBadBallsProperty2}
For large $n$ and for all $i\in J''$, we have 
\[
\deg_{\p B(x^{\v_n}_{i},R\v_{n})}(v_{\v_{n}})=1.
\]
\end{prop}
\begin{proof}

First we prove that, for large $n$  and for all $i$, we have 
\[
\deg_{\p B(x_{i}^{\v_n},R\v_{n})}(v_{\v_n})\neq0.
\]
We argue by contradiction and we assume that, up to  a subsequence, there is $i$ s.t. $\deg_{\p B(x_{i}^{\v_n},R\v_{n})}(v_{\v_n})=0$.

Set 
\begin{equation}\label{EqNameMNNN}
M_n=\min\left\{b\displaystyle\min_{i\neq j}\frac{|x_i^{\v_n}-x_j^{\v_n}|}{8R\v_{n}},\delta^{-1}\right\}
\end{equation}
and set
\[
\begin{array}{cccc}u'_n:&B(0,M_n)&\to&\C\\&x&\displaystyle\mapsto&\displaystyle\frac{u_{\v_n}\left(\displaystyle\frac{\v_{n}}{b}x+x_i^{\v_n}\right)}{b}\end{array}.
\]
 Note that, $B(x_{i}^{\v_n},M_n\v_{n})\subset\o_{\v_n}$ and by Proposition \ref{P8.BBHBadBallsProperty}, we have $M_n\to\infty$.

It is easy to check that $u'_n$ solves $-\Delta u'_n=u'_n(1-|u'_n|^2)$. Following \cite{BMR1}, up to a subsequence, 
\begin{equation}\label{8.BlowupConvergence}
\text{$u'_n\to u_0$ in $C^2_{\rm loc}(\R^2)$;}
\end{equation}
here $u_0:\R^2\to\C$ solves $-\Delta u_0=u_0(1-|u_0|^2)$ in $\R^2$.

Then two cases occur: $\displaystyle\int_{\R^2}(1-|u_0|^2)^2<\infty$ or $\displaystyle\int_{\R^2}(1-|u_0|^2)^2=\infty$.

Assume first that $\displaystyle\int_{\R^2}(1-|u_0|^2)^2<\infty$. From \cite{BMR1}, noting that the degree of $u_0$ on large circles centered in $0$ is $0$, we obtain that $u_0={\rm Cst}\in\S^1$ and consequently $\displaystyle\int_{\R^2}(1-|u_0|^2)^2=0$.

Since $u'_n\to u_0$ in $L^4(B(0,2bR))$ ($R\geq\kappa$), we find that
\begin{eqnarray*}
\int_{B(0,2bR)}(1-|u_n'|^2)^2&=&\frac{b^2}{\v^2_{n}}\int_{B(x_i^{\v_n},2R\v_{n})}(1-|u_n/b|^2)^2\\&=&\frac{b^2}{\v^2_{n}}\int_{B(x_i^{\v_n},2R\v_{n})}(1-|v_{\v_n}|^2)^2+o_n(1)\to0.
\end{eqnarray*}
Noting that $B(x_i^{\v_n},\kappa\v_{n})$ is a small bad disc and that $B(x_i^{\v_n},2\kappa\v_{n})\subset B(x_i^{\v_n},2R\v_{n})$, we have a contradiction.

Therefore $\displaystyle\int_{\R^2}(1-|u_0|^2)^2=\infty$. Consequently, there is $M_0>0$ s.t.
\[
\int_{B(0,bM_0)}(1-|u_0|^2)^2\geq\sup_n\left\{\frac{4b^2}{\v_{n}^2}\int_\O(1-|v_{\v_n}|^2)^2\right\}.
\]
Thus, for large $n$  we have
\begin{eqnarray*}
\int_{B(0,bM_n)}(1-|u_n'|^2)^2&=&\frac{b^2}{\v^2_{n}}\int_{B(x_i^{\v_n},M_n\v_{n})}(1-|u_{\v_n}/b|^2)^2
\\&=&\frac{b^2}{\v^2_{n}}\int_{B(x_i^{\v_n},M_n\v_{n})}(1-|v_{\v_n}|^2)^2+o_n(1)\\&\geq& \sup_n\left\{\frac{2b^2}{\v_n^2}\int_\O(1-|v_{\v_n}|^2)^2\right\},
\end{eqnarray*}
which is a contradiction with $B(x_i^{\v_n},M_n\v_n)\subset\O$.

Consequently we obtain that for large  $n$, $\deg_{\p B(x_{i}^{\v_n},R\v_{n})}(v_{\v_n})\neq0$.

Now we prove that
\begin{equation}\label{8.ToutLesDegDesBBHSontEgauxA1}
\deg_{\p B(x_{i}^{\v_n},R\v_{n})}(v_{\v_n})=1\text{ for all $i$ and  large $n$}.
\end{equation}
Note that each small bad disc contains at least a zero of $v_{\v_n}$. Consequently, for $\rho$ satisfying \eqref{8.ConditionOnTheMicroRadius}, all small bad discs are included in a $\rho$-bad disc $B(y,\rho)$ defined Notation \ref{NOTATIONRHOBAD} (page \pageref{NOTATIONRHOBAD}). (For sake of simplicity we wrote $B(y,\rho)$ instead of $B(y,2\kappa\rho)$).

If $B(y,\rho)$ is a $\rho$-bad disc, we denote $\Lambda_y=\{i\in J''\,|\,x_i^{\v_n}\in B(y,\rho)\}$. Clearly, if ${\rm Card}(\Lambda_y)=1$, then \eqref{8.ToutLesDegDesBBHSontEgauxA1} holds. 

We define 
\[
\textsf{\ae}^y_n:=\begin{cases}10^{-2}\min_{i,j\in \Lambda_y,\,i\neq j}|x_i^{\v_n}-x_j^{\v_n}|&\text{if }{\rm Card}(\Lambda_y)>1\\R\v_n&\text{otherwise}\end{cases}.
\]
From Proposition \ref{P8.BBHBadBallsProperty}, if ${\rm Card}(\Lambda_y)>1$ then $\textsf{\ae}_n/\v_n\to\infty$.

For simplicity, we assume that $y=0$ and we let
\[
\tilde{B}=B(0,8)\setminus\cup_{i\in\Lambda_0}\overline{B\left(\frac{x^{\v_n}_i}{\rho}, \frac{\textsf{\ae}^0_n}{\rho}\right)}.
\]
\begin{remark}
Note that from Corollary \ref{C8.TheCenterAreInGoodPosition} we have $B(y,16\rho)\subset\o_\v$.
\end{remark}
Clearly, we are in position to apply Theorem 2 in \cite{HS95} in the perforated domain $\tilde{B}$. After scaling, we find that
\[
\frac{1}{2}\int_{ B(y,8\rho)\setminus\cup \overline{B(x^{\v_n}_i, \textsf{\ae}^y_n)}}|\n v_{\v_n}|^2\geq\pi \left|\sum_{i\in\Lambda_y}\deg_{\p B(x_{i}^{\v_n},R\v_{n})}(v_{\v_n})\right|\ln\frac{\rho}{\textsf{\ae}^y_n}-C=\pi \ln\frac{\rho}{\textsf{\ae}^y_n}-C.
\]
In order to prove \eqref{8.ToutLesDegDesBBHSontEgauxA1}, we observe the case where there is  $y$ s.t. ${\rm Card}(\Lambda_y)>1$. Recall that if for all $y$ centers of $\rho$-bad discs we have ${\rm Card}(\Lambda_y)=1$, then \eqref{8.ToutLesDegDesBBHSontEgauxA1} holds. Since $\deg_{\p B(x_i^{\v_n},R\v_n)}(v_{\v_n})\neq0$, if ${\rm Card}(\Lambda_y)>1$, then we have
\[
\sum_{i\in\Lambda_y}|\deg_{\p B(x_{i}^{\v_n},R\v_{n})}(v_{\v_n})|>1.
\] 

We obtain easily the following lower bound for $i\in\Lambda_y$:
\[
\frac{1}{2}\int_{B(x^{\v_n}_i, \textsf{\ae}^y_n)\setminus \overline{B(x^{\v_n}_i,R\v_n)}}|\n v_{\v_n}|^2\geq\pi \left|\deg_{\p B(x_{i}^{\v_n},R\v_{n})}(v_{\v_n})\right|\ln\frac{\textsf{\ae}^y_n}{R\v_n}-C.
\]
Summing for $i\in\Lambda_y$, we obtain that
\[
\sum_{i\in\Lambda_y}\frac{1}{2}\int_{B(x^{\v_n}_i, \textsf{\ae}^y_n)\setminus \overline{B(x^{\v_n}_i,R\v_n)}}|\n v_{\v_n}|^2\geq2\pi \ln\frac{\textsf{\ae}^y_n}{R\v_n}-C.
\]
Consequently, we deduce that
\[
\sum_y\frac{1}{2}\int_{ B(y,8\rho)\setminus\cup \overline{B(x^{\v_n}_i, R\v_n)}}|\n v_{\v_n}|^2\geq \pi d\ln\frac{\rho}{R\v_n}+\pi\sum_{y\text{ s.t. }{\rm Card}(\Lambda_y)>1}\ln\frac{\textsf{\ae}^y_n}{R\v_n}-\mathcal{O}_n(1).
\]
From Lemma \ref{L8.UpperboundAuxPb} and Propositions \ref{P8.BadDiscSeparation} $\&$ \ref{P8.StudyOfS1Part}, we deduce easily 
\[
\frac{1}{2}\int_{\bigcup B(y,8\rho)\setminus\cup \overline{B(x^{\v_n}_i, R\v_{n})}}U_{\v_n}^2|\n v_{\v_n}|^2=\pi db^2\ln\frac{\rho}{\v_{n}}+\mathcal{O}_n(1).
\]
Combining the previous estimates, we obtain that
\[
\left\{y\text{ center of $\rho$-bad discs}\,|\,{\rm Card}(\Lambda_y)>1\right\}=\emptyset,
\]
and thus $\deg_{\p B(x_{i}^{\v_n},R\v_{n})}(v_{\v_n})=1$ for large $n$.
\end{proof}
\begin{cor}\label{C8.AuniqueZeroInsideEachBadBalls}
For large $n$, there is a unique zero inside each small and separated bad discs defined Notation \ref{Not.DefOfSepSmallBadDi} (page \pageref{Not.DefOfSepSmallBadDi}).
\end{cor}
\begin{proof}
From Proposition \ref{P8.BBHBadBallsProperty2}, one may assume that $v_{\v_{n}}(x_{i}^{\v_n})=0$.

 Let $i\in\{1,...,d\}$. In view of \eqref{8.BlowupConvergence}, if we denote
\begin{equation}\label{8.BlowUpBBHBalls}
\begin{array}{cccc}u'_n:&B(0,M_{n})&\to&\C\\&x&\displaystyle\mapsto&\displaystyle\frac{u_{\v_{n}}(\displaystyle\frac{\v_{n}}{b}x+x_i^{\v_n})}{b}\end{array},
\end{equation}
then, up to a subsequence, 
\begin{equation}\label{Eq.ConvMicrcrocro}
u'_n\to u_0\text{ in }C^1(\overline{B(0,bR)}).
\end{equation}
Here $M_n$ is defined in \eqref{EqNameMNNN}.

Using the main result of \cite{M2}, we have the existence of a universal function $f:\R^+\to[0,1]$ s.t.
\begin{equation}\label{8.SpecialRadialSolution}
\left\{\begin{array}{c}\text{$u_0(x)=f(|x|){\rm e}^{\imath(\theta+\theta_i)}$ where $x=|x|{\rm e}^{\imath\theta}$, $\theta_i\in\R$}\\\text{ and $f:\R^+\to\R^+$ is increasing.}\end{array}\right.
\end{equation}

Therefore, we may apply Theorem 2.3 in \cite{BCP1} in order to obtain that, for large $n$, $u'_n$ has a unique zero in $B(0,bR)$. Consequently, for large $n$, $v_{\v_{n}}$ has a unique zero in $B(x_i^{\v_n},R\v_{n})$.
\end{proof}
\begin{cor}\label{C8.TheRadiusCanBeConsideredIndeOnSubsequ}
One may consider that $R$ depends only on $l$ ($R$ is independent of the extraction we consider), \emph{i.e}, for  $l\geq2$ there is $R_l>0$ s.t. for small $\v$, denoting $\{x_{i}^{\v}\,|\,i\in \{1,...,d\}\}$ the set of zeros of a minimizer $v_\v$, we have  
\[
\{|v_{\v}|<1-1/l^2\}\subset\cup_i B(x_{i}^{\v},R_l\v).
\]
\end{cor}
\begin{proof}
From Corollary \ref{C8.AuniqueZeroInsideEachBadBalls}, one may assume that $v_{\v_{n}}(x_{i}^{\v_n})=0$.

Let $f:\R^+\to\R^+$ be defined as in \eqref{8.SpecialRadialSolution} and $u_n'$ as in \eqref{8.BlowUpBBHBalls}. For $l\geq2$, consider $R_l>0$ s.t. 
\[
l\mapsto R_l \text{ is increasing and }f(bR_l)\geq1-\frac{1}{2l^2}.
\]
Note that from \cite{Sha1}, one may consider $R_l\simeq\sqrt2 l/b$.

By uniqueness of $f$, the full sequence $|u_n'|$ converges to $f$ in $L^\infty\left[B(0,b\max\left\{R,R_l\right\})\right]$. Consequently, for $n$ sufficiently large, since $f$ is not decreasing,
\[
\{|v_{\v_{n}}|<1-1/l^2\}\subset\cup_i B(x_{i}^{\v_n},R_l\v_{n}).
\]
\end{proof}

\subsection{Asymptotic expansion of $F_\v(v_\v)$}
This section is essentially devoted to proof Theorem \ref{T8.MainThm3}. The key argument in this proof is Proposition \ref{P8.ExactExpanding}.
\subsubsection{Statement of the main result and a corollary}
We state a technical and fundamental result and a direct corollary.
\begin{prop}\label{P8.ExactExpanding}
For all $\v_n\downarrow0$, up to a subsequence, there is $\rho=\rho(\v_n)$ s.t. $\v_n\ll\rho\ll\lambda\delta^{3/2}$ and s.t. when $n\to\infty$ the following holds
\begin{equation}\label{8.ExactExpanding}
F_{\v_n}(v_{\v_n})\geq J_{\rho,\v_n}+ db^2(\pi\ln\frac{b\rho}{\v_n}+\gamma)+o_n(1),
\end{equation}
where $J_{\rho,\v}$ is defined in \eqref{8.StatementSecondAuxProblemDir} 
and $\gamma$ is the universal constant defined in \cite{BBH}, Lemma IX.1. 
\end{prop}

\begin{cor}\label{C8.SeparationDansLeProblemeAuxiliaire}
Let $\v_n\downarrow0,\rho$ be as in Proposition \ref{P8.ExactExpanding}. Then we have
\[
J_{\v_n,\v_n}-J_{\rho,\v_n}=\pi db^2\ln \frac{\rho}{\v_n}+o_n(1).
\] 
\end{cor}
\begin{proof}[Proof of Corollary \ref{C8.SeparationDansLeProblemeAuxiliaire}]


Using Proposition \ref{P8.MainAuxPb:DirVsDeg}, we may consider ${\bf x}_n=(x^n_1,...,x^n_d)\in\O^d$  a minimal configuration of points for $J_{\rho,\v_n}$, \emph{i.e.} s.t.
\[
\hat{\J}_{\rho,\v_n}({\bf x}_n,{\bf 1})= J_{\rho,\v_n}.
\]
Combining Corollaries \ref{C8.CorolSecondAuxPb} $\&$ \ref{C8.AnAlmostMinConfigIsAnAlmostMinConf} (or Proposition \ref{P8.CorolSecondAuxPb} if the pinning term is not periodic), we have the existence of $c>0$ s.t. $B(x^n_i,c\lambda\delta)\subset\o_\v$. 

Therefore, for a minimal map $w_n$ of $\hat{\J}_{\rho,\v_n}({\bf x}_n,{\bf 1})$, we may easily construct a map $\tilde{w}_n\in H^1(\O\setminus\cup_i \overline{B(x_i,\v_n)},\S^1)$ s.t. $\tilde{w}_n\in{\J}_{\v_n}({\bf x}_n,{\bf 1})$ and
\begin{eqnarray}\nonumber
J_{\v_n,\v_n}&\leq& \frac{1}{2}\int_{\O\setminus\cup \overline{B(x_i,\v_n)}}U_{\v_n}^2|\n \tilde{w}_n|^2
\\\nonumber&=&\frac{1}{2}\int_{\O\setminus\cup \overline{B(x_i,\rho)}}U_{\v_n}^2|\n {w}_n|^2+\frac{1}{2}\int_{\cup B(x_i,\rho)\setminus\overline{B(x_i,\v_n)}}U_{\v_n}^2|\n \tilde{w}_n|^2
\\\label{8.PartCroitij1}&=& J_{\rho,\v_n}+db^2\pi\ln\frac{\rho}{\v_n}+o_n(1).
\end{eqnarray}
On the other hand, Lemma \ref{L8.UpperboundAuxPb} combined with Proposition \ref{P8.ExactExpanding} yield
 \begin{equation}\label{8.sopdfgisdgihj}
 J_{\rho,\v_n}+ db^2(\pi\ln\frac{b\rho}{\v_n}+\gamma)+o_n(1)\leq F_{\v_n}(v_{\v_n})\leq J_{\v_n,\v_n}+db^2(\pi\ln b+\gamma).
 \end{equation}
 We conclude with the help of \eqref{8.PartCroitij1} and \eqref{8.sopdfgisdgihj}.
\end{proof}
\subsubsection{Proof of Theorem \ref{T8.MainThm3}}\label{ProofOfEnExpenenehehe}
We are now in position to prove Theorem \ref{T8.MainThm3}, \emph{i.e.}, we are going to prove that
\begin{equation}\nonumber
F_\v(v_\v)= J_{\v,\v}+ db^2(\pi\ln b+\gamma)+o_\v(1).
\end{equation}
Indeed, using Lemma \ref{L8.UpperboundAuxPb}, it suffices to prove that
\[
F_\v(v_\v)\geq J_{\v,\v}+ db^2(\pi\ln b+\gamma)+o_\v(1).
\]
This estimate is equivalent to: 
\[
\text{for all $\v_n\downarrow0$, up to subsequence, we have }F_{\v_n}(v_{\v_n})\geq J_{\v_n,\v_n}+ db^2(\pi\ln b+\gamma)+o_n(1).
\]
Let $\v_n\downarrow0$. Then, up to a subsequence, there is $\rho=\rho_n$ given by Proposition \ref{P8.ExactExpanding} s.t.
\[
F_{\v_n}(v_{\v_n})\geq J_{\rho,\v_n}+ db^2(\pi\ln\frac{b\rho}{\v_n}+\gamma)+o_n(1).
\]
We deduce from Corollary \ref{C8.SeparationDansLeProblemeAuxiliaire} that 
\begin{eqnarray*}
F_{\v_n}(v_{\v_n})&\geq &J_{\v_n,\v_n}-db^2\ln \frac{\rho}{\v_n}+ db^2(\pi\ln\frac{b\rho}{\v_n}+\gamma)+o_n(1)
\\&=&J_{\v_n,\v_n}+ db^2(\pi\ln b+\gamma)+o_n(1),
\end{eqnarray*}
which ends the proof of Theorem \ref{T8.MainThm3}.
\subsubsection{Proof of Proposition \ref{P8.ExactExpanding}}
 In order to construct $\rho$, we first define a suitable extraction.

For $l\in\N\setminus\{0,1\}$, consider $R_l$ given by Corollary \ref{C8.TheRadiusCanBeConsideredIndeOnSubsequ}.

Using Proposition \ref{P8.BBHBadBallsProperty2} and Corollary \ref{C8.AuniqueZeroInsideEachBadBalls}, for sufficiently large $n$, $v_{\v_n}$ has exactly $d$ zeros $x^n_1=x_1,...,x^n_d=x_d$. 

Clearly, these zeros are well separated and far from $\p\O$ (independently of $n$). 

Fix $i\in\{1,...,d\}$ and consider 
\[
\begin{array}{cccc}u'_n:&B(0,\lambda^2\delta^2/\v_n)&\to&\C\\&x&\displaystyle\mapsto&\displaystyle\frac{u_{\v_n}(\displaystyle\frac{\v_{n}}{b}x+x_i)}{b}\end{array}.
\]
For simplicity, assume $x_i=0$.

Up to a subsequence, one has, as in \eqref{8.SpecialRadialSolution},
\[
u_n'\to u_0\text{ in }C^2_{\rm loc}(\R^2,\C),\,u_0(x)=f(|x|){\rm e}^{\imath(\theta+\theta_i)}
\]
where $x=|x|{\rm e}^{\imath\theta}$, $\theta_i\in\R$ and $f:\R^+\to\R^+$ is increasing.

Consequently, for $l\in\N\setminus\{0,1\}$, one may construct an extraction $(n_l)_{l\geq2}$ s.t., denoting $u'_{n_l}=u'_l=|u'_l|{\rm e}^{\imath (\theta+\phi'_l)}$ and $v_{\v_{n_l}}=v_l$, we have
\begin{equation}\label{8.OutsidePseudoBad}
\text{ $\{|v_l|<1-1/l^2\}\subset\cup_iB(x_i,R_l\v_{n_l})$},
\end{equation}
\begin{equation}\nonumber
\rho_l:=R_l\v_{n_l}\leq\frac{\lambda^2\delta^2}{l},
\end{equation}
\begin{equation}\label{8.EstGlobalBBHBallExtract}
\left|\int_{B(0,bR_l)}{\left|\n u_l'\right|^2+\frac{1}{2}\left(1-\left|u_l'\right|^2\right)^2}-\int_{B(0,bR_l)}{\left|\n u_0\right|^2+\frac{1}{2}\left(1-\left|u_0\right|^2\right)^2}\right|\leq\frac{1}{l},
\end{equation}
and
\begin{equation}\label{8.EstPhaseBBHBallExtract}
\|\phi'_l-\theta_i\|_{C^1(\overline{B(0,bR_l)})}\leq \frac{1}{l}.
\end{equation}
Here $R_l\simeq\sqrt2 l/b$ and is defined in Corollary \ref{C8.TheRadiusCanBeConsideredIndeOnSubsequ}.

Following the proof of Proposition 1, Step 2 in \cite{BMR1}, one has
\begin{equation}\label{8.NonPhaseBoundedscal}
\int_{B(0,\frac{\lambda^2\delta^2}{\v_{n_l}})\setminus \overline{B(0,R_l)}}|\n \phi'_l|^2\leq C\text{ independently of }l.
\end{equation} 

In $B(0,\lambda^2\delta^2)\setminus\overline{B(0,\v_{n_l})}$, we denote $v_{n_l}=v_l=|v_l|{\rm e}^{\imath (\theta+\phi_l)}$ (${\rm e}^{\imath\theta}=x/|x|$). By conformal invariance, \eqref{8.EstPhaseBBHBallExtract} implies that
\begin{equation}\label{8.EstPhaseBBHBallExtractNonSacal}
\|\phi_l-\theta_i\|_{L^\infty(\p B(0,\rho_l))}+|\phi_l|_{H^{1/2}(\p B(0,\rho_l))}\leq \frac{C}{l}.
\end{equation}

Denote $W_l=B(0,2\rho_l)\setminus \overline{B(0,\rho_l)}$ and consider $\psi_{i}^l\in H^{1/2}(\p W_l,\R)$ s.t.
\[
\psi_{i}^l=\begin{cases}\phi_l-\theta_i&\text{on }\p B(0,\rho_l)\\0&\text{on }\p B(0,2\rho_l)\end{cases}.
\]
Using \eqref{8.EstPhaseBBHBallExtractNonSacal}, it is clear that $\|\psi_i^l\|_{L^\infty(\p W_l)}+|\psi_{i}^l|_{H^{1/2}(\p W_l)}=\mathcal{O}(1/l)$. From this, it is straightforward that there exists a constant 
$C_0>0$ (independent of $l$) and $\Psi_i^l\in H^1(W_l,\R)$ s.t.
\[
\text{$\tr_{\p W_l}\Psi_i^l=\psi_i^l$ and  }\frac{1}{2}\int_{W_l}|\n \Psi_i^l|^2\leq \frac{C_0}{l^2}.
\] 
Finally we define $\Psi_l\in H^1(\O\setminus\cup\overline{B(x_i,\rho_l)},\R)$ by
\[
\Psi_l=\begin{cases}\Psi_i^l(\cdot-x_i)&\text{in }x_i+W_l\\0&\text{otherwise }\end{cases}
\]
and
\[
\tilde{w}_l=\frac{v_l}{|v_l|}{\rm e}^{-\imath \Psi_l}\in\J_{\rho_l}({\bf x},{\bf 1})\text{ with }{\bf x}=(x_1,...,x_d).
\]
Therefore, denoting $w_l=\frac{v_l}{|v_l|}={\rm e}^{\imath (\theta+\phi_l)}$, $U_l=U_{\v_{n_l}}$ and $\O_{\rho_l}=\O\setminus\overline{B(x_i,\rho_l)}$, we have
\[
\hat{\J}_{\rho_l,\v_{n_l}}({\bf x},{\bf 1})\leq\frac{1}{2}\int_{\O_{\rho_l}}{U_l^2|\n\tilde{w}_l|^2}
=\frac{1}{2}\int_{\O_{\rho_l}}{U_l^2|\n w_l|^2+2U_l^2\n(\theta+\phi_l)\cdot\n\Psi_l}+o_l(1).
\]
From \eqref{8.NonPhaseBoundedscal}, we obtain easily that
\[
\left|\int_{\O_{\rho_l}}{\n(\theta+\phi_l)\cdot\n\Psi_l}\right|=\sum_i\left|\int_{x_i+W_l}{\n(\theta+\phi_l)\cdot\n\Psi_i^l(\cdot-x_i)}\right|=o_l(1)
\]
and consequently
\begin{equation}\label{8.BigStep}
\hat{\J}_{\rho_l,\v_{n_l}}({\bf x},{\bf 1})\leq\frac{1}{2}\int_{\O_{\rho_l}}{U_l^2|\n w_l|^2}+o_l(1).
\end{equation} 
On the other hand, from direct computations, one has
\[
\frac{1}{2}\int_{\O_{\rho_l}}{U_l^2|\n v_l|^2}\geq\frac{1}{2}\int_{\O_{\rho_l}}{U_l^2|\n w_l|^2}+\frac{1}{2}\int_{\O_{\rho_l}}{U_l^2(|v_l|^2-1)|\n (\theta+\phi_l)|^2}.
\]
Using the same argument as Mironescu in \cite{M2}, one may obtain that
\begin{equation}\label{8.EstiamtionAlAlalaMironescu}
\frac{1}{2}\int_{\O_{\rho_l}}{(1-|v_l|^2)^{1/2}|\n (\theta+\phi_l)|^2}\leq C\text{ with $C$ independent of $l$}.
\end{equation}
From \eqref{8.EstiamtionAlAlalaMironescu} and \eqref{8.OutsidePseudoBad}, we obtain
\begin{equation*}
\frac{1}{2}\int_{\O_{\rho_l}}{U_l^2|\n v_l|^2}\geq\frac{1}{2}\int_{\O_{\rho_l}}{U_l^2|\n w_l|^2}-o_l(1).
\end{equation*}
Therefore, with \eqref{8.BigStep},
\begin{equation}\label{8.BigStep2}
F_{\v_{n_l}}(v_l,\O_{\rho_l})+o_l(1)\geq\frac{1}{2}\int_{\O_{\rho_l}}{U_l^2|\n v_l|^2}+o_l(1)\geq\hat{\J}_{\rho_l,\v_{n_l}}({\bf x},{\bf 1}) .
\end{equation} 
In order to complete the proof of \eqref{8.ExactExpanding}, it suffices to estimate the contribution of the discs $B(x_i,\rho_l)$.

One has (using \eqref{8.EstGlobalBBHBallExtract})
\begin{eqnarray*}
F_{\v_{n_l}}(v_l,B(x_i,\rho_l))&=&\frac{b^2}{2}\int_{B(0,\rho_l)}{\left|\n \left(\frac{u_l}{b}\right)\right|^2+\frac{b^2}{2\v_{n_l}^2}\left(1-\left|\frac{u_l}{b}\right|^2\right)^2}+o_l(1)
\\&=&\frac{b^2}{2}\int_{B(0,bR_l)}{\left|\n u_l'\right|^2+\frac{1}{2}\left(1-\left|u_l'\right|^2\right)^2}+o_l(1)
\\&=&\frac{b^2}{2}\int_{B(0,bR_l)}{\left|\n u_0\right|^2+\frac{1}{2}\left(1-\left|u_0\right|^2\right)^2}+o_l(1).
\end{eqnarray*}
From Proposition 3.11 in \cite{SS1}, one has
\[
\frac{1}{2}\int_{B(0,bR_l)}{\left|\n u_0\right|^2+\frac{1}{2}\left(1-\left|u_0\right|^2\right)^2}=\pi\ln (bR_l)+\gamma+o_l(1),
\]
hence
\begin{equation}\label{8.BigStep3}
F_{\v_{n_l}}(v_l,B(x_i,\rho_l))=b^2[\pi\ln (bR_l)+\gamma]+o_l(1).
\end{equation}
By combining \eqref{8.BigStep2} with \eqref{8.BigStep3}, we obtain \eqref{8.ExactExpanding} with $\rho_l=R_l\v_{n_l}$.
\subsection{Proof of Theorems \ref{THMMAIN}, \ref{T8.MainThm1}, \ref{T8.MainThm2} and \ref{T8.Localisationdanslesinclusions}}
We prove  Quantization part of Theorem \ref{THMMAIN} and Theorem \ref{T8.MainThm1}: 
\begin{enumerate}[$\bullet$]
\item The existence of exactly $d$ zeros is a direct consequence of Corollary \ref{C8.AuniqueZeroInsideEachBadBalls}.
\item The facts that they are well included in $\o_\v$, well separated and that $v_\v$ has a degree equal to $1$ on small circles around the zeros are obtained by Proposition \ref{P8.BBHBadBallsProperty2} and Corollary \ref{C8.AuniqueZeroInsideEachBadBalls}.
\item The lower bound for $|v_\v|$ is given by Proposition \ref{P8.BadDiscSeparation}.4.
\end{enumerate}
We prove Macroscopic location part of Theorem \ref{THMMAIN}:
\begin{enumerate}[$\bullet$]
\item The macroscopic location part of Theorem \ref{THMMAIN} is a direct consequence of Theorem \ref{T8.MainThm3} (proved Subsection \ref{ProofOfEnExpenenehehe}), Proposition \ref{P8.BadDiscSeparation}, \eqref{8.BigStep3} and Proposition \ref{P.RenormalizedBBHEnergy}.

Indeed, from Theorem \ref{T8.MainThm3}, Proposition \ref{P8.BadDiscSeparation}.4 and \eqref{8.BigStep3}, we get that the zeros form a quasi-minimizer of $J_{\rho,\v}$ (defined Notation \ref{NOTATIONQUASIIN}, page \pageref{NOTATIONQUASIIN}). By using Proposition \ref{P.RenormalizedBBHEnergy} we deduce that they are a quasi-minimizer of the renormalized energy $W_g$ (defined Notation \ref{NOTATIONQUASIIN}). Thus, by smoothness of $W_g$, the zeros tend to a minimal configuration of $W_g$.
\end{enumerate}
We prove Microscopic location part of Theorem  \ref{THMMAIN} $\&$ \ref{T8.Localisationdanslesinclusions}:
\begin{enumerate}[$\bullet$]
\item In the case where $\lambda\to0$, the fact that we may localize the zeros inside the inclusions (microscopic location part of Theorem \ref{THMMAIN} and Theorem \ref{T8.Localisationdanslesinclusions}) is obtained via Theorem 4 in \cite{publi3}. 

Indeed we take $f_n(x)=\tr_{\p B((k_n,l_n),\delta/2)}v_{\v_n}\left((k_n,l_n)+\delta x\right)$ with $(k_n,l_n)$ a center of a cell containing a zero of $v_{\v_n}$. Using the main result of \cite{M1}, one may easily prove that $f_n$ satisfies the conditions ({A1}) and ({A2}) in \cite{publi3}. Thus we can apply Theorem 4 in \cite{publi3} and infer that the location of the zero inside the inclusion is governed by a renormalized energy which is independent of the boundary condition.
\end{enumerate}
Theorem  \ref{T8.MainThm2} is obtained by combining:
\begin{enumerate}[$\bullet$]
\item The weak $H^1$-convergence of $v_{\v_n}$ to $v_*$ is a direct  consequence of Proposition \ref{P7.LocalizationOfTheEnergyPart1}. The limiting equation for $v_*$ is a direct consequence of Proposition \ref{P8.asymptotiquedegnonzero} (this is explained right after  Proposition \ref{P8.asymptotiquedegnonzero}).
\item The behavior in an $\v$-neighborhood of the zeros of $v_{\v_n}$ is given by \eqref{8.BlowUpBBHBalls}, \eqref{Eq.ConvMicrcrocro} and \eqref{8.SpecialRadialSolution} (noting that in \eqref{Eq.ConvMicrcrocro} we have $R=R_l\to+\infty$ as $l\to\infty$).
\end{enumerate}

\appendix
\section{Proof of Proposition \ref{P8.ExistenceOfminInIJ}}\label{S8.ExistenceMiniMizingMap}
We prove the existence of minimal map in $\I_\rho$ and in $\J_\rho$. The main ingredient is the fact that these sets are closed under $H^1$-weak convergence (see \cite{TheseLassoued} or below). Thus, considering a minimizing sequence for $\displaystyle\frac{1}{2}\int_{\O_\rho}\alpha|\n \cdot|^2$ in above sets, we obtained the result.

We consider 
\begin{enumerate}[$\bullet$]
\item $\theta_i:\O_\rho\to\R$ the main argument of $x-x_i$, \emph{i.e.} ${\rm e}^{\imath\theta_i}=\frac{x-x_i}{|x-x_i|}$. Note that the $\theta_i$ are multivalued function with smooth gradient.  
\item For $d_i\in\N^*$ (given by the definition of $\I_\rho$ or $\J_\rho$) we let  $\theta_0=\sum d_i\theta_i$ and thus ${\rm e}^{\imath\theta_0}=\Pi_i\left(\frac{x-x_i}{|x-x_i|}\right)^{d_i}$.
\end{enumerate}
From Lemma 11 in \cite{glcoursep6}, there is $\phi_0\in C^\infty(\p\O,\R)$ s.t. $g{\rm e}^{-\imath\theta_0}={\rm e}^{\imath\phi_0}$.

Note that 
\begin{equation}\label{8.ConditionForPhaseForIRho}
w\in\I_\rho\Longleftrightarrow w={\rm e }^{\imath (\theta_0+\phi)}\text{ with }\phi\in H^1(\O_\rho,\R)\text{ and }\tr_{\p\O}\phi=\phi_0, 
\end{equation}
\begin{equation}\label{8.ConditionForPhaseForJRho}
w\in\J_\rho\Longleftrightarrow \left\{\begin{array}{c}w={\rm e }^{\imath (\theta_0+\phi)}\text{ with }\phi\in H^1(\O_\rho,\R),\\\displaystyle\sum_{j\neq i}d_j\theta_j+\phi={\rm Cst}_i\text{ on }\p B(x_i,\rho)\text{ and }\tr_{\p\O}\phi=\phi_0\end{array}\right..
\end{equation}

Clearly, from \eqref{8.ConditionForPhaseForIRho} and \eqref{8.ConditionForPhaseForJRho}, $\I_\rho$ and $\J_\rho$ are $H^1$-weakly closed.

We now prove the second part of Proposition \ref{P8.ExistenceOfminInIJ}.

One may easily obtain that for some $\lambda: \O_\rho\to\R$, denoting $w={\rm e}^{\imath(\theta_0+\phi)}$, $\phi\in H^1(\O_\rho,\R)$ (and thus $w\in\I_\rho$), we have 
\begin{equation}\label{8.CaracterisationEqS1Valued}
-\Div(\alpha\n w)=\lambda w\Longleftrightarrow\left\{-\Div\left[\alpha\n(\theta_0+\phi)\right]=0\text{ and }\lambda=\alpha|\n w|^2\right\}.
\end{equation}
This observation is a direct consequence of the following identity
\[
-\Div\left[\alpha\n{\rm e}^{\imath(\theta_0+\phi)}\right]=-\Div\left[\alpha\n(\theta_0+\phi)\right]\imath {\rm e}^{\imath(\theta_0+\phi)}+\alpha|\n(\theta_0+\phi)|^2{\rm e}^{\imath(\theta_0+\phi)}.
\]
Note that under these notations one has $|\n w|=|\n(\theta_0+\phi)|$. Thus $w$ is a minimizer in $\I_\rho$ or $\J_\rho$ if and only if $\theta_0+\phi$ minimizes the weighted Dirichlet functional under the condition fixed by the RHS of \eqref{8.ConditionForPhaseForIRho} or \eqref{8.ConditionForPhaseForJRho}. 

Consequently, we find that $\theta_0+\phi$ minimizes the weighted Dirichlet functional under its Dirichlet boundary condition.  

Therefore, we obtain easily that $-\Div\left[\alpha\n(\theta_0+\phi)\right]=0$. The identity $\n(\theta_0+\phi)=w\times\n w$ yields $-\Div(\alpha\n w)=\lambda w$. 

Hence, the Euler-Lagrange equations in \eqref{8.EquationForCriticalPointDir+Deg} and \eqref{8.EquationForCriticalPointDir+AlmostDir} are direct consequences of \eqref{8.CaracterisationEqS1Valued}.

The condition on the boundary of the holes for $w^{\rm deg}_{\rho,\alpha}$ (resp. $w^{\rm Dir}_{\rho,\alpha}$) follows from multiplying the equation satisfied by $\theta_0+\phi^{\rm deg}_{\rho,\alpha},\, w^{\rm deg}_{\rho,\alpha}={\rm e}^{\imath(\theta_0+\phi^{\rm deg}_{\rho,\alpha})}$ (resp. $\theta_0+\phi^{\rm Dir}_{\rho,\alpha},\, w^{\rm Dir}_{\rho,\alpha}={\rm e}^{\imath(\theta_0+\phi^{\rm Dir}_{\rho,\alpha})}$) by $\psi\in \mathcal{D}(\O,\R)$ (resp. $\psi\in \mathcal{D}(\O,\R)\text{ s.t }\psi\equiv {\rm Cst}_i$ in $B(x_i,\rho)$).

Since $\alpha$ is sufficiently smooth, we can rewrite the Euler-Lagrange equation as
\[
-\Delta \phi=\frac{\n\alpha\cdot\n(\phi+\theta_0)}{\alpha}\text{ with }\frac{\n\alpha\cdot\n(\phi+\theta_0)}{\alpha}\in L^2(\O_\rho).
\]
So, by elliptic regularity $\phi^{\rm deg}_{\rho,\alpha},\phi^{\rm Dir}_{\rho,\alpha}\in H^2(\O_\rho,\R)$, and consequently  $w^{\rm deg}_{\rho,\alpha},w^{\rm Dir}_{\rho,\alpha}\in H^2(\O_\rho,\S^1)$.

\section{Proof of Proposition \ref{P8.MainAuxPb:DirVsDeg}}\label{S8.ProofThirdAuxPb}




\vspace{3mm}
We prove the existence of a minimal configuration $\{{\bf x},{\bf d}\}=\{(x_1,...,x_N),(d_1,...,d_n)\}$ for $I_{\rho,\alpha}$.

Let $(\{{\bf x}_n,{\bf d}_n\})_n$ be a minimizing sequence of configuration of $I_{\rho,\alpha}$, \emph{i.e.}, 
\[
\inf_{\substack{w\in H^1(\O_\rho^n,\S^1)\text{ s.t. }\\w=g\text{ in }\O'\setminus\overline{\O\bigcup\cup B(x^n_i,\rho)}\\\deg_{\p B(x_i^n,\rho)}(w)=d^n_i\text{ for all }i}}\frac{1}{2}\int_{\O_\rho^n}\alpha|\n w|^2\to I_{\rho,\alpha};
\]
here $\O_\rho^n=\O'\setminus \cup\overline{B(x_i^n,\rho)}$.

Up to a subsequence, we have $N_n=N={\rm Cst}$, ${\bf d}_n={\bf d}={\rm Cst}$ and ${\bf x}_n\to{\bf x}$ with ${\bf x}=(x_1,...,x_N)$ s.t. $\min_{i\neq j}|x_i-x_j|\geq 8\rho$. 

Consider $w_n\in\mathcal{I}_\rho({\bf x}_n,{\bf d})$ a minimal map. Since $w_n$ is bounded independently of $n$ in $H^1(\O_\rho^n)$, up to a subsequence, we have $w_n\weak w_0$ in $H^1_{\rm loc}(\O^0_\rho)$, $\O^0_\rho=\O'\setminus \cup\overline{B(x_i,\rho)}$.

Clearly the following properties hold:
\begin{enumerate}[$\bullet$]
\item $w_0\in H^1_{\rm loc}(\O^0_\rho,\S^1)$ and  $w_0=g$ in $\O^0_\rho\setminus\overline{\O}$.
\item For all compact $K\subset \O^0_\rho$ we have $\displaystyle\frac{1}{2}\int_K \alpha|\n w_0|^2\leq\liminf\frac{1}{2}\int_K \alpha|\n w_n|^2\leq I_{\rho,\alpha}$.
\end{enumerate}
Thus $w_0\in H^1_g(\O^0_\rho,\S^1)$ and $\displaystyle\int_{\O^0_\rho} \alpha|\n w_0|^2\leq I_{\rho,\alpha}$. 

Now, it suffices to check that $\deg_{\p B(x_i,\rho)}(w_0)\in\N^*$ for all $i$.  Since $w_0$ is $\S^1$-valued, this fact is equivalent to $\deg_{\p B(x_i,\rho')}(w_0)\in\N^*$ for all $i$ and for all $\rho'\in(\rho,2\rho)$. 

In view of the facts:
\begin{enumerate}[$\bullet$]
\item for $\rho'\in(\rho,2\rho)$ we have $w_n'=w_{n|\O'\setminus \cup\overline{B(x_i,\rho')}}\weak w_0'=w_{0|\O'\setminus \cup\overline{B(x_i^n,\rho')}}$ 
\item the set $
\I':=\{w'\in H^1(\O'\setminus \cup\overline{B(x_i,\rho')},\S^1)\,|\,\deg_{\p B(x_i,\rho')}(w')=d_i\text{ for all }i\in\{1,...,N\}\}$ is closed under the $H^1$-weak convergence (see Appendix \ref{S8.ExistenceMiniMizingMap} or \cite{TheseLassoued}),
\end{enumerate}
since $w_n'\in\I'$, we obtain that $w'_0\in\I'$. Therefore  $\{{\bf x},{\bf d}\}=\{(x_1,...,x_N),(d_1,...,d_n)\}$ is a minimal configuration for $I_{\rho,\alpha}$.

\vspace{3mm}

Now we prove the existence of a minimal configuration for $J_{\rho,\alpha}$.

Let $({\bf x}_n)_n$ be a minimizing sequence of configuration for $J_{\rho,\alpha}$, \emph{i.e.}, 
\[
\hat{\mathcal{J}}_{\rho,\alpha}({\bf x}_n,{\bf 1})\to J_{\rho,\alpha}.
\]

Up to a subsequence, one may assume that there is ${\bf x}=(x_1,...,x_d)\in\O^d$ s.t. $x_i^n\to x_i$, $|x_i-x_j|\geq8\rho$ and $\dist(x_i,\p\O)\geq8\rho$.

Let $\eta_n=8\max|x_i^n-x_i|$. There is a smooth diffeomorphism $\phi_n:\R^2\to\R^2$  satisfying
\[
\begin{cases}\phi_n\equiv{\rm Id}_{\R^2}&\text{in }\R^2\setminus\cup\overline{ B(x^n_i,\rho+\eta_n^{1/2})}\\\phi_n\left[x_i+(1+\eta_n)x\right]=x^n_i+x&\text{for }x\in B(0,\rho)\\\|\phi_n-{\rm Id}_{\R^2}\|_{C^1(\R^2)}=o_n(1)
\end{cases}.
\]
For example we can consider $\phi_n={\rm Id}_{\R^2}+H_n$ with
\[
\begin{cases}
H_n\equiv0&\text{in }\R^2\setminus\cup\overline{ B(x^n_i,\rho+\eta_n^{1/2})}
\\
H_n\left[x_i+(1+\eta_n)x\right]=\left[1-\psi_n(|x|)\right](x_i^n-x_i-\eta_nx)&\text{for }x\in B(0,\displaystyle\frac{\rho+\eta_n^{1/2}}{1+\eta_n})
\end{cases}.
\]
Here $\psi_n:\R^+\to[0,1]$ is a smooth function satisfying 
\[
\psi_n(r)=\begin{cases}0&\text{if }r\leq \rho\\1&\text{if }r\geq \rho+\eta_n^{1/2}/2\end{cases}\text{ and $|\psi_n^\prime|=\mathcal{O}(\eta_n^{-1/2})$.}
\]

For $w_n\in\mathcal{J}_{\rho}({\bf x}_n,{\bf 1})$ a minimal map, we consider
\[
\begin{array}{cccc}
\tilde{w}_n:&\O\setminus\cup_i\overline{B(x_i,(1+\eta_n)\rho)}&\to&\S^1\\&x&\mapsto&w_n\left[\phi_n(x)\right]
\end{array}.
\]
Clearly $\tilde{w}_n$ is well defined and we have 
\[
\int_{\O\setminus\cup_i\overline{B(x_i,(1+\eta_n)\rho)}}\alpha|\n \tilde{w}_n|^2=\int_{\O\setminus\cup_i\overline{B(x^n_i,\rho)}}\alpha|\n {w}_n|^2+o_n(1),
\]
\[
\tilde{w}_n\left[x_i+(1+\eta_n)\rho{\rm e}^{\imath \theta}\right]={w}_n\left[\phi(x_i+(1+\eta_n)\rho{\rm e}^{\imath \theta})\right]={w}_n\left[x^n_i+\rho{\rm e}^{\imath \theta}\right]={\rm e}^{\imath (\theta+\theta_i)}.
\]
We can extend $\tilde{w}_n$ in $\cup_i B(x_i,(1+\eta_n)\rho)\setminus\overline{B(x_i,\rho)}$ by $\tilde{w}_n(x_i+r{\rm e}^{\imath\theta})={\rm e}^{\imath (\theta+\theta_i)}$, $\rho<r<(1+\eta_n)\rho$.

Clearly, we have $\tilde{w}_n\in \J_{\rho,\alpha}({\bf x},{\bf 1})$ and $\displaystyle\frac{1}{2}\int_{\O\setminus\cup_i\overline{B(x_i,\rho)}}\alpha|\n \tilde{w}_n|^2=J_{\rho,\alpha}+o_n(1)$.

Thus considering $w\in \J_{\rho,\alpha}({\bf x},{\bf 1})$ a minimizer of $\displaystyle\frac{1}{2}\int_{\O\setminus\cup_i\overline{B(x_i,\rho)}}\alpha|\n\cdot|^2$, we obtain
\[
\frac{1}{2}\int_{\O\setminus\cup_i\overline{B(x_i,\rho)}}\alpha|\n w|^2\leq\frac{1}{2}\int_{\O\setminus\cup_i\overline{B(x_i,\rho)}}\alpha|\n \tilde{w}_n|^2=J_{\rho,\alpha}+o_n(1).
\]
Letting $n\to\infty$ we deduce that the configuration ${\bf x}=(x_1,...,x_d)$ is minimal.


\section{Proof of Proposition \ref{P8.AuxResult1}}\label{S8.ProofOFfirstAuxiliaryProblem}
As explained  Section \ref{S8.DegCondDirCond}, Proposition \ref{P8.AuxResult1} is easily established when either $N=1$ or when the points are well separated. It remains to consider the case where $N\geq2$ and there are $i\neq j$ s.t. $|x_i-x_j|\leq4\eta_{\rm stop}$.
\subsection{The separation process}\label{S8.SeparationProcess}


We assume that $N\geq2$ and that the points are not well separated. Our purpose is to compare the energy of $\hat{\J}_{\rho,\alpha}$ to the energy of $\hat{\I}_{\rho,\alpha}$. To this purpose, we decompose $\O_\rho$ into several regions and we compare energies in each regions. These regions are constructed recursively using the following version of   Theorem IV.1 in \cite{BBH}.

\begin{lem}\label{L8.SeparationLemma}
Let $N\geq2$, $x_1,...,x_{N}\in\R^2$ and ${\eta}>0$. There are $\kappa\in\{9^0,...,9^{N-1}\}$ and $\{y_1,...,y_{N'}\}\subset\{x_1,...,x_N\}$ s.t.
\[
\cup_{i=1}^NB(x_i,{\eta})\subset\cup_{i=1}^{N'} B(y_i,\kappa{\eta})
\]
and 
\[
|y_i-y_j|\geq8\kappa{\eta}\text{ for }i\neq j.
\]
\end{lem}
We let $x_1^0,...,x_N^0$ denote the initial points $x_1,...,x_N$ and $N_0=N$ the initial number of points. For $k\geq1$ (here, $k$ is an iteration in the construction of the regions), we let $N_k$ denote the number of points selected at Step $k$, and denote the points we select by $x_1^k,...,x_{N_k}^k$. 


The recursive construction is made in such a way that $N_{k}>N_{k+1}$ and $N_k\geq 1$ for all $k\geq1$.

The process will stop at the end of Step $k$ if and only if one of the following conditions yields
\begin{enumerate}[Rule 1:]
\item there is a unique point in the selection (\emph{i.e.} $N_k=1$),
\item $\min_{i\neq j}|x_i^{k}-x_j^{k}|>4\eta_{\rm stop}$.
\end{enumerate}
{\bf Step $k$, $k\geq1$:} Let $\eta'_k=\frac{1}{4}\min_{i\neq j}|x^{k-1}_i-x^{k-1}_j|$. 

Using Lemma \ref{L8.SeparationLemma}, there are
\[
\kappa_k\in\{9^1,...,9^{N_{k-1}-1}\}\text{ and }\{x_1^k,...,x_{N_k}^k\}\subset\{x_1^{k-1},...,x_{N_{k-1}}^{k-1}\}
\]
s.t.
\[
\cup_{i=1}^{N_{k-1}}B(x^{k-1}_i,\eta_k')\subset \cup_{i=1}^{N_{k}}B(x_i^k,\kappa_k\eta'_k)\text{ and }|x_i^k-x_j^k|\geq8\kappa_k\eta_k'\text{ for }i\neq j.
\]We denote $\eta_k=2\kappa_k\eta_k'$. We stop the construction if $N_k=1$ (Rule 1) or if $\frac{1}{4}\min|x^{k-1}_i-x^{k-1}_j|>\eta_{\rm stop}$ (Rule 2). 

In Figure \ref{F8.StopSeparetedBall} $\&$ \ref{F8.StopUniqueBall} both  stop-conditions are presented.

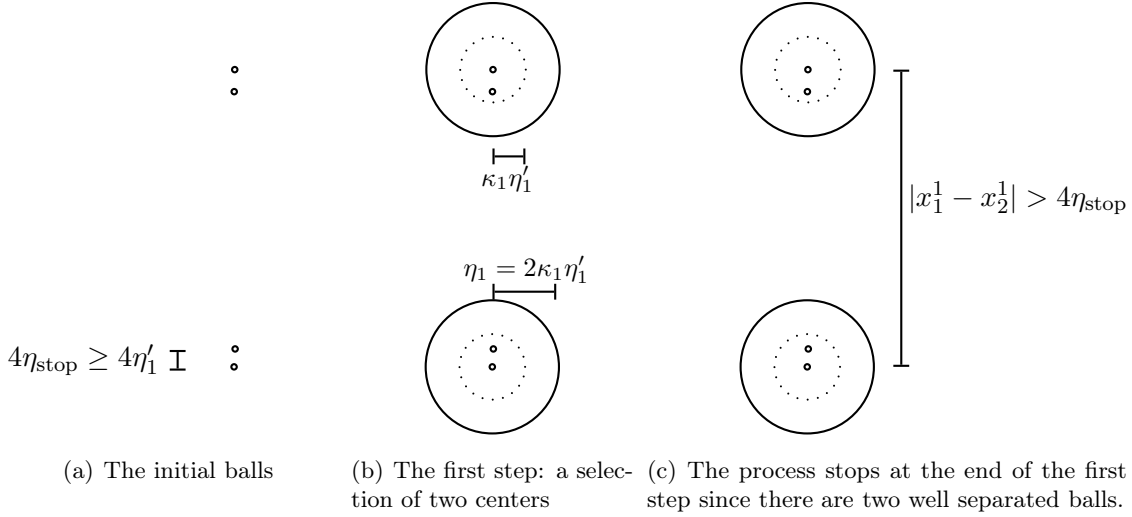
\begin{figure}[h!]
\subfigure[The initial balls]{\newrgbcolor{ttttff}{0.2 0.2 1}
\psset{xunit=0.05cm,yunit=0.05cm,algebraic=true,dotstyle=o,dotsize=3pt 0,linewidth=0.8pt,arrowsize=3pt 2,arrowinset=0.25}
\begin{pspicture*}(-65,-95)(20,22.1)
\psline{|-|}(-20,-76)(-20,-70.5)
\rput(-45,-73){$4\eta_{\rm stop}\geq4\eta_1'$}
\pscircle(-4.96,-1.77){0.05}
\pscircle(-4.82,4.09){0.05}
\pscircle(-5,-75){0.05}
\pscircle(-4.72,-70.23){0.05}
\end{pspicture*}
}\hfill
\subfigure[The first step: a selection of two centers]{\newrgbcolor{ttttff}{0.2 0.2 1}
\psset{xunit=0.05cm,yunit=0.05cm,algebraic=true,dotstyle=o,dotsize=3pt 0,linewidth=0.8pt,arrowsize=3pt 2,arrowinset=0.25}
\begin{pspicture*}(-40,-95)(30,22.1)
\pscircle(-4.96,-1.77){0.05}
\psline{|-|}(-4.96,-19)(3.8,-19)
\rput(-1,-25){{\small$\kappa_1\eta'_1$}}
\psline{|-|}(-4.96,-55)(12,-55)
\rput(4,-49){{\small$\eta_1=2\kappa_1\eta'_1$}}
\pscircle(-4.82,4.09){0.05}
\pscircle[linestyle=dotted](-4.82,4.09){0.45}
\pscircle(-4.82,4.09){0.9}
\pscircle(-5,-75){0.9}
\pscircle[linestyle=dotted](-5,-75){0.45}
\pscircle(-5,-75){0.05}
\pscircle(-4.72,-70.23){0.05}
\psdots[dotstyle=*,linecolor=blue](-60,60)
\end{pspicture*}
}\hfill
\subfigure[normal][The process stops at the end of the first step since there are two well separated balls.]{\newrgbcolor{ttttff}{0.2 0.2 1}
\psset{xunit=0.05cm,yunit=0.05cm,algebraic=true,dotstyle=o,dotsize=3pt 0,linewidth=0.8pt,arrowsize=3pt 2,arrowinset=0.25}
\begin{pspicture*}(-45,-95)(80,22.1)
\pscircle(-4.96,-1.77){0.05}
\pscircle(-4.82,4.09){0.05}
\pscircle[linestyle=dotted](-4.82,4.09){0.45}
\pscircle(-4.82,4.09){0.9}
\pscircle(-5,-75){0.9}
\pscircle[linestyle=dotted](-5,-75){0.45}
\pscircle(-5,-75){0.05}
\psline{->}(-60,60)(-20,60)
\psline{->}(-20.26,53.19)(-60.1,53.19)
\psline{->}(47.93,68.91)(48.21,73.68)
\pscircle(-4.72,-70.23){0.05}
\psline{|-|}(20,4.09)(20,-75)
\rput(51,-30){$|x_1^1-x_2^1|>4\eta_{\rm stop}$}

\end{pspicture*}}\caption{The process stops when we obtain well separated balls}\label{F8.StopSeparetedBall}
\end{figure} 
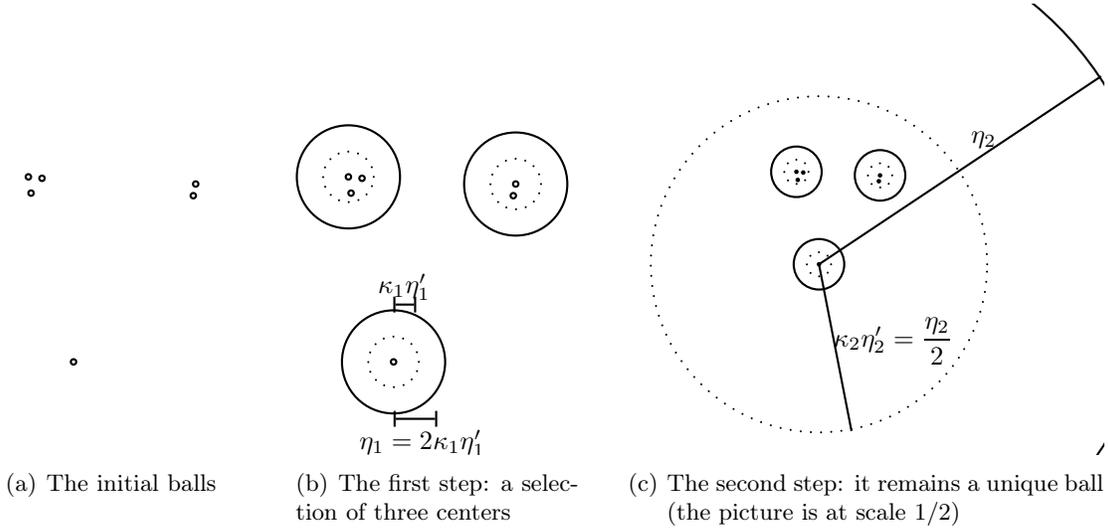
\begin{figure}[h!]
\subfigure[The initial balls]{\psset{xunit=0.08cm,yunit=0.08cm,algebraic=true,dotstyle=o,dotsize=3pt 0,linewidth=0.8pt,arrowsize=3pt 2,arrowinset=0.25}
\begin{pspicture*}(-21,-42)(24,20)
\pscircle(-12.26,4.24){0.05}
\pscircle(-10,4){0.05}
\pscircle(-11.82,1.54){0.05}
\pscircle(15.54,3.05){0.05}
\pscircle(15.14,1.11){0.05}
\pscircle(-4.76,-26.53){0.05}
\end{pspicture*}}\hfill
\subfigure[The first step: a selection of three centers]{\psset{xunit=0.08cm,yunit=0.08cm,algebraic=true,dotstyle=o,dotsize=3pt 0,linewidth=0.8pt,arrowsize=3pt 2,arrowinset=0.25}
\begin{pspicture*}(-21,-42)(26,20)
\pscircle(-12.26,4.24){0.05}
\pscircle(-10,4){0.05}
\pscircle(-11.82,1.54){0.05}
\pscircle(15.54,3.05){0.05}
\pscircle(15.14,1.11){0.05}
\pscircle[linestyle=dotted](-12.26,4.24){0.35}
\pscircle[linestyle=dotted](15.54,3.05){0.35}
\pscircle[linestyle=dotted](-4.76,-26.53){0.35}
\pscircle(-4.76,-26.53){0.05}
\pscircle(-4.76,-26.53){0.7}
\pscircle(-12.26,4.24){0.7}
\pscircle(15.54,3.05){0.7}
\psline{|-|}(-4.76,-36)(2.5,-36)
\psline{|-|}(-4.8,-17)(-1,-17)
\rput(0,-40){{\small$\eta_1=2\kappa_1\eta'_1$}}
\rput(-3,-14){{\small$\kappa_1\eta'_1$}}
\end{pspicture*}}\hfill
\subfigure[normal][The second step: it remains a unique ball \\ $\phantom{aaa}$(the picture is at scale $1/2$)]{
\psset{xunit=0.04cm,yunit=0.04cm,algebraic=true,dotstyle=o,dotsize=3pt 0,linewidth=0.8pt,arrowsize=3pt 2,arrowinset=0.25}
\begin{pspicture*}(-65,-90)(90,60)
\pscircle(-12.26,4.24){0.03}
\pscircle(-10,4){0.03}
\pscircle(-11.82,1.54){0.03}
\pscircle(15.54,3.05){0.03}
\pscircle(15.14,1.11){0.03}
\pscircle[linestyle=dotted](-12.26,4.24){0.175}
\pscircle[linestyle=dotted](15.54,3.05){0.175}
\pscircle[linestyle=dotted](-4.76,-26.53){0.175}
\pscircle(-4.76,-26.53){0.03}
\pscircle[linestyle=dotted](-4.76,-26.53){2.25}
\pscircle(-4.76,-26.53){0.35}
\pscircle(-12.26,4.24){0.35}
\pscircle(15.54,3.05){0.35}
\pscircle(-4.76,-26.53){4.5}
\psline(89,36)(-4.76,-26.53)
\psline(-4.76,-26.53)(6.05,-82)
\rput(20,-52){{\small$\kappa_2\eta'_2=\displaystyle\frac{\eta_2}{2}$}}
\rput(50,15){{\small$\eta_2$}}
\end{pspicture*}}\caption{The process stops when we obtain a unique ball}\label{F8.StopUniqueBall}
\end{figure}

\begin{remark}\begin{enumerate}[i.]
\item From the definitions of $\eta'_k$ and $\eta_k$, we have $N_k<N_{k-1}$ and $\eta_{k-1}\leq\eta_k'<\eta_k$. 
\item The balls $B(x^k_j,2\eta_k)$ are disjoint.
\item Denoting $\Lambda^k_j\subset\{1,...,N_{k-1}\}$ the set of indices $i$ s.t. $x_i^{k-1}\in B(x_j^k,\kappa_k\eta_k')$, then for $i\in\Lambda^k_j$ we have $B(x^{k-1}_i,\eta_k')\subset B(x_j^k,\kappa_k\eta_k')$. Furthermore, by construction, $|x^{k-1}_i-x^{k-1}_j|\geq 4\eta_k'$.
\end{enumerate}
\end{remark}

\subsection{The separation process gives a natural partition of $\O$}\label{S8.SeparationProcessGive}
Let $\O,g$, $x_1,...,x_N$, ${\bf d}$ and $\rho,\eta_{\rm stop}$ like in  Section \ref{S8.DegCondDirCond} with $N\geq2$ and s.t. the points are not well separated.

We apply the separation process. The process stops after $K$ steps, $1\leq K\leq N-1$.

We denote 
\begin{center}
$\{y_1,..., y_{N'}\}\subset\{x_1,...,x_N\}$ the selection that we obtain,\emph{ i.e.}, $y_j=x_j^K$ and $N'=N_K$,
\end{center}
\begin{equation}\label{8.EstimateAuxMutualDistance}
\eta=\begin{cases}
9^{N}\cdot\eta_{\rm stop}&\text{if $N'=1$}\\ \min\left\{9^{N}\cdot\eta_{\rm stop}\,,\,\frac{1}{4}\min|y_i-y_j|\right\}&\text{if }N'>1\end{cases},\text{ so }\eta\geq\max( \eta_K,\eta_{\rm stop}),
\end{equation}
\[
 \Lambda_j=\{i\in\{1,...,N\}\,|\,x_i\in B(y_j,\eta)\}\text{ and }\eta_0=\rho.
\]
We denote
\begin{equation}\label{8.PerforatedBall}
D_{j,k}=B(x^k_j,\eta_k)\setminus\cup_{x_i^{k-1}\in B(x^k_j,\eta_k)} \overline{B(x_i^{k-1},\eta'_{k})},\,k\in\{1,...,K\},\,j\in\{1,...,N_k\},
\end{equation}
\begin{equation}\label{8.Ring}
R_{j,k}=B(x^k_j,\eta'_{k+1})\setminus \overline{B(x^k_j,\eta_{k})},\,k\in\{0,...,K-1\},\,j\in\{1,...,N_k\},
\end{equation}
\begin{equation}\label{8.RingUnique}
R_j=B(y_j,\eta)\setminus \overline{B(y_j,\eta_K)},\,j\in\{1,...,N'\}
\end{equation}
and 
\begin{equation}\nonumber
D=\O\setminus\cup_{j\in\{1,...,N'\}} \overline{B(y_j,\eta)}.
\end{equation}
\begin{figure}[h!]
\subfigure[The macroscopic perforated domain and the first mesoscopic rings]{
\psset{xunit=0.5cm,yunit=0.5cm,algebraic=true,dotstyle=o,dotsize=3pt 0,linewidth=0.8pt,arrowsize=3pt 2,arrowinset=0.25}
\begin{pspicture*}(-16.3,-5)(0,7.3)
\pscircle[fillstyle=vlines](-11.96,2.41){0.90}
\pscircle[fillstyle=solid](-11.96,2.41){0.30}
\pscircle(-11.96,2.41){0.01}
\pscircle(-11.90,2.55){0.01}
\pscircle(-11.86,2.34){0.01}

\pscircle[fillstyle=vlines](-3.87,2.63){0.90}
\pscircle[fillstyle=solid](-3.87,2.63){0.30}

\psline(-11,3)(-10,4.8)
\psline(-5,3)(-6,4.8)
\pscircle(-3.87,2.63){0.01}
\rput{-178.51}(-8.14,2.3){\psellipse(0,0)(7.87,4.79)}
\rput{0}(-8,0){$D=\O\setminus\cup B(y_j,\eta)$}
\rput{0}(-10,5.2){ $R_1$}
\rput{0}(-6,5.2){ $R_2$}
\end{pspicture*}}\hfill
\subfigure[A mesoscopic ring and a mesoscopic perforated domain]{
\psset{xunit=0.2cm,yunit=0.2cm,algebraic=true,dotstyle=o,dotsize=3pt 0,linewidth=0.8pt,arrowsize=3pt 2,arrowinset=0.25}
\begin{pspicture*}(-35.8,-20.34)(8.05,24.62)
\pscircle
(-11.96,2.21){4}
\pscircle[fillstyle=vlines](-11.96,2.21){1.4}
\pscircle[fillstyle=solid](-11.96,2.21){0.2}
\pscircle[fillstyle=solid](-9.3,4){0.2}
\pscircle[fillstyle=solid](-10.46,-.5){0.2}
\pscircle[fillstyle=solid](-11.96,2.21){0.05}
\pscircle[fillstyle=solid](-9.3,4){0.05}
\pscircle[fillstyle=solid](-10.46,-.5){0.05}

\psline{|-|}(-32,21.4)(-32,2.37)
\psline{|-|}(-32,2.37)(-32,-4.88)
\rput{0}(-34,-1.75){$\eta_k$}
\rput{0}(-34,15){$\eta'_{k+1}$}
\psline{|-|}(0,3.2)(0,4.7)
\rput{0}(2,3.8){ $2\eta'_k$}
\rput{0}(-13,-10){ $R_{j,k}$}
\psline(-8,6)(-3,10)
\rput{0}(-1,10){ $D_{j,k}$}

\psline(-21,7)(-10.6,0)
\psline(-21,7)(-12.5,2)
\psline(-21,7)(-9.9,4)
\rput{0}(-25,7){ $R_{j',k-1}$'s}
\psline(-19,-4)(-10.5,-.5)
\rput{0}(-24.5,-4){$B(x_{i'}^{k-1}\!\!,\eta_{k-1})$}
\end{pspicture*}
}
\end{figure}
{Note that by construction of $\eta'_{k}$, $\eta_{k}$ and $x_i^k$  the following properties are satisfied:
\begin{equation}\label{8.TheAuxBallAreDisj}
\text{the balls $B(x_i^{k-1},2\eta'_{k})$ are disjoint}
\end{equation}
and
\begin{equation}\label{8.TheRadiusHaveTheSameOrder}
\text{$2\cdot9\eta'_k\leq \eta_k\leq 9^{N}\eta_k'$}.
\end{equation}

Therefore
\begin{equation}\label{8.DecompositionOfThePerforatedDomain}
\O_\rho=D\bigcup\cup_{j,k}\overline{D_{j,k}}\bigcup\cup_{j,k}R_{j,k}\bigcup \cup_{j}\overline{R_j}\text{ with disjoint unions}.
\end{equation}
}
\subsection{Construction of test functions}\label{Sectionkjsdhsdjfghpp}
\paragraph{Construction of test functions in $D$ and $D_{j,k}$ }~~\\
\begin{lem}\label{L8.UpperBoundS1ValuedMapInGoodCondition}
\begin{enumerate}
\item Let $\eta>0$. There is $C_1(\eta)>0$ (depending on $\O,g$ and $\eta$) s.t. if $x_1,...,x_N\in\O$ satisfy $\min_{i\neq j}|x_i-x_j|,\min_i\dist(x_i,\p\O)>4\eta$ and $d_1,...,d_N\in\N^*$ are s.t. $\sum d_i=d$  then there is $w\in H^1_g(\O\setminus\cup\overline{B(x_i,\eta)},\S^1)$ s.t. $w(x)=\frac{(x-x_i)^{d_i}}{\eta^{d_i}}$ on $\p B(x_i,\eta)$ and
\[
\int_{\O\setminus\cup\overline{B(x_i,\eta)}}|\n w|^2\leq C_1(\eta).
\]
Moreover $C_1$ can be considered decreasing with $\eta$.
\item Let $\eta>0,\kappa\geq8$, $d_0,d_1,...,d_N\in\N^*$ be s.t. $\sum_{1\leq i\leq N}d_i=d_0$. Then, there is $C_2(\kappa,d_0)$ s.t. for $x_1,...,x_N\in B(0,\kappa\eta)$ satisfying $\min_{i\neq j}|x_i-x_j|\geq4\eta$ we can associate a map  $w\in H^1(B(0,2\kappa\eta)\setminus\cup \overline{B(x_i,\eta)},\S^1)$ s.t.
\[
w(x)=\begin{cases}
\displaystyle\frac{x^{d_0}}{(2\kappa\eta)^{d_0}}&\text{on $\p B(0,2\kappa\eta)$ }\\\displaystyle\frac{(x-x_i)^{d_i}}{\eta^{d_i}}&\text{on }\p B(x_i,\eta)
\end{cases}
\]
and
\[
\int_{B(0,2\kappa\eta)\setminus\cup \overline{B(x_i,\eta)}}{|\n w|^2}\leq C_2(\kappa,d_0).
\]
Moreover $C_2$ can be considered increasing with $\kappa$, $d_0$. 
\end{enumerate}
\end{lem}
\begin{proof}
In order to prove 1., we consider, \emph{e.g.}, the test function defined in $\O_\eta:=\O\setminus\cup\overline{B(x_i,\eta)}$ by
\[
w={\rm e}^{\imath H}\Pi_i\frac{(x-x_i)^{d_i}}{|x-x_i|^{d_i}}\text{ with }H\text{ s.t. }\begin{cases}H:\O_\eta\to\R\\H\equiv0\text{ in }\left\{\dist\left[x,\p\O_\eta\right]\geq\eta\right\} \\-\Delta H=0\text{ in }\left\{\dist\left[x,\p\O_\eta\right]<\eta\right\}\\\text{$w\in H^1_g(\O_\eta,\S^1)$} \text{ and  $w(x)=\frac{(x-x_i)^{d_i}}{\eta^{d_i}}$ on $\p B(x_i,\eta)$}\end{cases}.
\]

Assertion 2. was essentially established in \cite{HS95}, Section 3. We adapt here the argument in \cite{HS95}. By conformal invariance, we may assume that $\eta=1$. We let 
\[
w(x)=\begin{cases}
\displaystyle\Pi_i\frac{\left[x+2x_i\left(\frac{|x|}{\kappa}-2\right)\right]^{d_i}}{\left|x+x_i\left(\frac{|x|}{\kappa}-2\right)\right|^{d_i}}&\text{in }B(0,2\kappa)\setminus\overline{B(0,\frac{3\kappa}{2})}
\\
\displaystyle\Pi_i\frac{(x-x_i)^{d_i}}{|x-x_i|^{d_i}}&\text{in }B(0,\frac{3\kappa}{2})\setminus\cup\overline{B(x_i,3/2)}
\\
\displaystyle\frac{(x-x_i)^{d_i}}{|x-x_i|^{d_i}}{\rm e}^{\imath(2|x-x_i|-2)\varphi_i}&\text{in }B(x_i,3/2)\setminus\cup\overline{B(x_i,1)}
\end{cases};
\]
here $\varphi_i\in C^\infty(B(x_i,3/2),\R)$ is defined by ${\rm e}^{\imath\varphi_i}=\Pi_{j\neq i}\dfrac{(x-x_j)^{d_j}}{|x-x_j|^{d_j}}$ and $\varphi_i(x_i)\in[0,2\pi)$. Clearly $\|\varphi_i\|_{H^1(B(x_i,3/2)\setminus\overline{B(x_i,1)})}$ is bounded by a constant which depends only on $d_0$.
\end{proof}

 
By \eqref{8.EstimateAuxMutualDistance} and Lemma \ref{L8.UpperBoundS1ValuedMapInGoodCondition}.1, one may find a map $w_0\in H^1(D,\S^1)$ s.t.
 \[
 w_0=\begin{cases}g&\text{on }\p\O\\ w_0(x)=\displaystyle\frac{(x-y_j)^{\tilde{d_j}}}{\eta^{\tilde{d}_j}}&\text{on }\p B(y_j,\eta)
 \end{cases}(\text{ where }\tilde{d}_j=\sum_{x_i\in B(y_j,\eta)}d_i)
 \]
satisfying in addition
\begin{equation}\label{8.ContributionOnD}
\int_{D}|\n w_0|^2\leq C_1(\eta)\leq C_1(\eta_{\rm stop}).
\end{equation}
For each $D_{j,k}$, combining  \eqref{8.PerforatedBall}, \eqref{8.TheAuxBallAreDisj}, \eqref{8.TheRadiusHaveTheSameOrder}  and using Lemma \ref{L8.UpperBoundS1ValuedMapInGoodCondition}.2, there exists a map $w_{j,k}\in H^1(D_{j,k},\S^1)$ s.t.
 \[
 w_{j,k}(x)=\begin{cases}\displaystyle\frac{(x-x^k_j)^{\tilde{d}_{j,k}}}{\eta_k^{\tilde{d}_{j,k}}}&\text{for }x\in\p B(x_j^k,\eta_k)\\\displaystyle\frac{(x-x_i^{k-1})^{\tilde{d}_{i,k-1}}}{\eta_k'^{\tilde{d}_{i,k-1}}}&\text{for }x\in\p B(x_i^{k-1},\eta'_k)
 \end{cases}.
 \]
Here,
 \[
 \tilde{d}_{j,k}=\sum_{x_i\in B(x_j^{k},\eta_k)}d_i
 \]
and
\begin{equation}\label{8.ContributionOnDjk}
\int_{D_{j,k}}|\n w_{j,k}|^2\leq C_2(2\kappa_k,d_{j,k})\leq C_2(2\cdot 9^{d-1},d).
\end{equation}

\paragraph{Construction of test functions in $R_j$'s and $R_{j,k}$'s}~~\\
For $R>r>0$ and $x_0\in\R^2$ we denote $\Ring(x_0,R,r):=B(x_0,R)\setminus \overline{B(x_0,r)}$. For $\alpha\in L^{\infty}(\R^2,[b^2,1])$, we define
\begin{equation}\label{8.MinAnnular}
\mu_\alpha{}(\Ring(x_0,R,r),\tilde{d})=\inf_{\substack{w\in H^1(\Ring(x_0,R,r),\S^1)\\\deg_{\p B(x_0,R)}(w)=\tilde{d}}}\frac{1}{2}\int_{\Ring(x_0,R,r)}\alpha|\n w|^2
\end{equation}
and
\begin{equation}\label{8.MinAnnularDir}
\mu_\alpha^{\rm Dir}(\Ring(x_0,R,r),\tilde{d})=\inf_{\substack{w\in H^1(\Ring(x_0,R,r),\S^1)\\w(x_0+R{\rm e}^{\imath{\theta}})={\rm e}^{\imath\tilde{d}{\theta}}\\w(x_0+r{\rm e}^{\imath{\theta}}){\rm e}^{-\imath\tilde{d}{\theta}}={\rm Cst}}}{\frac{1}{2}\int_{\Ring(x_0,R,r)}{\alpha|\n w|^2}}.
\end{equation}
In the special case $\alpha=U_\v^2$, we denote
\begin{equation}\nonumber
\mu_\v{}(\Ring(x_0,R,r),\tilde{d})=\mu_{U_\v^2}(\Ring(x_0,R,r),\tilde{d})
\end{equation}
and
\begin{equation}\nonumber
\mu_\v^{\rm Dir}(\Ring(x_0,R,r),\tilde{d})=\mu_{U^2_\v}^{\rm Dir}(\Ring(x_0,R,r),\tilde{d}).
\end{equation}

Note that the minimization problems \eqref{8.MinAnnular} and \eqref{8.MinAnnularDir} admit solutions; this is obtained by adapting the proof of Proposition \ref{P8.ExistenceOfminInIJ}.


We present an adaptation of a result of Sauvageot, Theorem 2 in \cite{MS1}.
\begin{prop}\label{P7.MyrtoRingDegDir}
There is  $C_b>0$ depending only on $b\in(0,1)$ s.t. for $R>r>0$ and $\alpha\in L^{\infty}(\R^2,\R)$ satisfying $1\geq\alpha\geq b^2$, we have
\[
\mu_\alpha^{\rm Dir}(\Ring(x_0,R,r),\tilde{d})\leq \mu_\alpha{}(\Ring(x_0,R,r),\tilde{d})+{\tilde{d}}^2 C_b.
\]
\end{prop}

\begin{proof}
This result was obtained by Sauvageot with $\alpha\in W^{1,\infty}(\R^2,[b^2,1])$. We may extend this estimate to $\alpha\in L^\infty(\R^2,[b^2,1])$. 

Indeed, let $(\rho_t)_{1>t>0}$ be a classical mollifier, namely $\rho_t(x)=t^{-2}\rho(x/t)$ with $\rho\in C^\infty(\R^2,[0,1])$, ${\rm Supp}\,\rho\subset B(0,1)$ and $\int_{\R^2}\rho=1$. 

Set $\alpha_t=\alpha\ast\rho_t\in W^{1,\infty}(B(x_0,R),[b^2,1])$. We have 
\begin{equation}\label{7.Extensonsdihfresultsauv}
\displaystyle\lim_{t\to0} \mu_{\alpha_t}(\Ring(x_0,R,r),\tilde{d})=\mu_\alpha(\Ring(x_0,R,r),\tilde{d})
\end{equation}
and 
\begin{equation}\label{7.ExtensonsdihfresultsauvDir}
\displaystyle\lim_{t\to0} \mu_{\alpha_t}^{\rm Dir}(\Ring(x_0,R,r),\tilde{d})=\mu_\alpha^{\rm Dir}(\Ring(x_0,R,r),\tilde{d}).
\end{equation}
We prove \eqref{7.Extensonsdihfresultsauv},  Equality \eqref{7.ExtensonsdihfresultsauvDir} follows with the same lines.

Let $w$ be a minimizer of $\mu_\alpha(\Ring(x_0,R,r),\tilde{d})$. By using Dominated convergence theorem, since $\alpha_t\to\alpha$ in $L^1(B(x_0,R))$, we obtain that $\alpha_t|\n w|^2\to\alpha|\n w|^2$ in $L^1(\Ring(x_0,R,r))$ as $t\to0$. Consequently
\[
\displaystyle\lim_{t\to0} \mu_{\alpha_t}(\Ring(x_0,R,r),\tilde{d})\leq\mu_\alpha(\Ring(x_0,R,r),\tilde{d}).
\]
On the other hand, let $w_t$  be a minimizer of $\mu_\alpha(\Ring(x_0,R,r),\tilde{d})$ and let $t_n\downarrow0$. Up to a subsequence, $w_{t_n}\weak w_0$ in $H^1( \Ring(x_0,R,r))$ as $n\to\infty$ and $\sqrt{\alpha_{t_n}}\n w_{t_n}\weak \sqrt{\alpha}\n w_0$ in $L^2(\Ring(x_0,R,r))$.

Since the class $\I:=\{w\in H^1( \Ring(x_0,R,r),\S^1)\,|\,\deg_{B(x_0,R)}(w)=\tilde{d}\}$ is closed under the $H^1$-weak convergence (see Appendix \ref{S8.ExistenceMiniMizingMap} or \cite{TheseLassoued}), we obtain that $w_0\in\I$. Consequently, we have 
\[
\liminf _{t\to0}\mu_{\alpha_t}(\Ring(x_0,R,r),\tilde{d})\geq \mu_\alpha(\Ring(x_0,R,r),\tilde{d}).
\]
Thus the proof of \eqref{7.Extensonsdihfresultsauv} is complete.

Therefore, without loss of generality, we may assume that $\alpha$ is Lipschitz.
\vspace{2mm}

One may easily prove that if $R\leq4r$, then $\mu_\alpha^{\rm Dir}(\Ring(x_0,R,r),\tilde{d})\leq2\tilde{d}^2\pi\ln4$. Thus we assume that $R>4r$. Clearly, it suffices to obtain the result for $\tilde{d}=1$ and $x_0=0$.

Let $w$ be a global minimizer of $\mu_\alpha{}(\Ring(x_0,R/2,2r),1)$. As explained  Section \ref{S8.ExistenceMiniMizingMap}, denoting $x/|x|={\rm e}^{\imath\theta}$, one may write $w= {\rm e}^{\imath(\theta+\phi)}$ for some $\phi\in H^2(\Ring(x_0,R/2,2r),\R)$. Now we switch to polar coordinates.

Consider 
\[
I=\left\{\rho\in[2r,R/2]\,\left|\,\int_0^{2\pi}\alpha|\n(\theta+\phi)|^2(\rho,\theta)\,{\rm d}\theta\leq\frac{1}{\rho^2}\int_0^{2\pi}\alpha(\rho,\theta)\,{\rm d}\theta\right.\right\}.
\]
Then $I$ is closed (since $\phi\in H^2$). On the other hand, $I$ is non empty, by the mean value theorem.

Let $r_1=\min I$ and $r_2=\max I$. We may assume that $\phi(r_2,0)=0$ and $\phi(r_1,0)=\theta_0$. We construct a test function:
\[
\phi'(\rho,\theta)=\begin{cases}
0&\text{if }2r_2\leq\rho\leq R
\\
\displaystyle\frac{2r_2-\rho}{r_2}\phi(r_2,\theta)&\text{if }r_2\leq\rho\leq2r_2
\\
\phi(\rho,\theta)&\text{if }r_1\leq\rho\leq r_2
\\
\displaystyle\frac{2\rho-r_1}{r_1}\phi(r_1,\theta)+2\frac{r_1-\rho}{r_1}\theta_0&\text{if }r_1/2\leq \rho\leq r_1
\\
\theta_0&\text{if }r\leq\rho\leq r_1/2
\end{cases}.
\]
As explained in \cite{MS1}, there is $C$ depending only on $b$ s.t. 
\[
\frac{1}{2}\int_{\Ring(0,R/2,2r)}\alpha\left(|\n (\theta+\phi')|^2-|\n (\theta+\phi)|^2\right)\leq C.
\]
Thus the result follows.
\end{proof}

As a direct consequence of Proposition \ref{P7.MyrtoRingDegDir} (the two first assertions of the next proposition are direct), we have
\begin{prop}\label{P8.DirectPropAnnProb}
Let $\alpha\in L^{\infty}(\R^2,[b^2,1])$, $R>r_1>r>0$, $\tilde{d}\in\Z$ and $x_0\in\R^2$,  we have
\begin{enumerate}
\item $\mu_\alpha{}(\Ring(x_0,R,r),\tilde{d})=\tilde{d}^2\mu_\alpha{}(\Ring(x_0,R,r),1)$,
\item $\displaystyle b^2\pi\ln\frac{R}{r}\leq\mu_\alpha{}(\Ring(x_0,R,r),1)\leq \pi\ln\frac{R}{r}$,
\item $\mu_\alpha{}(\Ring(x_0,R,r),1)\leq\mu_\alpha{}(\Ring(x_0,R,r_1),1)+\mu_\alpha(\Ring(x_0,r_1,r),1)+2C_b$ where $C_b$ is given by Proposition \ref{P7.MyrtoRingDegDir} and depends only on $b$.
\end{enumerate}
\end{prop}

 

We turn to the construction of test functions in $R_j$ and $R_{j,k}$.\vspace{2mm}

Using Proposition \ref{P7.MyrtoRingDegDir}, there is $C_b$ depending only on $b\in(0,1)$ s.t. for $\alpha\in L^{\infty}(\O,[b^2,1])$ and for all $k\in\{1,...,K-1\}$, $j\in\{1,...,N_k\}$,  there is $w_{\alpha,j,k}\in H^1(R_{j,k},\S^1)$ s.t.
\[
w_{\alpha,j,k}(x)=\begin{cases}\displaystyle\frac{(x-x_j^k)^{\tilde{d}_{j,k}}}{{\eta'}_{k+1}^{\tilde{d}_{j,k}}}&\text{for }x\in\p B(x_j^k,\eta_{k+1}')
\\
\displaystyle\gamma_{\alpha,j,k}\frac{(x-x_j^k)^{\tilde{d}_{j,k}}}{\eta_{k}^{\tilde{d}_{j,k}}}&\text{for }x\in\p B(x_j^k,\eta_{k})\text{ where }\gamma_{\alpha,j,k}\in\S^1
\end{cases}
\]
and s.t. for all $w\in H^1(R_{j,k},\S^1)$ satisfying $\deg_{\p B(x_j^k,\eta_k)}(w)=\tilde{d}_{j,k}$ one has
\begin{equation}\label{8.ContributionOnRjk}
\int_{R_{j,k}}\alpha|\n w_{\alpha,j,k}|^2\leq \int_{R_{j,k}}\alpha|\n w|^2+C_b\tilde{d}_{j,k}^2\leq \int_{R_{j,k}}\alpha|\n w|^2+2d^2C_b.
\end{equation}
Now we consider the rings $R_j$. For $j\in\{1,...,N'\}$, we denote 
\[
\tilde{d}_{j}=\sum_{x_i\in B(y_j,\eta)}d_i.
\]
Using Proposition \ref{P7.MyrtoRingDegDir}, for $j\in\{1,...,N'\}$, we obtain $w_{\alpha,j}\in H^1(R_j,\S^1)$ s.t.
\[
w_{\alpha,j}(x)=\begin{cases}\displaystyle\frac{(x-y_j)^d}{{\eta}^d}&\text{for }x\in\p B(y_j,\eta)
\\
\displaystyle\gamma_{\alpha,j}\frac{(x-y_j)^{d}}{\eta_{K}^{d}}&\text{for }x\in\p B(y_j,\eta_{K})\text{ where }\gamma_{\alpha,j}\in\S^1
\end{cases}
\]
and s.t. for all $w\in H^1(R_{j},\S^1)$ satisfying $\deg_{\p B(y_j,\eta)}(w)=\tilde{d}_j$ one has
\begin{equation}\label{8.ContributionOnR0}
\int_{R_{j}}\alpha|\n w_{\alpha,j}|^2\leq \int_{R_{j}}\alpha|\n w|^2+2d^2C_b.
\end{equation}
\subsection{Proof of Proposition \ref{P8.AuxResult1}}
Note that there are at most $d^2$ regions $D_{j,k}$, at most $d^2$ rings $R_{j,k}$ and at most $d$ rings $R_j$. Consequently, denoting
\begin{equation}\nonumber
C_4(\eta_{\rm stop})=C_1(\eta_{\rm stop})+d^2 C_2(2\cdot 9^{d-1},d)+4d^4C_b
\end{equation}
and using \eqref{8.DecompositionOfThePerforatedDomain}, \eqref{8.ContributionOnD}, \eqref{8.ContributionOnDjk}, \eqref{8.ContributionOnRjk}, \eqref{8.ContributionOnR0}, one may construct a test function $w_\alpha\in\J_{\rho}$ (up to multiply by some $\S^1$-Constants each function previously constructed) s.t. for all $w\in\I_{\rho}$, one has
\begin{equation}\label{8.FundamentalEstAuxProblS1ValuedMap}
\int_{\O_{\rho}}{\alpha|\n w_\alpha|^2}\leq \int_{\O_{\rho}}{\alpha|\n w|^2}+C_4.
\end{equation} 
Clearly, \eqref{8.FundamentalEstAuxProblS1ValuedMap} allows us to prove Proposition \ref{P8.AuxResult1} with $C_0=C_4/2$.

\section{Proof of Proposition \ref{P8.ToMinimizeSecondPbThePointAreFarFromBoundAndHaveDegree1}}\label{S8.ProofOfSecondAuxPb}
\subsection{Description of the special solution $U_\v$}\label{S8.DescrSpecSol}

From Proposition \ref{P8.UepsCloseToaeps}, we know that far away $\p\o_\v$, $U_\v$ is uniformly close to $a_\v$. Here we prove that, in a neighborhood of $\p\o_\v$, $U_\v$ is very close to \emph{a cell regularization  of $a_\v$}.

Let 
\[
\begin{array}{cccc}
a^\lambda:&Y=(-\frac{1}{2},\frac{1}{2})\times (-\frac{1}{2},\frac{1}{2})&\to&\{b,1\}\\&x&\mapsto&\begin{cases}b&\text{if }x\in\o^\lambda=\lambda\cdot\o\\1&\text{otherwise}\end{cases}
\end{array}.
\]
Consider $V_\xi$ {\bf the} unique minimizer of 
\begin{equation}\label{SpecialSolutionForPeriodicity}
E_\xi^{a^\lambda}(V,Y)=\frac{1}{2}\int_{Y}{|\n V|^2+\frac{1}{2\xi^2}({a^\lambda}^2-V^2)^2},\,V\in H^1_1(Y,\R).
\end{equation}

\begin{lem}\label{L8.UIsAlmostPeriodic}
We have the existence of $C,\gamma>0$ s.t. for $\v>0$ and $x \in Y$
\[
|U_\v[y_{i,j}^\v+\delta^j x]-V_{\v/\delta^j}\left(x\right)|\leq C{\rm e}^{-\frac{\gamma \delta^j}{\v}}.
\]
Thus in the periodic case, we have $U_\v$ which is almost a $\delta\cdot(\Z\times\Z)$-periodic function in $\O_\delta^{\rm incl}$ in the sense that
\[
|U_\v(x)-U_\v\left[x+(\delta k,\delta l)\right]|\leq C{\rm e}^{-\frac{\gamma\delta}{\v}}\text{ if }x,x+(\delta k,\delta l)\in\O_\delta^{\rm incl}\text{ and }k,l\in\Z.
\]
\end{lem}
\begin{proof}


{\bf Step 1.} We first prove that, for all $s>0$ and for sufficiently small $\v$, we have $U_\v^2\geq\dfrac{b^2+1}{2}-s$ in $\O\setminus\o_\v$. The same argument leads to $U_\v^2\leq\dfrac{b^2+1}{2}+s$ in $\overline{\o_\v}$ and for sufficiently small $\xi$: $V^2_\xi\geq\dfrac{b^2+1}{2}-s$ in $Y\setminus\o^\lambda$ and $V_\xi^2\leq\dfrac{b^2+1}{2}+s$ in $\o^\lambda$.

From Proposition \ref{P8.UepsCloseToaeps}, it suffices to prove that for 
\[
R=\alpha^{-1}\ln\frac{C}{1-\sqrt{\frac{1+b^2}{2}}},
\] 
we have  $U^2_\v\geq\dfrac{b^2+1}{2}-s$  in $\{x\in\O\setminus\o_\v\,|\,\dist(x,\p\o_\v)<R\v\}$ (for sufficiently small $\v$). Here $C>1,\alpha>0$ are given by \eqref{8.UepsCloseToaeps}. 

We fix $0<s<1$ and we 
let $z_\v=y^\v_{i,j}+\lambda\delta^j z_\v^0\in\p\o_\v$, $z_\v^0\in\p\o$. For $x\in B(z_\v,\lambda\delta^{\num+1})$, we write $x=z_\v+\v \tilde{x}$ with $\tilde{x}\in B(0,{\lambda\delta^{\num+1}}/{\v})$. Here $\num=1$ and $y^\v_{i,j}\in\delta\Z\times\delta\Z$ if we are in the periodic situation. 

We define  
\[
\begin{array}{cccc}
\tilde{U}_\v(\tilde{x}):&B(0,{\lambda\delta^{\num+1}}/{\v})&\to&[b,1]\\&\tilde{x}&\mapsto&U_\v(z_\v+\v \tilde{x}) 
\end{array}.
\]

It is easy to check that
\begin{equation}\label{ProcheinterfEq}
\begin{cases}
-\Delta \tilde{U}_\v=\tilde{U}_\v(\tilde{a}_\v^2-\tilde{U}_\v^2)\text{ in } B(0,{\lambda\delta^{\num+1}}/{\v})
\\
\tilde{U}_\v\in H^1\cap L^\infty(B(0,{\lambda\delta^{\num+1}}/{\v}),[b,1])
\end{cases}
\end{equation}
where
\[
\tilde{a}_\v=\begin{cases}
b&\text{in }\dfrac{\o_\v-z_\v}{\v}\cap B(0,{\lambda\delta^{\num+1}}/{\v})
\\
1&\text{in }\dfrac{(\R^2\setminus\o_\v)-z_\v}{\v}\cap B(0,{\lambda\delta^{\num+1}}/{\v})
\end{cases}.
\]
Clearly
\begin{eqnarray*}
\dfrac{\o_\v-z_\v}{\v}\cap B(0,\lambda\delta^{\num+1}/{\v})
&=&\left[\dfrac{\lambda\delta^j}{\v}\cdot(\o-z_\v^0)\right]\cap B(0,\lambda\delta^{\num+1}/{\v})
\\&=&\dfrac{\lambda\delta^j}{\v}\cdot\left[(\o-z_\v^0)\cap B(0,\delta^{\num+1-j})\right],
\end{eqnarray*}
and thus
\begin{eqnarray*}
\dfrac{(\R^2\setminus\o_\v)-z_\v}{\v}\cap B(0,\lambda\delta^{\num+1}/{\v})
&=&\dfrac{\lambda\delta^j}{\v}\cdot\left[(\R^2\setminus\o)-z_\v^0)\cap B(0,\delta^{\num+1-j})\right].
\end{eqnarray*}
Note that ${\lambda\delta^{\num+1}}/{\v}\to\infty$ and $\delta^{\num+1-j}\to0$, thus by smoothness of $\o$, up to a subsequence, we have $\dfrac{\lambda\delta^j}{\v}\cdot\left\{[(\R^2\setminus\o)-z_\v^0]\cap B(0,\delta^{{\num+1}-j})\right\}\to\mathcal{R}_{\theta_0}(\R\times\R^+)$. Here $\mathcal{R}_{\theta_0}$ is the vectorial rotation of angular $\theta_0\in[0,2\pi)$.

For sake of simplicity, we assume that $\theta_0=0$. 

 From \eqref{ProcheinterfEq} and standard elliptic estimates, we obtain that $\tilde{U}_\v$ is bounded in $W^{2,p}(B(0,R))$ for $p\geq2$, $R>0$. Thus up to consider a subsequence, we obtain that $\tilde{U}_\v\underset{}{\to} \tilde{U}_b$ in $C^1_{\rm loc}(\R^2)$ ($\v\to0$) where $\tilde{U}_b\in C^1(\R^2,[b,1])$ is a solution of
\begin{equation}\label{EquaLimite}
\begin{cases}
-\Delta \tilde{U}_b=\tilde{U}_b(1-\tilde{U}_b^2)&\text{in }\R\times\R^+
\\
-\Delta \tilde{U}_b=\tilde{U}_b(b^2-\tilde{U}_b^2)&\text{in }\R\times\R^-
\\
\tilde{U}_b\in C^1(\R^2)\cap H^2_{\rm loc}(\R^2)\cap L^\infty(\R^2)
\end{cases}.
\end{equation}
It is proved in \cite{K1} (Theorem 2.2), that \eqref{EquaLimite} admits a unique positive solution. Moreover $\tilde{U}_b(x,y)=U_b(y)$ ($\tilde{U}_b$ is independent of its first variable) and $U_b$ is {\bf the unique} solution of 
\[\begin{cases}
-U''_b=U_b(1-U_b^2)&\text{in }\R^+\\- U''_b=U_b(b^2-U_b^2)&\text{in }\R^-
\\U_b\in C^1(\R,\R),\, U_b'>0,&\displaystyle\lim_{+\infty}U_b=1,\,\lim_{-\infty}U_b=b
\end{cases}.
\]
Note that since the limit is unique, the convergence is valid for the whole sequence. 

This solution $U_b$ may be explicitly obtained by looking for $U_b$ under the form
\[
U_b(x)=
\begin{cases}
\dfrac{A{\rm e}^{\sqrt2 x}-1}{A{\rm e}^{\sqrt2 x}+1}&\text{if }x\geq0\\b\dfrac{B{\rm e}^{-b\sqrt2 x}-1}{B{\rm e}^{-b\sqrt2 x}+1}&\text{if }x<0
\end{cases}.
\]
We get $B=-\dfrac{3b^2+1+2b\sqrt{2(b^2+1)}}{1-b^2}$,  $A=\dfrac{B(1+b)+1-b}{B(1-b)+1+b}$ and 
\[
U_b(0)=b\dfrac{B-1}{B+1}=\frac{1+b^2+b\sqrt{2(b^2+1)}}{2b+\sqrt{2(b^2+1)}}=\frac{1-b^2}{2b+\sqrt{2(b^2+1)}}+b=\sqrt{\dfrac{b^2+1}{2}}
\]

Since $U_b(0)^2=\dfrac{b^2+1}{2}$ and $U_b$ is an increasing function, for $x\geq0$, $U_b(x)^2\geq\dfrac{b^2+1}{2}$. From the convergence $\tilde{U}_\v\to \tilde{U}_b$ in  $L^\infty(B(0,R))$, we obtain that, for $\v$ sufficiently small, $\tilde{U}_\v^2\geq\dfrac{b^2+1}{2}-s$ in $B(0,R)\cap\left\{\dfrac{\lambda\delta^j}{\v}\cdot[(\R^2\setminus\o)-z_\v^0]\right\}$.
\\{\bf Step 2.} Fix $j\in\{1,...,\num\}$ s.t. $\M^\v_{j}\neq\emptyset$ and fix $i\in\M^\v_{j}$. Note that if we are in the periodic case then $j=1$ and we fix $y_{k,l}=(\delta k ,\delta l)\in\delta\Z\times\delta\Z$ s.t. $y_{k,l}+\delta\cdot Y\subset \O$.

We denote $\xi:=\dfrac{\delta^j}{\v}$. For $x\in Y$, consider $W(x)=V_\xi(x)-U_\v(y_{i,j}^\v+\delta^j x)$ which satisfies (using \eqref{8.UepsCloseToaeps})
\begin{equation}\nonumber
\begin{cases}
\displaystyle-\xi^2\Delta W(x)={W(x)}{}\left\{{a^\lambda}(x)^2-[V_\xi(x)^2+U_\v(y_{i,j}^\v+\delta^j x)V_\xi(x)+U_\v(y_{i,j}^\v+\delta^j x)^2]\right\}&\text{in }Y\\0\leq W\leq C{\rm e}^{-\frac{\gamma}{\xi}}&\text{on }\p Y 
\end{cases}.
\end{equation}
Here $\gamma=\alpha\cdot\dist(\p Y,\o)$, $C$ and $\alpha$ given by \eqref{8.UepsCloseToaeps}.

By Step 1, taking $s=b^2$, for sufficiently small $\v$, we have for $x\in Y\setminus\o^\lambda$
\[
U^2_\v(y_{i,j}^\v+\delta^j x),V^2_\xi(x)\geq\max\left(b^2,\dfrac{1-b^2}{2}\right)\geq\dfrac{1}{3}.
\]
Thus, using the weak maximum principle, we find that $W\geq0$ in $Y$. Consequently, since $W$ is subharmonic, we deduce that $W\leq  C{\rm e}^{-\frac{\gamma}{\xi}}$.
\end{proof}

\subsection{Behavior of almost minimizers of $I_{\rho,\v}$}\label{S8.Intermediateresult}
We recall that for $x_0\in\R^2$ and $R>r>0$, we denoted $\Ring(x_0,R,r):=B(x_0,R)\setminus \overline{B(x_0,r)}$.

\subsubsection{Useful results for the periodic situation}
We establish three preliminary results for the periodic situation represented Figure \ref{Intro.FigureTermeChevillage}. Thus in this subsection we assume that $U_\v$ is {\bf the} unique global minimizer of $E_\v$ in $H^1_1$ with the periodic pinning term $a_\v$ represented Figure \ref{Intro.FigureTermeChevillage}.

\paragraph{Energetic estimates in rings and global energetic upper bounds}~\\

From Lemma \ref{L8.UIsAlmostPeriodic} ($U_\v$ is close to a periodic function) we obtain 
\begin{lem}\label{L8.AtomicPartEnergy}
For all $1\geq R>r\geq\v$, $x,x_0\in\R^2\text{ s.t. }B(x_0,R)\subset\O_\delta^{\rm incl}$ and $x-x_0\in\delta\cdot\Z^2$, we have
\[
\mu_\v(\Ring(x,R,r),1)\geq \mu_\v(\Ring(x_0,R,r),1)-o_\v(1).
\]
Adding the condition that $B(x,R)\subset\O_\delta^{\rm incl}$, we have
\[
\left|\mu_\v(\Ring(x,R,r),1)- \mu_\v(\Ring(x_0,R,r),1)\right|\leq o_\v(1).
\]
Moreover the $o_\v(1)$ may be considered independent of $x,x_0,R,r$.
\end{lem}
Lemma \ref{L8.AtomicPartEnergy} implies easily the following estimate.
\begin{prop}\label{P8.BoundOneAUsingAtomicDecomposition}
Let $\eta>0$ and  $\eta>\rho\geq\v$. Then there is $C=C(\O,\O',g,b,\eta)>0$ s.t.  for $x_0\in\R^2$ we have
\[
I_{\rho,\v}\leq d\,\mu_\v(\Ring(x_0,\eta,\rho),1)+C(\eta),
\]
where $C(\eta)$ is independent of $x_0$ and $\rho$.
\end{prop}

From Lemma \ref{L8.UIsAlmostPeriodic} we get the {almost periodicity \it} of $\mu_\v(\Ring(\cdot,R,r),1)$ w.r.t. a $\delta\times\delta$-grid (expressed in Lemma \ref{L8.UIsAlmostPeriodic}). Therefore, the "best points" to minimize $\mu_\v(\Ring(\cdot,R,r),1)$ should be almost periodic. 

Another important result is the next proposition. It expresses that the center of an inclusion is not too far to a good point to minimize $\mu_\v(\Ring(\cdot,R,r),1)$. This proposition may be seen as a first step in the proof of the pinning effect of $\o_\v$. 
\begin{prop}\label{P.LaProp.quifmqshfsqmdjfh}
There is $C_*$ which depends only on $\o$, $b$ and $\O$ s.t. for sufficiently small $\v$, for $x\in\O$ and $x_{\rm per}\in B(x,3\delta\sqrt2/2)\cap(\delta\Z\times\delta\Z)\cap\o_\v$ we have for $1>R>r>\v$
\begin{equation}\label{LaCléDuMySQtlkdsjgf}
\mu_\v(\Ring(x_{\rm per},R,r),1)\leq\mu_\v(\Ring(x,R,r),1)+C_*.
\end{equation} 
\end{prop}
\begin{proof}
If $R\leq 10^2r$, then the result is obvious with $C_*=2\pi\ln10$. Thus we assume that $R>10^2r$.

We share the proof in three cases: 
\begin{enumerate}[{Case} 1.]
\item $r \geq\delta$, 
\item $\delta\geq R>r\geq\lambda\delta$,
\item $R\leq\lambda\delta$.
\end{enumerate}
Assume for the moment that 
\begin{equation}\label{AssumptEstimateSpecialRiohdsflkjsqdh}
\text{There exists $\tilde{C_*}>0$ s.t. \eqref{LaCléDuMySQtlkdsjgf} yields  in the three previous cases with $C_*=\tilde{C_*}$.}
\end{equation}
For the general case, we divide $\Ring(x,R,r)$ into $R_1(x)\cup {R_2(x)}\cup R_3(x)$ with 
\[
R_1(x)=\Ring\left(x,R,\min\{\max(\delta,r),R\}\right),
\]
\[
R_2(x)=\Ring\left(x,\min\{\max(\delta,r),R\},\min\left\{R,\max(\lambda\delta,r)\right\}\right),
\]
\[
R_3(x)=\Ring\left(x,\min\left\{R,\max(\lambda\delta,r)\right\},r\right).
\] 
\begin{remark}
\begin{enumerate}\label{RemarkSurLaDecsdqkfhskdjfh}
\item For $k\in\{1,2,3\}$, we have $\emptyset\subseteq R_k\subseteq\Ring(x,R,r)$.
\item It is easy to check that
\begin{itemize}
\item[$\bullet$] $\left[R_1=\Ring(x,R,r)\Leftrightarrow r\geq \delta\text{ (Case 1.)}\right]$ and  $\left[R_1=\emptyset\Leftrightarrow R\leq\delta\right]$,
\item[$\bullet$] $\left[R_2=\Ring(x,R,r)\Leftrightarrow \lambda\delta\leq r<R\leq \delta\text{ (Case 2.)}\right]$\\\phantom{aaaaaaaaassssssssssaaa}and  $\left[R_2=\emptyset\Leftrightarrow \{\lambda=1\text{ or }r\geq\delta\text{ or }R\leq\lambda\delta\}\right]$,
\item[$\bullet$] $\left[R_3=\Ring(x,R,r)\Leftrightarrow R\leq \lambda\delta\text{ (Case 3.)}\right]$ and  $\left[R_3=\emptyset\Leftrightarrow r\geq\lambda\delta\right]$.
\end{itemize}
\item If $\lambda\equiv1$ then Case 2. never occurs and $R_2=\emptyset$.
\end{enumerate}
\end{remark}
Therefore we have (using Propositions \ref{P7.MyrtoRingDegDir}, \ref{P8.DirectPropAnnProb}.3 and \eqref{AssumptEstimateSpecialRiohdsflkjsqdh})
\[
\begin{array}{ccl}
\mu_\v(\Ring(x,R,r),1)&\geq&\mu_\v(R_1(x),1)+\mu_\v(R_2(x),1)+\mu_\v(R_3(x),1)
\\
\text{\eqref{AssumptEstimateSpecialRiohdsflkjsqdh}}&\geq&\mu_\v(R_1(x_{\rm per}),1)+\mu_\v(R_2(x_{\rm per}),1)+\mu_\v(R_3(x_{\rm per}),1)-3\tilde{C_*}
\\
(\text{Prop. \ref{P7.MyrtoRingDegDir}})&\geq&\mu^{\rm Dir}_\v(R_1(x_{\rm per}),1)+\mu^{\rm Dir}_\v(R_2(x_{\rm per}),1)+\mu^{\rm Dir}_\v(R_3(x_{\rm per}),1)-3(\tilde{C_*}+C_b)
\\
&\geq&\mu_\v(\Ring(x_{\rm per},R,r),1)-3(\tilde{C_*}+C_b).
\end{array}
\]
The last line is obtained by constructing a test function. Therefore, it suffices to take $C_*:=3(\tilde{C_*}+C_b)$.

We now turn to the proof of \eqref{AssumptEstimateSpecialRiohdsflkjsqdh} in Case 1, 2 and 3. 
 Recall that we assumed that $R> 10^2r$. 
 
We treat Case 1. ($R>r \geq\delta$):
\[
\begin{array}{rl}
\mu_\v(\Ring(x,R,r),1)\geq(\text{Prop. \ref{P7.MyrtoRingDegDir}})\geq&\mu^{\rm Dir}_\v(\Ring(x,R,r),1)-C_b
\\\geq&\mu^{\rm Dir}_\v(\Ring(x,10R,10^{-1}r),1)-2\pi\ln10-C_b
\\
\{\Ring(x_{\rm per},R,10r)\subset\Ring(x,10R,10^{-1}r)\}\geq&\mu_\v(\Ring(x_{\rm per},R,10r),1)-2\pi\ln10-C_b
\\\geq(\text{Prop. \ref{P7.MyrtoRingDegDir}})\geq&\mu_\v(\Ring(x_{\rm per},R,r),1)-3\pi\ln10-2C_b.
\end{array}
\]
Thus we may take $\tilde{C_*}=3\pi\ln10+2C_b$.

We treat Case 2. Note that from Remark \ref{RemarkSurLaDecsdqkfhskdjfh}.3, we may assume that $\lambda\to0$. On the one hand, it is clear that
\[
\mu_\v(\Ring(x_{\rm per},R,r),1)\leq\pi \ln\dfrac{R}{r}.
\]
On the other hand, letting 
\[
\begin{array}{cccc}
\alpha_\v:&\R^2&\to&\{b^2,1\}\\&x&\mapsto&\begin{cases}b^2&\text{if }x\in\cup_{M\in\Z^2}\overline{B(\delta M,\lambda\delta)}\\1&\text{otherwise}\end{cases}
\end{array},
\]
we have from Proposition \ref{P8.UepsCloseToaeps} that $\alpha_\v\leq U_\v^2+V_\v$ with $\|V_\v\|_{L^\infty}=o(\v^2)$.

If $\Ring(x,R,r)\cap \{\alpha_\v=b^2\}=\emptyset$, then we have $\mu_\v(\Ring(x,R,r),1)\geq\pi \ln\dfrac{R}{r}+o(\v^2\ln\lambda)$. And thus the result holds with $\tilde{C_*}=1$ (for sufficiently small $\v$).

Otherwise we have $\Ring(x,R,r)\cap \{\alpha_\v=b^2\}\neq\emptyset$. In this situation, because $R\leq\delta$, we get that $\Ring(x,R,r)\cap \{\alpha_\v=b^2\}$ is a union of at most four connected components. Therefore $S=\{\rho\in(r,R)\,|\,\p B(x,\rho)\cap\{\alpha_\v=b^2\}\}$ is a union of at most four segments whose length is lower than $8\lambda\delta$. Consequently, denoting  $\overline{S}=\cup_{i=1}^k [s_i,t_i]$ (with $s_i<s_{i+1}$), we have for $w_*\in H^1(\Ring(x,R,r),\S^1)$ which minimizes $\mu_\v(\Ring(x,R,r),1)$
\begin{eqnarray*}
\frac{1}{2}\int_{\Ring(x,R,r)}U_\v^2|\n w_*|^2+o_\v(1)&\geq&\frac{1}{2}\int_{\Ring(x,R,r)}\alpha_\v|\n w_*|^2
\\(t_0=r\,\&\,s_{k+1}=R)&\geq&\sum_{i=0}^{k}\frac{1}{2}\int_{t_i}^{s_{i+1}}\frac{{\rm d}\rho}{\rho}\int_{0}^{2\pi}|\p_\theta w_*|^2
\\&\geq&\pi\sum_{i=0}^{k}\ln\frac{s_{i+1}}{t_i}=\pi\ln\frac{R}{r}-\pi\sum_{i=1}^k\ln\frac{t_i}{s_i}
\end{eqnarray*}
Since $\lambda\delta\leq s_i\leq t_i\leq s_i+8\lambda\delta$, we have $1\leq\dfrac{t_i}{s_i}\leq1+\dfrac{8\lambda\delta}{s_i}\leq9$. Therefore we may take $\tilde{C_*}=4\pi\ln9+1$.

We treat the last case. Since $R\leq\lambda\delta$ and $x_{\rm per}\in (\delta\Z\times\delta\Z)\cap\o_\v$, there is $\tilde{C_*}$ s.t. we have (for sufficiently small $\v$) $\mu_\v(\Ring(x_{\rm per},R,r),1)\leq\pi b^2 \ln\dfrac{R}{r}+\tilde{C_*}$. On the other hand (Proposition \ref{P8.DirectPropAnnProb}.2) we have $\mu_\v(\Ring(x,R,r),1)\geq\pi b^2\ln\dfrac{R}{r}$. Therefore the estimate in the third case is proved.
\end{proof}

\paragraph{Estimates for almost minimizers}~\\

In this subsection we establish a fundamental result: fix an almost minimal configuration $\{{\bf x},{\bf 1}\}$ for $I_{\rho,\v}$ (the existence of such configuration is proved Section \ref{Part1Periodic}) and a map which {\it almost} minimizes $\frac{1}{2}\int_{\O'\setminus\cup\overline{B(x_i,\rho)}}U_\v^2|\n \cdot|^2$. Then the map {\it almost} minimizes the weighted Dirichlet functional $\frac{1}{2}\int_{\Ring(x_i,\rho',\rho)}U_\v^2|\n \cdot|^2$, $10^{-2}\min_{i\neq j}|x_i-x_j|>\rho'>\rho$.
\begin{lem}\label{L8.AlmostminimizingtestFunction}
\begin{enumerate}
\item Let $x\in\R^2$, $0<r<R$, $\alpha\in L^{\infty}(\R^2,[b^2,1])$, $C_0>0$ and a map $w\in H^1(\Ring(x,R,r),\S^1)$ s.t. $\deg_{\p B(x,R)}(w)=1$ and 
\[
\frac{1}{2}\int_{\Ring(x,R,r)}\alpha|\n w|^2-\mu_\alpha(\Ring(x,R,r),1) \leq C_0.
\]
Then for all $r',R'$ s.t. $r<r'<R'<R$ one has
\[
\frac{1}{2}\int_{\Ring(x,R',r')}\alpha|\n w|^2-\mu_\alpha(\Ring(x,R',r'),1)\leq 4C_b+C_0,
\]
where $C_b$ depends only on $b$ and is given by Proposition \ref{P7.MyrtoRingDegDir}.
\item Let $x_1,...,x_d\in\O$ ($x_i\neq x_j$ for $i\neq j$), $d_i=1$, $\v<\rho<10^{-2}\eta$, $\eta:=10^{-2}\cdot\min\left\{|x_i-x_j|,\dist(x_i,\p\O)\right\}$, $C_0>0$ and $w\in H^1(\O'_\rho,\S^1)$ s.t. 
\[
\frac{1}{2}\int_{\O'_\rho}U_\v^2|\n w|^2\leq I_{\rho,\v}+C_0.
\]
Then for  $\rho\leq r<R<\eta$ one has for all $i$
\[
\frac{1}{2}\int_{\Ring(x_i,R,r)}U_\v^2|\n w|^2-\mu_\v(\Ring(x_i,R,r),1)\leq C_0+C(\eta);
\]
here $C(\eta)$ depends only on $b,g,\O,\O'$ and $\eta$.
\item Under the hypotheses of 2., we also have for $\eta>\rho_0>\rho$
\[
\frac{1}{2}\int_{\O'_{\rho_0}}U_\v^2|\n w|^2\leq C(\rho_0,C_0);
\]
here $C(\rho_0,C_0)$ depends only on $b,g,\O,\O',C_0,\rho_0$ and $\eta$.
\end{enumerate}
\end{lem}
\begin{proof}
Using the third part of Proposition \ref{P8.DirectPropAnnProb}, we have
\begin{eqnarray*}
\frac{1}{2}\int_{\Ring(x,R,r)}\alpha|\n w|^2&\leq&\mu_\alpha(\Ring(x,R,R'),1)+\mu_\alpha(\Ring(x,R',r'),1)\\&&\phantom{aaaaaassss}+\mu_\alpha(\Ring(x,r',r),1)+4C_b+C_0.
\end{eqnarray*}
We easily obtain
\begin{eqnarray*}
\frac{1}{2}\int_{\Ring(x,R,r)}\alpha|\n w|^2&\geq&\mu_\alpha(\Ring(x,R,R'),1)+\frac{1}{2}\int_{\Ring(x,R',r')}\alpha|\n w|^2+\mu_\alpha(\Ring(x,r',r),1)
\end{eqnarray*}
which proves the first assertion.

The second assertion is obtained by using the same argument combined with Proposition \ref{P8.BoundOneAUsingAtomicDecomposition}.

Last assertion is a straightforward consequence of Proposition \ref{P8.BoundOneAUsingAtomicDecomposition} and both previous assertions.
\end{proof}
\subsubsection{Lower bound on circles}

In this subsection we prove an estimate for the minimization of weighted $1$-dimensional Dirichlet functionals. In the following this estimate will be used to get lower bounds in rings.
\begin{lem}\label{Lestimatesurdescercle}
Let $\theta_0\in(0,2\pi)$ and let $\alpha\in L^\infty([0,2\pi],\{b^2,1\})$ be s.t. $\H^1(\{\alpha=b^2\})=\theta_0$.

Let $\varphi\in H^1([0,2\pi],\R)$ s.t. $\varphi(2\pi)-\varphi(0)=2\pi$. The following lower bound holds
\[
\frac{1}{2}\int_0^{2\pi}{\alpha(\theta)|\p_\theta\varphi(\theta)|^2\,{\rm d}\theta}\geq\frac{2\pi^2}{\int_0^{2\pi}\frac{1}{\alpha}}=\frac{2\pi^2}{2\pi+\theta_0(b^{-2}-1)}.
\]
Here $\H^1$ is the $1$-dimensional Hausdorff measure.
\end{lem}
\begin{proof}
The proof of this lower bound is based on the computation of the minimal energy. 

It is easy to check that a minimal function $\varphi_{\rm min}\in H^1([0,2\pi],\R)$ for $\frac{1}{2}\int_0^{2\pi}{\alpha(\theta)|\p_\theta\cdot|^2\,{\rm d}\theta}$ under the constraint $\varphi(2\pi)-\varphi(0)=2\pi$ exists and satisfies $\p_\theta(\alpha \p_\theta\varphi_{\rm min})=0$. Thus $ \p_\theta\varphi_{\rm min}=\dfrac{{\rm Cst}}{\alpha}$ with ${\rm Cst}=\dfrac{2\pi}{\int_0^{2\pi}\alpha^{-1}}$. Therefore
\[
\frac{1}{2}\int_0^{2\pi}{\alpha(\theta)|\p_\theta\varphi(\theta)|^2\,{\rm d}\theta}\geq\frac{1}{2}\int_0^{2\pi}{\alpha(\theta)|\p_\theta\varphi_{\rm min}(\theta)|^2\,{\rm d}\theta}=\frac{2\pi^2}{\int_0^{2\pi}\frac{1}{\alpha}}=\frac{2\pi^2}{2\pi+\theta_0(b^{-2}-1)}.
\]
\end{proof}

\subsection{Proof of the first part of Proposition \ref{P8.ToMinimizeSecondPbThePointAreFarFromBoundAndHaveDegree1}}\label{Part1Periodic}

Let $x^n_1,...,x^n_N\in\O$ s.t. $|x^n_i-x^n_j|\geq8\rho$ and $d_1,...,d_N>0,\sum d_i=d$ (up to a subsequence the degrees may be considered independent of $n$). 

Assume that 
\begin{equation}\label{HypContrSeparationpoint}
\text{there is $i_0\in\{1,...,N\}$ s.t. $d_{i_0}\neq1$ or there are $i\neq j$ s.t. $|x_i^n-x_j^n|\to0$.}
\end{equation}
Up to pass to a subsequence, there are $a_1,...,a_M\in\overline{\O}$ and $\{\Lambda_1,...,\Lambda_M\}$ a partition of $\{1,...,N\}$ s.t.
\[
i\in\Lambda_l\Longleftrightarrow x_i^n\to a_l. 
\]
For sake of simplicity, we drop the superscript $n$ for the points, \emph{i.e.}, we write  $x_i$ instead of $x^n_i$.

We let $\rho_0:=10^{-2}\min\{\min_{k\neq l}|a_k-a_l|\,,\,\dist(\p\O,\p\O')\}$ with $\min_{k\neq l}|a_k-a_l|=+\infty$ if $M=1$.

Note that since $d_i>0$, \eqref{HypContrSeparationpoint} is equivalent to
\begin{equation}\label{NUMADDadd}
\text{there exists $l_0\in\{1,...,M\}$ s.t. $\displaystyle\tilde{d}_{l_0}=\sum_{i\in\Lambda_{l_0}}d_i>1$}.
\end{equation}
We are going to prove that \eqref{NUMADDadd} is not possible for almost minimal configurations. In order to do this, for $l\in\{1,...,M\}$, we obtain a lower bound for the weighted Dirichlet functional defined around $a_l$. Then using Proposition \ref{P8.BoundOneAUsingAtomicDecomposition} we will conclude.

For $l\in\{1,...,M\}$, there are two cases:
\begin{enumerate}
\item  ${\rm Card}(\Lambda_l)>1$, 
\item  ${\rm Card}(\Lambda_l)=1$.
\end{enumerate}
In the first case (${\rm Card}(\Lambda_l)>1$), we apply the separation process (defined  Section \ref{S8.SeparationProcess}) in $\O_l^n=B(a_l,2\rho_0)\setminus\cup_{i\in\Lambda_l}\overline{ B(x_i,\rho)}$ with $\eta_{\rm stop}=10^{-2}\rho_0$.
 
By construction, the process stops after $K$ steps. For $k\in\{1,...,K\}$ we denote:
\begin{itemize}
\item[$\bullet$] $\{x_1^k,...,x_{N_k}^k\}$ the selection of points made in Step $k$ ($x_i^0=x_i$, $i\in\Lambda_l$), 
\item[$\bullet$] $\eta'_k$  the radius of the intermediate balls in Step $k$ ($\eta'_k=\frac{1}{4}\min_{i\neq j}|x^{k-1}_i-x^{k-1}_j|$), 
\item[$\bullet$] $\eta_k$ the radius of the final balls in Step $k$ ($\eta_k=2\kappa_k\eta_k'$, $\kappa_k\in\{9^0,...,9^d\}$ and $\eta_0=\rho$).
\end{itemize}

Since for $i,j\in\Lambda_l$ we have $|x_i-x_j|\to0$, then, in the end of the process (after $K$ steps), we obtain a unique $x_1^K=y_l\in\{x_i\,|\,i\in\Lambda_l\}$ in the final selection of points and $\eta_K\to0$.

From \eqref{8.Ring} and \eqref{8.RingUnique}, the following rings are mutually disjoint (denoting $\eta_0=\rho$)
\begin{equation}\nonumber
R^l_0=\Ring(y_l,\rho_0,\eta_K)\text{ and }R_{j,k}=\Ring(x^k_j,\eta'_{k+1},\eta_{k})\text{ for }k\in\{0,...,K-1\},\,j\in\{1,...,N_k\}.
\end{equation}
We let
\begin{itemize}
\item[$\bullet$] for $k\in\{0,...,K-1\}$ and $j\in\{1,...,N_k\}$,  $\tilde{d}_{j,k}:=\sum_{x_i\in B(x^k_j,\eta'_{k+1})}d_i$,
\item[$\bullet$] for $n\geq 1$ we let $x_0=x_0(n)\in(\delta\Z\times\delta\Z)\cap\o_{\v_n}$ be s.t. $B(x_0,2\rho_0)\subset\O$. Thus combining Lemma \ref{L8.AtomicPartEnergy} with Proposition \ref{P.LaProp.quifmqshfsqmdjfh}, we get that (for sufficiently large $n$) and  for $\rho_0\geq R>r\geq\rho$
\begin{equation}\label{TheKeyEqForDescsdjhslkqdjhg}
\mu_{\v_n}(\Ring(x_0, R,r),1)\leq\inf_{x\in\O}\mu_{\v_n}(\Ring(x,R,r),1)+C_*+1.
\end{equation}
\end{itemize}
For $w\in H_g^1(\O'_{\rho},\S^1)$ we have 
\begin{align}\nonumber
\frac{1}{2}\int_{\O_l^n}U_{\v_n}^2|\n w|^2\geq\:&\:\frac{1}{2}\int_{R^l_0}U_{\v_n}^2|\n w|^2+\sum_{k=0}^{K-1}\sum_{j=1}^{N_k}\frac{1}{2}\int_{R_{j,k}}U_{\v_n}^2|\n w|^2
\\\nonumber
\geq\:&\:\frac{1}{2}\int_{R^l_0}U_{\v_n}^2|\n w|^2+\sum_{k=0}^{K-1}\sum_{j=1}^{N_k}\mu_{\v_n}(\Ring(x^k_j,\eta'_{k+1},\eta_{k}),\tilde{d}_{j,k})
\\\nonumber\text{\eqref{TheKeyEqForDescsdjhslkqdjhg}}\geq\:&\:\frac{1}{2}\int_{\Ring(x_0,\rho_0,\eta_{K})}U_{\v_n}^2|\n w|^2+\sum_{k=0}^{K-1}\sum_{j=1}^{N_k}\mu_{\v_n}(\Ring(x_0,\eta'_{k+1},\eta_{k}),\tilde{d}_{j,k})-\mathcal{O}(1)\\\nonumber
\text{\eqref{8.TheRadiusHaveTheSameOrder}}\geq\:&\:\tilde{d}_l^2\,\mu_{\v_n}(\Ring(x_0,\rho_0,\eta_{K}),1)+\sum_{k=0}^{K-1}\sum_{j=1}^{N_k}\tilde{d}_{j,k}\,\mu_{\v_n}(\Ring(x_0,\eta_{k+1},\eta_{k}),1)-\mathcal{O}(1)
\\\label{8;NotCloseBound0bis}\text{(Prop. \ref{P8.DirectPropAnnProb}.3)}\geq&\:\tilde{d}_l\,\mu_{\v_n}(\Ring(x_0,\rho_0,\rho),1)+(\tilde{d}_l^2-\tilde{d}_l)\pi b^2|\ln\eta_K|-\mathcal{O}(1).
\end{align}
In the second case (${\rm Card} (\Lambda_l)=1$) the computations are direct
\begin{align}\nonumber
\frac{1}{2}\int_{\O_l^n}U_{\v_n}^2|\n w|^2\:\geq&\:\frac{1}{2}\int_{\Ring(x_i,\rho_0,\rho)}U_{\v_n}^2|\n w|^2
\\\label{8;NotCloseBound0bisbis}\geq&\:\tilde{d}_l\,\mu_{\v_n}(\Ring(x_0,\rho_0,\rho),1)+(\tilde{d}_l^2-\tilde{d}_l)\pi b^2|\ln\rho|-\mathcal{O}(1).
\end{align}
Summing the lower bounds \eqref{8;NotCloseBound0bis} and \eqref{8;NotCloseBound0bisbis} over $l$ and applying Proposition \ref{P8.BoundOneAUsingAtomicDecomposition}, we obtain that if \eqref{NUMADDadd} occurs, then the configuration $\{{\bf x},{\bf d}\}$ cannot be almost minimal because $\eta_K,\rho\to0$ and $\tilde{d}_{l_0}>1$. Therefore \eqref{NUMADDadd} cannot occur for almost minimal configurations.

\subsection{Proof of the second part of Proposition \ref{P8.ToMinimizeSecondPbThePointAreFarFromBoundAndHaveDegree1}}\label{Part2Periodic}
We now prove the second part of Proposition \ref{P8.ToMinimizeSecondPbThePointAreFarFromBoundAndHaveDegree1}: we establish the repelling effect of $\p\O$ on the centers $x_i$'s.

Let $x_1^n,...,x_d^n\in\O$ and $\rho=\rho(\v_n)\downarrow0$ be s.t. $|x_i^n-x_j^n|\geq 8\rho$  ($i\neq j$) and $\dist(x_1^n,\p\O)\to0$. From the previous subsection we may assume that there is $\eta_0>0$ (independent of $n$) s.t.
\[
\min\left\{\min_{i\neq j}|x_i^n-x_j^n|\,,\,\dist(\O,\p\O')\right\}\geq10^2\eta_0.
\]
Up to pass to a subsequence, we may assume that $x_i^n\to a_i\in\overline{\O}$ with $a_i\neq a_j$ for $i\neq j$ and that $\eta=\max\{\sqrt{\dist(x_1^n,\p\O)},\rho\}\to0$.

For sake simplicity, we assume that for $i=2,...,d$ we have $a_i\in\O$. If this condition is not satisfied, then a direct adaptation of the following argument may be done. We assume that $\eta_0$ is s.t. for $i=2,...,d$ we have $\dist(x_i^n,\p\O)\geq 10^2\eta_0$.

We fix $x_0=x_0(\v_n)\in\O$ s.t. 
\[
x_0-x_1^n\in\delta\Z\times\delta\Z,\, \dist(x_0^n,\p\O)\geq 10^2\eta_0\text{ and }\min_{i=1,...,d}|x_0-x_j^n|\geq10^2\eta_0.
\]
We are going to prove that for $w\in H^1_g(\O'\setminus\cup_i\overline{B(x_i^n,\rho)},\S^1)$ we have
\begin{equation}\label{TheWonderfullContrdiction}
\frac{1}{2}\int_{\Ring(x_1^n,\sqrt\eta,\eta)}U_{\v_n}^2|\n w|^2-\mu_{\v_n}(\Ring(x_0,\sqrt\eta,\eta),1)\to\infty.
\end{equation}
\begin{remark}\label{TheRemarkForTherepEf}Estimate \eqref{TheWonderfullContrdiction} implies that $\{x_1^n,...,x_d^n\}$ can not be an almost minimal configuration of points. 

Indeed, we may construct a suitable test function $\tilde{w}$ as follows:
\begin{construction}The test function $\tilde{w}\in H^1_g(\O'\setminus(\overline{B(x_0,\rho)}\cup\cup_{i=2}^d\overline{B(x_i^n,\rho)}),\S^1)$\label{Constr.NearBoundaryTestFunction}
\begin{enumerate}[$\bullet$]
\item For $i=2,...,d$, we define $\tilde{w}_{|\Ring(x^n_i,\eta_0,\rho)}$ by taking a minimal map for $\frac{1}{2}\int_{\Ring(x^n_i,\eta_0,\rho)}U_{\v_n}^2|\n\cdot|^2$ in $H^1(\Ring(x^n_i,\eta_0,\rho),\S^1)$ with the boundary conditions $\tilde{w}(x_i^n+\eta_0{\rm e}^{\imath\theta})={\rm e}^{\imath\theta}$ and $\tilde{w}(x_i^n+\rho{\rm e}^{\imath\theta})={\rm Cst}_i{\rm e}^{\imath\theta}$, ${\rm Cst}_i\in\S^1$.
From Proposition \ref{P7.MyrtoRingDegDir} we have 
\[
\frac{1}{2}\int_{\Ring(x_i^n,\eta_0,\rho)}U_{\v_n}^2|\n\tilde{w}|^2\leq\mu_{\v_n}(\Ring(x_i^n,\eta_0,\rho),1)+C_b.
\]
\item We divide $\Ring(x_0,\eta_0,\rho)$ into $\Ring(x_0,\eta_0,\sqrt\eta)$, $\Ring(x_0,\sqrt\eta,\eta)$ and $\Ring(x_0,\eta,\rho)$. In each of these rings we consider the minimal maps for $\frac{1}{2}\int_{{\rm ring}}U_{\v_n}^2|\n\cdot|^2$ with the boundary conditions $\tilde{w}(x_0+R{\rm e}^{\imath\theta})={\rm e}^{\imath\theta}$ and $\tilde{w}(x_0+r{\rm e}^{\imath\theta})={\rm Cst}_i{\rm e}^{\imath\theta}$, ${\rm Cst}_i\in\S^1$ where ${\rm ring}\in\{\Ring(x_0,\eta_0,\sqrt\eta),\Ring(x_0,\sqrt\eta,\eta),\Ring(x_0,\eta,\rho)\}$, $r<R$ and ${\rm ring}=\Ring(x_0,R,r)$. 

Up to consider suitable rotations, we glue these functions to get a map $\tilde{w}_{|\Ring(x_0,\eta_0,\rho)}\in H^1(\Ring(x_0,\eta_0,\rho),\S^1)$ which is s.t. $\tilde{w}(x_0+\eta_0{\rm e}^{\imath\theta})={\rm e}^{\imath\theta}$ and (from Proposition \ref{P7.MyrtoRingDegDir})
\[
\frac{1}{2}\int_{\rm ring}U_{\v_n}^2|\n\tilde{w}|^2\leq\mu_{\v_n}({\rm ring},1)+C_b\,\text{}
\]
with ${\rm ring}\in\{\Ring(x_0,\eta_0,\sqrt\eta),\Ring(x_0,\sqrt\eta,\eta),\Ring(x_0,\eta,\rho)\}$.
\item We extend $\tilde w$ in $\O\setminus(\overline{B(x_0,\eta_0)}\cup\cup_{i=2}^d\overline{B(x_i^n,\eta_0)})$ using Lemma \ref{L8.UpperBoundS1ValuedMapInGoodCondition}.1. Then we finally obtain $\tilde{w}\in H^1_g(\O'\setminus(\overline{B(x_0,\rho)}\cup\cup_{i=2}^d\overline{B(x_i^n,\rho)}),\S^1)$.
\end{enumerate}
\end{construction}
From Lemma \ref{L8.AtomicPartEnergy}, \eqref{TheWonderfullContrdiction} and by construction of $\tilde{w}$, for $w_n\in H^1_g(\O'\setminus\cup_i\overline{B(x_i^n,\rho)},\S^1)$  we have easily that
\[
\int_{\O'\setminus\cup_i\overline{B(x_i^n,\rho)}}U_{\v_n}^2|\n w_n|^2-\int_{\O'\setminus(\overline{B(x_0,\rho)}\cup\cup_{i=2}^d\overline{B(x_i^n,\rho)})}U_{\v_n}^2|\n \tilde{w}|^2\to+\infty
\]
which implies that $\{x_1^n,...,x_d^n\}$ can not be an almost minimal configuration of points.

\end{remark}


We now turn to the proof of \eqref{TheWonderfullContrdiction}. We argue by contradiction and we assume that there is $w_*=w_*^{\v_n}\in H^1_g(\O'\setminus\cup_i\overline{B(x_i^n,\rho)},\S^1)$ s.t. 
\begin{equation}\label{TheWonderfullContrdictionBIS}
\frac{1}{2}\int_{\Ring(x_1^n,\sqrt\eta,\eta)}U_{\v_n}^2|\n w_*|^2\leq\mu_{\v_n}(\Ring(x_0,\sqrt\eta,\eta),1)+\mathcal{O}(1).
\end{equation}
In particular (using Lemma \ref{L8.AtomicPartEnergy}) we have 
\[
\frac{1}{2}\int_{\Ring(x_1^n,\sqrt\eta,\eta)}U_{\v_n}^2|\n w_*|^2=\mu_{\v_n}(\Ring(x_1^n,\sqrt\eta,\eta),1)+\mathcal{O}(1).
\]

The key ingredient to get a contradiction is the fact that the map $w_*$ is almost constant in the "half" ring $\Ring(x_1^n,\sqrt\eta,\eta)\setminus\O$.

By smoothness of $\O$, we may assume that the cone $K_{\sqrt\eta,\eta}:=\{x=x_1^n+\rho{\rm e}^{\imath \theta}\,|\,\theta\in[0,\pi/2],\, \eta\leq\rho\leq \sqrt\eta\}$ does not intersect $\O$: $K_{\sqrt\eta,\eta}\cap\O=\emptyset$.

We consider the map 
\[
w_0(x_1^n+r{\rm e^{\imath\theta}})=\begin{cases}{\rm e}^{\imath 4\theta}&\text{if }\theta\in[0,\pi/2]\\1&\text{otherwise}\end{cases},\,r>0
\]which is s.t. $w_0\in H^1(\Ring(x_1^n,\sqrt\eta,\eta),\S^1)$ and $\deg_{\p B(x_1^n,\sqrt\eta)}(w_0)=1$.

For $\tilde{d}\in\N^*$ (to be fixed later) we define the map $w_{\rm test}=w_*^{\tilde{d}}w_0\in H^1(\Ring(x_1^n,\sqrt\eta,\eta),\S^1)$ and $\deg_{\p B(x_1^n,\sqrt\eta)}(w_{\rm test})=\tilde{d}+1$.

Thus, we have 
\[
\frac{1}{2}\int_{\Ring(x_1^n,\sqrt\eta,\eta)}{U_{\v_n}^2|\n w_{\rm test}|^2}\geq\mu_{\v_n}(\Ring(x_1^n,\sqrt\eta,\eta),\tilde{d}+1)=(\tilde{d}+1)^2\mu_{\v_n}(\Ring(x_1^n,\sqrt\eta,\eta),1).
\]
On the other hand, letting $\varphi_*,\varphi_0:\Ring(x_1^n,\sqrt\eta,\eta)\to\R$ s.t. $w_*={\rm e}^{\imath \varphi_*}$ and $w_0={\rm e}^{\imath \varphi_0}$, (note that $\varphi_*,\varphi_0$ are locally defined and those gradients are globally defined and lie in $L^2(\Ring(x_1^n,\sqrt\eta,\eta),\R)$), we have (using \eqref{TheWonderfullContrdictionBIS}),
\begin{eqnarray*}
\frac{1}{2}\int_{\Ring(x_1^n,\sqrt\eta,\eta)}{U_{\v_n}^2|\n w_{\rm test}|^2}\!\!&=&\!\!\frac{1}{2}\int_{\Ring(x_1^n,\sqrt\eta,\eta)}{U_{\v_n}^2|\tilde{d}\n\varphi_*+ \n\varphi_0|^2}
\\\!\!&=&\!\!\frac{\tilde{d}^2}{2}\int_{\Ring(x_1^n,\sqrt\eta,\eta)}{U_{\v_n}^2|\n\varphi_*|^2}+\frac{1}{2}\int_{\Ring(x_1^n,\sqrt\eta,\eta)}{U_{\v_n}^2| \n\varphi_0|^2}+
\\\!\!&&\!\!\phantom{aaaaaaaaaaaaaaaaaaasssss}+\tilde{d}\int_{\Ring(x_1^n,\sqrt\eta,\eta)}{U_{\v_n}^2\n\varphi_*\cdot \n\varphi_0}
\\\!\!&\leq&\!\! \tilde{d}^2\mu_{\v_n}(\Ring(x_1^n,\sqrt\eta,\eta),1)+2\pi|\ln\eta|+
\\&&\phantom{aaaaaaaaaaaaaaaaa}+\tilde{d}\int_{\Ring(x_1^n,\sqrt\eta,\eta)}{U_{\v_n}^2\n\varphi_*\cdot \n\varphi_0}+\mathcal{O}(1).
\end{eqnarray*}
Since $w_*=g$ in $\Ring(x_1^n,\sqrt\eta,\eta)\setminus\O$ and $\|\n\varphi_0\|_{L^2(\Ring(x_1^n,\sqrt\eta,\eta)\cap\O)}=0$, we have (using Cauchy-Schwarz inequality)
\[
\int_{\Ring(x_1^n,\sqrt\eta,\eta)}{U_{\v_n}^2|\n\varphi_*||\n\varphi_0|}=\int_{\Ring(x_1^n,\sqrt\eta,\eta)\setminus\overline{\O}}{U_{\v_n}^2|\n\varphi_*||\n\varphi_0|}=\mathcal{O}(\sqrt{|\ln\eta|}).
\]
Therefore we obtain
\[
\tilde{d}^2\mu_{\v_n}(\Ring(x_1^n,\sqrt\eta,\eta),1)+2\pi|\ln\eta|+\mathcal{O}(\sqrt{|\ln\eta|})\geq(\tilde{d}+1)^2\mu_{\v_n}(\Ring(x_1^n,\sqrt\eta,\eta),1)
\]
which implies that $2\pi|\ln\eta|+\mathcal{O}(\sqrt{|\ln\eta|})\geq(2\tilde{d}+1)\mu_{\v_n}(\Ring(x_1^n,\sqrt\eta,\eta),1)\geq(2\tilde{d}+1)b^2\pi|\ln\eta|$. Clearly we obtain a contradiction taking $\tilde{d}>(2-b^2)/(2b^2)$.

Thus, by using Remark \ref{TheRemarkForTherepEf}, the second part of Proposition \ref{P8.ToMinimizeSecondPbThePointAreFarFromBoundAndHaveDegree1} is proved.
\subsection{Proof of the third part of Proposition \ref{P8.ToMinimizeSecondPbThePointAreFarFromBoundAndHaveDegree1}}\label{SubsectionProofjfhsldkjfhsldkjfh}
In this subsection,  we prove the third part of Proposition \ref{P8.ToMinimizeSecondPbThePointAreFarFromBoundAndHaveDegree1}: the attractive effect of the inclusions.
 
Assume that there exist $C_0>0$, sequences $\v_n,\rho\downarrow0$, $\rho=\rho(\v_n)\geq\v_n$ s.t. $\rho/(\lambda\delta)\to0$ and distinct points $x_1^n,...,x_d^n$,  satisfying
\begin{equation}\label{8.ContradictionThirdAssertionSecondAuxPb}
\inf_{\substack{w\in H_g^1(\O'_{\rho},\S^1)\\\deg_{\p B(x_i,{\rho})}(w)=1}}{\frac{1}{2}\int_{\O'_{\rho}}{U_{\v_n}^2|\n w|^2}}-I_{\rho,\v_n}\leq C_0.
\end{equation}
We denote ${\bf x}_n=(x_1^n,...,x_d^n)$. From the first and the second assertion, there exists $\eta_0>0$ (independent of $n$) s.t.
\[
\min\left\{\min_{i\neq j}|x_i^n-x_j^n|,\min_i\dist(x_i^n,\p\O)\right\}\geq10^2\cdot\eta_0>0.
\]
We want to prove that there is some $c>0$ s.t. for $i=1,...,d$ we have (for large $n$) $B(x_i^n,c\lambda\delta)\subset\o_{\v_n}$. 

To this end, we argue by contradiction  and we assume that either $x_1^n\notin\o_{\v_n}$ or $x_1^n\in\o_{\v_n}$ and $\dfrac{\dist(x_1^n,\p\o_{\v_n})}{\lambda\delta}\to0$. 

We are going to prove that letting $y_n\in\delta\cdot(\Z\times\Z)$ s.t. $x_1^n,y_n\in\overline{Y_{k,l}^\delta}$, then, 
\begin{equation}\label{ContradictAttrctionVortexPeriod}
 \widehat{\mathcal{I}}_{\rho,\v_n}({\bf x}_n,{\bf 1})-\widehat{\mathcal{I}}_{\rho,\v_n}((y_n,x^n_2,...,x_d^n),{\bf 1})\to\infty.
\end{equation}
Up to a subsequence, we may assume that $\lim_n\frac{\dist(x_1^n,\o_{\v_n})}{\lambda\delta}$ exists. We divide the proof into two steps: 
\begin{enumerate}[Step 1.]
\item if $x_1^n\notin{\o_{\v_n}}$ and $\displaystyle\frac{\dist(x_1^n,\o_{\v_n})}{\lambda\delta}\to c\in(0,\infty]$, then \eqref{ContradictAttrctionVortexPeriod} holds;
\item if $\displaystyle\frac{\dist(x_1^n,\p\o_{\v_n})}{\lambda\delta}\to0$, then \eqref{ContradictAttrctionVortexPeriod} holds.
\end{enumerate} 
We now prove Step 1. Assume that $x_1^n\notin{\o_{\v_n}}$, $\displaystyle\frac{\dist(x_1^n,\o_{\v_n})}{\lambda\delta}\to c\in(0,\infty]$ and
\begin{equation*}
\inf_{\substack{w\in H_g^1(\O'_{\rho},\S^1)\\\deg_{\p B(x_i,{\rho})}(w)=1}}{\frac{1}{2}\int_{\O'_{\rho}}{U_{\v_n}^2|\n w|^2}}-I_{\rho,\v_n}\leq C_0.
\end{equation*}
Denote $w_n$ a minimizer for $\widehat{\mathcal{I}}_{\rho,\v_n}({\bf x}_n,{\bf1})$ (see Proposition \ref{P8.ExistenceOfminInIJ}). Using Lemma \ref{L8.AlmostminimizingtestFunction}.2, for $\rho\leq r<R<\eta_0$, one has
\[
\frac{1}{2}\int_{B(x_1^n,R)\setminus\overline{B(x_1^n,r)}}U_{\v_n}^2|\n w_n|^2-\mu_{\v_n}(B(x_1^n,R)\setminus\overline{B(x_1^n,r)},1)\leq C_0+C(\eta_0).
\]
Let $\kappa\in(0,10^{-2}\cdot c)$ be s.t. $B(0,10^2\kappa)\subset\o\subset Y$ and $\dist(\o,\p Y)\geq10^2\kappa$.

From Lemma \ref{L8.AlmostminimizingtestFunction} (Assertions 2 and 3), we have
\begin{equation}\nonumber
\widehat{\mathcal{I}}_{\rho,\v_n}({\bf x}_n,{\bf 1})=\sum_{i=1}^d\mu_{\v_n}(\Ring(x^n_i,\eta_0,\rho),1)+\mathcal{O}(1)
\end{equation}
and
\begin{equation}\nonumber
\widehat{\mathcal{I}}_{\rho,\v_n}((y_n,x^n_2,...,x_d^n),{\bf 1})=\mu_{\v_n}(\Ring(y_n,\eta_0,\rho),1)+\sum_{i=2}^d\mu_{\v_n}(\Ring(x^n_i,\eta_0,\rho),1)+\mathcal{O}(1).
\end{equation}
Recall that $y_n\in\delta\cdot\Z^2$ is s.t. $x_1^n,y_n\in \overline{Y_{k,l}^\delta}$. Since $|x^n_1-y_n|\leq\delta$, using  Lemma  \ref{L8.UpperBoundS1ValuedMapInGoodCondition}.2 and Propositions \ref{P7.MyrtoRingDegDir}, \ref{P8.DirectPropAnnProb}.3, we have
\[
\mu_{\v_n}(\Ring(y_n,\eta_0,\rho),1)=\mu_{\v_n}(\Ring(x_1^n,\eta_0,\delta),1)+\mu_{\v_n}(\Ring(y_n,\kappa\delta,\rho),1)+\mathcal{O}(1)
\]
Therefore
\begin{equation}\label{BlaBliEst1}
 \begin{array}{c}\widehat{\mathcal{I}}_{\rho,\v_n}({\bf x}_n,{\bf 1})-\widehat{\mathcal{I}}_{\rho,\v_n}((y_n,x^n_2,...,x_d^n),{\bf 1})\phantom{cccccccccccc}\phantom{cccccccccccc}
 \\\phantom{cccccccccccc}\phantom{cccccccccccc}=\mu_{\v_n}(\Ring(x_1^n,\kappa\delta,\rho),1)-\mu_{\v_n}(\Ring(y_n,\kappa\delta,\rho),1)+\mathcal{O}(1).
 \end{array}
\end{equation}
Thus it suffices to estimate the energies in the rings with radii $\kappa\delta$ and $\rho$. We have (using \eqref{8.UepsCloseToaeps})
\begin{equation}\label{BlaBliEst2}
\mu_{\v_n}(\Ring(y_n,\kappa\delta,\rho),1)=\pi|\ln\lambda|+b^2\pi\ln\frac{\lambda\delta}{\rho}+\mathcal{O}(1).
\end{equation}
In order to estimate $\mu_{\v_n}(\Ring(x_1^n,\kappa\delta,\rho),1)$, we divide  the argument according to the asymptotic of $\lambda$. If $\lambda\equiv1$, then $c\in(0,\infty)$ and thus $\dist(B(x_i,c\delta/3,\o_{\v_n})\geq c\delta/3$. Consequently, from Proposition \ref{P8.UepsCloseToaeps}, we have
\[
\mu_{\v_n}(\Ring(x_1^n,\kappa\delta,\rho),1)=\pi\ln\frac{\delta}{\rho}+\mathcal{O}(1).
\]
Therefore \eqref{ContradictAttrctionVortexPeriod} holds.

If $\lambda\to0$, we let $\chi=\begin{cases}\dfrac{c}{2}\lambda\delta&\text{if }c<\infty\\\kappa\lambda\delta&\text{otherwise}\end{cases}$ and $\eta=\dist(x_1^n,\p\o_{\v_n})$. Note that $\dfrac{\eta+2\lambda\delta}{\eta-\chi}=\mathcal{O}(1)$ and that $U_{\v_n}=1+V_n$ in $\Ring(x_1^n,\kappa\delta,\eta+2\lambda\delta)\cup\Ring(x_1^n,\eta-\chi,\rho)$, $\|V_n\|_{L^\infty}=o(\v_n^2)$ (from \eqref{8.UepsCloseToaeps}).

Thus we obtain
\begin{eqnarray}\nonumber
\mu_{\v_n}(\Ring(x_1^n,\kappa\delta,\rho),1)&\geq&\pi\ln\frac{\delta}{\eta+2\lambda\delta}+\pi b^2\ln\frac{\eta+2\lambda\delta}{\eta-\chi}+\pi\ln\frac{\eta-\chi}{\rho}+\mathcal{O}(1)
\\\label{BlaBliEst3}&=&\pi\ln\frac{\delta}{\rho}+\mathcal{O}(1).
\end{eqnarray}
Therefore if $c\in(0,\infty]$, then \eqref{BlaBliEst3} holds. Estimates \eqref{BlaBliEst1},\eqref{BlaBliEst2} and \eqref{BlaBliEst3} contradict \eqref{8.ContradictionThirdAssertionSecondAuxPb} (because \eqref{ContradictAttrctionVortexPeriod} holds).

\vspace{2mm}
We now turn to Step 2. Arguing as in Step 1., it suffices to prove that
\begin{equation}\label{8.ConeDivToinfty}
\mu_{\v_n}(\Ring(x^n_1,\kappa\delta,\rho),1)-\mu_{\v_n}(\Ring(y_n,\kappa\delta,\rho),1),1)\to\infty\text{ for some fixed $\kappa$}.
\end{equation}
(And $y_n\in\delta\cdot\Z^2$ s.t. $x^n_i,y_n\in\overline{Y_{k,l}^\delta}$)

We let $\kappa>0$ (depending only on $\o$) be s.t. 
\[
\text{$\kappa<10^{-2}\cdot\dist(\o,\p Y)$  and $B(0,10^2\cdot\kappa)\subset \o$}.
\]
In order to prove \eqref{8.ConeDivToinfty}, we divide the annular $\Ring(x_1^n,\kappa\delta,\rho)$ into three regions :
\[
\Ring(y_n,\kappa\delta,\rho)=\Ring(x^n_1,\kappa\delta,\kappa\lambda\delta)\cup\overline{\Ring(x^n_1,\kappa\lambda\delta,r_n)}\cup\Ring(x^n_1,r_n,\rho)
\]
with 
\[
r_n=\max\left\{\v_n^{1/4},\rho,\sqrt{\lambda\delta\cdot\dist(x_n,\p\o_\v)}\right\}+\sqrt{\v_n}.
\]
We are going to prove that $\mu_{\v_n}(\Ring(x^n_1,\kappa\lambda\delta,r_n),1)$ is too large.

We consider $K_n$ the cone of vertex $x_1^n$ and aperture ${\pi}/{2}$ which admits the line $(x_1^n,\Pi_{\p\o_{\v_n}}x_1^n)$ for symmetry axis and s.t. $ K_n\cap\o_{\v_n}\cap \Ring(x_1^n,\kappa\lambda\delta,r_n)=\emptyset$. Here $\Pi_{\p\o_{\v_n}}(x_1^n)$ is the orthogonal projection of $x_1^n$ on $\p\o_{\v_n}$. 

Note that since $\displaystyle\frac{\dist(x_1^n,\o_{\v_n})}{\lambda\delta}\to0$, for large $n$  and small $\kappa$  (independently of $n$), by smoothness of $\o$, $K_n$ is well defined (see Figure \ref{F8.Kone}).
\begin{figure}
\psset{xunit=7.7cm,yunit=7.7cm,algebraic=true,dotstyle=o,dotsize=3pt 0,linewidth=0.8pt,arrowsize=3pt 2,arrowinset=0.25}
\begin{pspicture*}(-10.8,5)(-9.0,6)
\psaxes[labelFontSize=\scriptstyle,xAxis=true,yAxis=true,Dx=0.1,Dy=0.1,ticksize=-2pt 0,subticks=2]{->}(0,0)(-10.8,5)(-9.5,6)

\pscustom[linecolor=black,fillstyle=vlines]{\parametricplot{0.7367471492768088}{2.4260991204521316}{1*0.79*cos(t)+0*0.79*sin(t)+-10.11|0*0.79*cos(t)+1*0.79*sin(t)+5.17}\lineto(-10.11,5.17)\closepath}
\pscustom[linecolor=black,fillstyle=solid]{\parametricplot{0.7367471492768242}{2.424661120120924}{1*0.39*cos(t)+0*0.39*sin(t)+-10.11|0*0.39*cos(t)+1*0.39*sin(t)+5.17}\lineto(-10.11,5.17)\closepath}
\psline{|-|}(-10.16,5.12)(-10.75,5.63)
\psline{|-|}(-10.06,5.11)(-9.77,5.37)

\rput{-1.53}(-10.12,1.11){\psellipse(0,0)(5.88,4.12)}
\rput{-1.53}(-10.12,1.11){\psellipse[linecolor=black,fillstyle=hlines](0,0)(5.81,4.08)}

\rput(-10.46,5.3){\psframebox*{$\kappa\lambda\delta$}}
\rput(-9.85,5.19){\psframebox*{$r_n$}}
\rput(-10.7,5.15){\psframebox*{$\o_{\v_n}$}}
\rput(-10.1,5.75){\psframebox*{$\Ring(x^n_1,\kappa\lambda\delta,r_n)\cap K_n$}}
\psline{|-|}(-9.25,5.175)(-9.255,5.13)
\rput(-9.15,5.15){\psframebox*{$\geq\sqrt{\v_n}$}}
\psarc[linecolor=black]{->}(-10.11,5.17){1}{40}{140}
\rput(-10.11,5.1){\psframebox*{$x_1^n$}}
\rput(-10.11,5.17){$\bullet$}
\rput(-10.11,5.34){${\pi}/{2}$}

\end{pspicture*}

\caption{The domain $\Ring(x^n_1,\kappa\lambda\delta,r_n)\cap K_n$}\label{F8.Kone}
\end{figure}
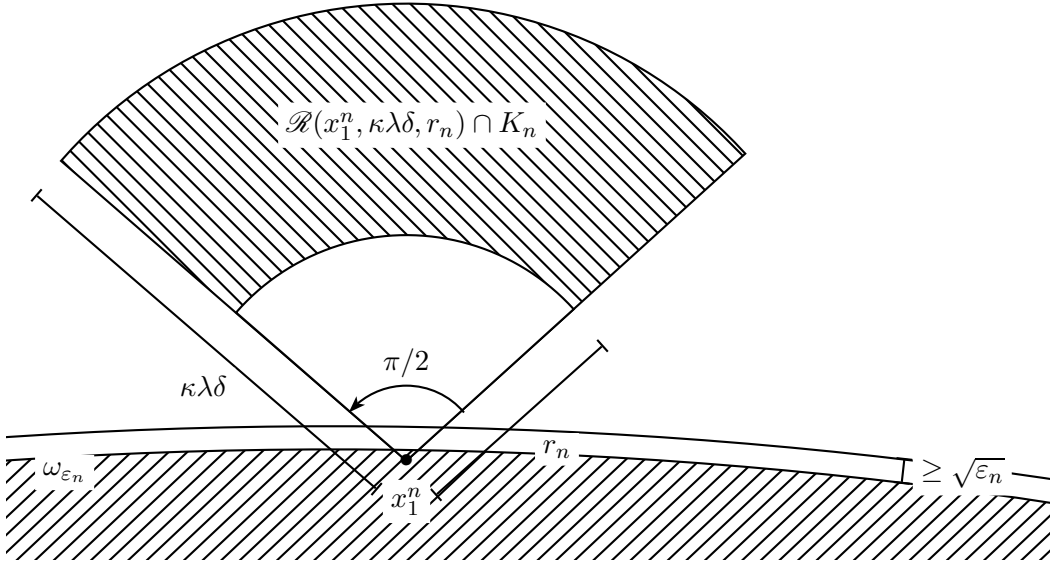

We have $U_{\v_n}=1+V_n$ in $\Ring(x^n_1,\kappa\lambda\delta,r_n)\cap K_n$ where, $\|V_n\|_{L^\infty}=o(\v_n^2)$. Thus, if we define $\alpha_n=\begin{cases}1&\text{in }K_n\\b^2&\text{otherwise}\end{cases}$, then, from Lemma \ref{Lestimatesurdescercle} with $\theta_0=3\pi/2$, for $w\in H^1(\Ring(x^n_1,\kappa\lambda\delta,r_n),\S^1)$ s.t. $\deg_{\p B(x_1^n,r_n)}(w)=1$,  we have
\[
\frac{1}{2}\int_{\Ring(x^n_1,\kappa\lambda\delta,r_n)}{\alpha_n|\n w|^2}\geq b^2\frac{4\pi}{b^2+3}\ln\dfrac{\kappa\lambda\delta}{r_n}.
\]
Clearly, from construction, $U_{\v_n}^2\geq \alpha_n+o(\v_n^2)$, thus if $w_n$ is a minimal map for $ \widehat{\mathcal{I}}_{\rho,\v_n}({\bf x}_n,{\bf 1})$, then we have
\begin{equation}\nonumber
\frac{1}{2}\int_{\Ring(x^n_1,\kappa\lambda\delta,r_n)}{U^2_{\v_n}|\n w_n|^2}\geq b^2\frac{4\pi}{b^2+3}\ln\dfrac{\kappa\lambda\delta}{r_n}+o_{n}(1).
\end{equation}
Now the computations are direct
\begin{eqnarray}\nonumber
\mu_{\v_n}(\Ring(x^n_1,\kappa\delta,\rho),1)&=&\mu_{\v_n}(\Ring(x^n_1,\kappa\delta,3\lambda\delta),1)+\mu_{\v_n}(\Ring(x^n_1,\kappa\lambda\delta,r_n),1)\\\nonumber&&\phantom{aaaaaaaaaaaadddddddddddd}+\mu_{\v_n}(\Ring(x^n_1,r_n,\rho),1)+\mathcal{O}(1)
\\\label{MisterCAca2}
&\geq&\pi|\ln\lambda|+b^2\frac{4\pi}{b^2+3}\ln\dfrac{\lambda\delta}{r_n}+b^2\pi\ln\dfrac{r_n}{\rho}+\mathcal{O}(1).
\end{eqnarray}
Therefore, \eqref{8.ConeDivToinfty} is a direct consequence of \eqref{BlaBliEst2} and \eqref{MisterCAca2} since $\dfrac{\lambda\delta}{r_n}\to+\infty$.

\section{Proof of Propositions \ref{P.RenormalizedBBHEnergy} and \ref{P8.CorolSecondAuxPb}}\label{S.ProofOfRenEnergyBBH}
We now prove the results specific to th pinning terms with dilution. We begin these proofs by their key ingredient.
\subsection{An important effect of the dilution of inclusions}
We first state a result which establishes that a "sufficiently large" circle has a small intersection with $\o_\v$ if $\lambda\to0$.
\begin{lem}\label{L.LargeCircleSmallInclusionLemma}
\begin{enumerate}We denote $C_\rho$ a circle with radius $\rho$.
\item Assume that the pinning term is periodic. Once $\lambda\leq1/8\pi$, for $\rho\geq\delta/3$ we have $\H^1(C_\rho\cap\o_\v)\leq16\pi^2\,\lambda\rho$.
\item Assume that the pinning term is not periodic, and recall that the inclusions with size $\lambda\delta^j$ is $\o_\v^j=\cup_{i\in\mathcal{M}^\v_j}{\{y_{i,j}^\v+\lambda\delta^j\cdot\o\}}$, $j\in\{1,...,\num\}$.  
\\Once $\lambda\leq1/8\pi$, for $\rho\geq\delta^j/3$ we have $\H^1(C_\rho\cap\o^j_\v)\leq16\pi^2\,\lambda\rho$.
\end{enumerate}
Here $\H^1(C_\rho\cap\o_\v)$ is the $1$-dimensional Hausdorff measure of $C_\rho\cap\o_\v$.
\end{lem}
\begin{proof}

In order to unify the notations, we fix $j=1$ if we are in the periodic case (and $j\in\{1,...,\num\,|\,M^\v_j\in\N^*\}$ if we are in the non-periodic case).  

Assume that $C_\rho\cap\o_\v\neq\emptyset$ and let 
\[
S_j:=\begin{cases}\left\{\tilde{Y}_\v=(\delta k,\delta l)+\delta\cdot Y\,\left|\begin{array}{c}\,k,l\in\Z^2,\,{\tilde{Y}_\v}\subset\O\\\text{ and $\tilde{Y}_\v\cap C_\rho\neq\emptyset$}\end{array}\right.\right\}&\text{in the periodic case}
\\
\\
\left\{\tilde{Y}_\v=B(y_{i,j}^\v,\delta^{j})\,\left|\begin{array}{c}y_{i,j}^\v\in\M_j^\v\text{ and $\tilde{Y}_\v\cap C_\rho\neq\emptyset$}\end{array}\right.\right\}&\text{in the non-periodic case}
\end{cases}.
\] 

For $\tilde{Y}_\v\in S_j$, we denote 
\begin{itemize}
\item[$\bullet$]$\tilde{\o}_\v$ {\bf the} connected component of $\o_\v$ which is included in $\tilde{Y}_\v$
\item[$\bullet$] $z_j$ the center of $\tilde{Y_\v}= \begin{cases}z_1+\delta\cdot Y,\,z_1\in\delta\cdot\Z^2&\text{in the periodic case}\\B(z_j,\delta^j),\, z_j\in\mathcal{M}_j^\v&\text{in the non-periodic case}\end{cases}$. 
\end{itemize}

We first treat the case where $C_\rho\subset \tilde{Y}_\v\in S_j$: since $\rho\geq\delta^j/3$ and $\tilde{\o}_\v\subset B(z_j,\lambda\delta^j)$ (because $\o\subset B(0,1)$), we have
\[
\H^1(C_\rho\cap\tilde{\o}_\v)=\H^1(C_\rho\cap\o^j_\v)\leq\H^1(\p B(z_j,\lambda\delta^j))=2\pi\lambda\delta^j\leq{6\pi}\lambda\rho.
\]
Otherwise, for $\tilde{Y}_\v\in S_j$, $C_\rho\nsubseteq\tilde{Y}_\v$ and thus
\[
\H^1(C_\rho\cap\tilde{\o}_\v)\leq \H^1(C_\rho\cap \overline{B(y^\v_{i,j},\lambda\delta^j)})\leq2\pi \lambda\delta^j\text{ (because $\o\subset B(0,1)$)}
\]
and
\[
\H^1(C_\rho\cap\tilde{Y}_\v\setminus\tilde{\o}_\v)\geq\delta^j\cdot\left(\frac{1}{2}-2\pi \lambda\right).
\]
Last estimate comes from the fact that $C_\rho\nsubseteq \tilde{Y}_\v$. Thus $\H^1(C_\rho\cap\tilde{Y}_\v\setminus\tilde{\o}_\v)$ is at least a radius of $\tilde{Y}_\v$ (if we are in the non periodic or half of side length of $\tilde{Y}_\v$ otherwise) minus the previous upper bound.
Thus we obtain (for $\lambda\leq1/8\pi$)
\[
\H^1(C_\rho\cap\tilde{\o}_\v)\leq 2\pi \lambda\frac{\H^1(C_\rho\cap\tilde{Y}_\v\setminus\tilde{\o}_\v)}{\dfrac{1}{2}-2\pi \lambda}\leq 8\pi \lambda\H^1(C_\rho\cap\tilde{Y}_\v\setminus\tilde{\o}_\v).
\]
Consequently,
\begin{eqnarray*}
\H^1(C_\rho\cap\o_\v)&=&\sum_{\tilde{Y}_\v\in S_j}\H^1(C_\rho\cap\tilde{\o}_\v)
\\&\leq&8\pi \lambda\sum_{\tilde{Y}_\v\in S_j}\H^1(C_\rho\cap\tilde{Y}_\v\setminus\tilde{\o}_\v)\leq 8\pi \lambda \H^1(C_\rho)=16\pi^2\lambda\rho.
\end{eqnarray*}
\end{proof}
\subsection{Proof of Proposition \ref{P.RenormalizedBBHEnergy}}
We are now in position to prove Proposition \ref{P.RenormalizedBBHEnergy}. The proof is done in 3 steps. 

Let $\v_n\downarrow0$, $\rho=\rho(\v_n)\downarrow0,\,\rho\geq\v_n$ and let ${\bf x}_n$ be a quasi-minimizer for $J_{\rho,\v_n}$ (defined Notation \ref{NOTATIONQUASIIN}). 

From Corollaries \ref{C8.CorolSecondAuxPb} $\&$ \ref{C8.AnAlmostMinConfigIsAnAlmostMinConf}, up to pass to a subsequence, there are $\eta_0>0$ and ${\bf a}=(a_1,...,a_d)\in\O^d$ s.t. $x_i^n\to a_i$, $|a_i-a_j|,\dist(a_i,\p\O)>10^2\eta_0$.

We prove that $W_g(a_1,...,a_d)=\min_{b_1,...,b_n\in\O}W_g(b_1,...,b_n)$. We argue by contradiction and we assume that, up to consider a smaller value for $\eta_0$ if necessary, we have the existence of ${\bf b}=(b_1,...,b_d)\in\O^d$ s.t. $|b_i-b_j|\geq10^2\eta_0$, $\dist(b_i,\p\O)>10^2\eta_0$ and
 
\[
W_g({\bf b})<W_g({\bf a})-10^{2}\eta_0.
\]
\vspace{2mm}
{\bf Step 1.} We estimate the energies in perforated domains with a fixed perforation size\vspace{2mm}

The goal of this step is to prove the existence of small $\rho_0$ (independent of $n$) s.t. we have for ${\bf c}\in\{{\bf a},{\bf b}\}$ and ${\bf x}\in\O^d\text{ satisfying }\max_i|x_i-c_i|\leq\rho_0$
\begin{equation}\label{8.Step1ComparisonAux}
\hat{\J}_{\rho_0,\1}({\bf x})-\hat{\J}_{\rho_0,\v_n}({\bf x})\leq2\eta_0.
\end{equation}
From \cite{CM1} ((15) and Lemma 2), we may fix $\eta_0>\rho_0>0$ independent of $n$ s.t. for ${\bf c}\in\{{\bf a},{\bf b}\}$, we have
\[
\hat{\J}_{\rho_0,\1}({\bf x})-\hat{\I}_{\rho_0,\1}({\bf x})\leq \eta_0 \text{ for all }{\bf x}\in\O^d\text{ s.t. }\max_i|x_i-c_i|\leq\rho_0,
\] 
\[
\left|\hat{\I}_{\rho_0,\1}({\bf x})-\pi d|\ln\rho_0|-W_g({\bf x})\right|\leq \eta_0 \text{ for all }{\bf x}\in\O^d\text{ s.t. }\max_i|x_i-c_i|\leq\rho_0
\]
and 
\[
|W_g({\bf c})-W_g({\bf x})|\leq \eta_0 \text{ for all }{\bf x}\in\O^d\text{ s.t. }\max_i|x_i-c_i|\leq\rho_0.
\]
For ${\bf c}\in\{{\bf a},{\bf b}\}$ and ${\bf x}\in\O^d\text{ s.t. }\max_i|x_i-c_i|\leq\rho_0$:
\begin{enumerate}[$\bullet$]
\item We let $\theta_{\bf x}=\sum_{i=1}^d\theta_{x_i}$ where $\theta_{x_i}\in(-\pi,\pi]$, $\dfrac{x-x_i}{|x-x_i|}={\rm e}^{\imath\theta_{x_i}}$ ($x\neq x_i$) is main determination of the argument of $x-x_i$.
\item We fix $\phi_0^{\bf x}\in C^\infty(\p\O,\R)$ s.t. ${\rm e}^{\imath\phi_0^{\bf x}}=g{\rm e}^{-\imath\theta_{\bf x}}$. Clearly, since $\deg_{\p\O}(g{\rm e}^{-\imath\theta_{\bf x}})=0$, and since $g{\rm e}^{-\imath\theta_{\bf x}}\in C^\infty(\p\O,\S^1)$, $\phi_0^{\bf x}\in C^\infty(\p\O,\R)$ is well defined \cite{BBH1}.
\item We let $\phi_*=\phi_*^{{\bf x}}, \phi=\phi^{{\bf x}}\in H^1$ be {\bf the} solutions of 
\begin{equation*}
\begin{cases}
-\Delta\phi_*=0&\text{in }\O\setminus\cup \overline{B(x_i,\rho_0)}
\\
\phi_*=\phi_0&\text{on }\p\O
\\
\p_\nu\phi_*=-\sum_{j\neq i}\p_\nu\theta_{x_j}&\text{on }\p B(x_i,\rho_0),\,i=1,...,d
\end{cases}
\end{equation*}
and
\begin{equation*}
\begin{cases}
-\Div(U_\v^2\n\phi)=\Div(U_\v^2\n\theta_{\bf x})&\text{in }\O\setminus\cup \overline{B(x_i,\rho_0)}
\\
\phi=\phi_0&\text{on }\p\O
\\
\p_\nu\phi=-\sum_{j\neq i}\p_\nu\theta_{x_j}&\text{on }\p B(x_i,\rho_0),\,i=1,...,d
\end{cases}.
\end{equation*}
\item We let $\psi=\phi-\phi_*$ be {\bf the} solution of 
\begin{equation*}
\begin{cases}
-\Div(U_\v^2\n\psi)=\Div[(U_\v^2-1)(\n\theta_{\bf x}-\n\phi_*)]&\text{in }\O\setminus\cup \overline{B(x_i,\rho_0)}
\\
\psi=0&\text{on }\p\O
\\
\p_\nu\psi=0&\text{on }\p B(x_i,\rho_0),\,i=1,...,d
\end{cases}.
\end{equation*}
\end{enumerate}
\begin{remark}
\begin{enumerate}
\item From Proposition \ref{P8.ExistenceOfminInIJ}, the functions $\phi_*,\phi$ are s.t. $w_*={\rm e}^{\imath (\theta_{\bf x}+\phi_*)},w={\rm e}^{\imath (\theta_{\bf x}+\phi_*)}\in  \I_{\rho_0}({\bf x})$ satisfy
\[
\hat{\I}_{\rho_0,\1}({\bf x})=\frac{1}{2}\int_{\O\setminus\cup \overline{B(x_i,\rho_0)}}|\n w_*|^2=\frac{1}{2}\int_{\O\setminus\cup \overline{B(x_i,\rho_0)}}|\n (\theta_{\bf x}+\phi_*)|^2
\]
and
\[
\hat{\I}_{\rho_0,\v}({\bf x})=\frac{1}{2}\int_{\O\setminus\cup \overline{B(x_i,\rho_0)}}U_\v^2|\n w|^2=\frac{1}{2}\int_{\O\setminus\cup \overline{B(x_i,\rho_0)}}U_\v^2|\n (\theta_{\bf x}+\phi)|^2.
\]
\item $\n\phi$ and $\n\phi_*$ are bounded independently of ${\bf x}$ and $\v_n$ in $L^2(\O\setminus\cup \overline{B(x_i,\rho_0)})$.
\item From a Poincaré inequality we have the existence of $C_0$ independent of ${\bf x}$ s.t. 
\[
\|\psi\|_{L^2(\O\setminus\cup \overline{B(x_i,\rho_0)})}\leq C_0 \|\n\psi\|_{L^2(\O\setminus\cup \overline{B(x_i,\rho_0)})}.
\] Therefore, using a trace inequality in $\Ring(x_i,2\rho_0,\rho_0)$ we obtain $\|\psi\|_{L^2(\p{B(x_i,\rho_0)})}\leq {C}_0'$, ${C}_0'$ is independent of ${\bf x},n$.

\item We have $|\n\phi_*|$ which is bounded in $L^\infty(\O\setminus\cup \overline{B(x_i,\rho_0)})$:
\begin{equation*}
|\n\phi_*|\leq C_0\text{ with $C_0$ independent of ${\bf x}$.}
\end{equation*}
Indeed, with standard result of elliptic interior regularity, we have 
\[
\|\phi_*\|_{C^2(\p B(x_i, {8\rho_0}))},\|\phi_*\|_{C^2(\p B(c_i, {4\rho_0}))}\leq C_0'.
\]
Thus, from global regularity for the Laplacian, we have 
\[
\|\n\phi_*\|_{L^\infty(\O\setminus\cup\overline{B(c_i,4\rho_0)})},\|\n\phi_*\|_{L^\infty(\Ring(x_i,8\rho_0,\rho_0))}\leq C_0''.
\]
\end{enumerate}
\end{remark}
We let $\O_{\rho_0}=\O_{\rho_0}({\bf x}):=\O\setminus\cup \overline{B(x_i,\rho_0)}$. We are now in position to prove that $\displaystyle \int_{\O_{\rho_0}}|\n\psi|^2\to0$ when $n\to\infty$ uniformly on ${\bf x}$. This estimate will easily imply \eqref{8.Step1ComparisonAux}. Indeed 
\begin{eqnarray*}
\hat{\I}_{\rho_0,\1}({\bf x})-\hat{\I}_{\rho_0,{\v_n}}({\bf x})&=&\frac{1}{2}\int_{\O_{\rho_0}}U_{\v_n}^2\left[|\n (\theta_{\bf x}+\phi_*)|^2-|\n (\theta_{\bf x}+\phi)|^2\right]\\&&\phantom{aaaaaaaaaaqqqqqqq}+\frac{1}{2}\int_{\O_{\rho_0}}(1-U_{\v_n}^2)|\n (\theta_{\bf x}+\phi_*)|^2
\\\left(\begin{array}{c}\text{Cauchy-Schwarz}\\\text{inequality}\end{array}\right)&\leq& \tilde{C_0}\left(\|\n\psi\|_{L^2(\O_{\rho_0})}+\|1-U_{\v_n}^2\|_{L^2(\O_{\rho_0})}\right)\to0.
\end{eqnarray*}

Consequently we obtain
\[
\hat{\J}_{\rho_0,\1}({\bf x})-\hat{\J}_{\rho_0,{\v_n}}({\bf x})\leq \hat{\I}_{\rho_0,\1}({\bf x})-\hat{\I}_{\rho_0,{\v_n}}({\bf x})+\eta_0\leq\eta_0+o_n(1)\leq2\eta_0
\]
which is exactly \eqref{8.Step1ComparisonAux}.

Thus it remains to establish that  $\displaystyle \int_{\O_{\rho_0}}|\n\psi|^2\to0$ when $n\to\infty$ uniformly on ${\bf x}$:
\begin{eqnarray}\nonumber
\int_{\O_{\rho_0}}U_{\v_n}^2|\n\psi|^2&=&\int_{\O_{\rho_0}}\Div[(U_{\v_n}^2-1)(\n\theta_{\bf x}-\n\phi_*)]\psi
\\\nonumber&=&\int_{\O_{\rho_0}}(1-U_{\v_n}^2)(\n\theta_{\bf x}-\n\phi_*)\cdot\n\psi+\int_{\p\O_{\rho_0}}(U_{\v_n}^2-1)\p_\nu(\theta_{\bf x}-\phi_*)\psi.
\end{eqnarray}
From the $L^2$ bound on $\n\psi$ and the $L^\infty$ bounds on $\n\phi_*,\n\theta_{\bf x}$ we have (with $C_0$ independent of ${\bf x}$)
\begin{eqnarray}\nonumber
\int_{\O_{\rho_0}}U_{\v_n}^2|\n\psi|^2&\leq&\left(\int_{\O_{\rho_0}}|1-U_{\v_n}^2|^2|\n\theta_{\bf x}-\n\phi_*|^2\right)^{1/2}\left(\int_{\O_{\rho_0}}|\n\psi|^2\right)^{1/2}+\\\nonumber&&\phantom{aaaaaaaaaaa}+\left(\int_{\p\O_{\rho_0}}(U_{\v_n}^2-1)^2|\p_\nu(\theta_{\bf x}-\phi_*)|^2\right)^{1/2}\left(\int_{\p\O_{\rho_0}}|\psi|^2\right)^{1/2}
\\\nonumber&\leq&C_0\left(\|1-U_{\v_n}^2\|_{L^2(\O_{\rho_0})}+\|1-U_{\v_n}^2\|_{L^2(\p\O_{\rho_0})}\right).
\end{eqnarray}
From Proposition \ref{P8.UepsCloseToaeps} and Lemma \ref{L.LargeCircleSmallInclusionLemma} we have $\|1-U_{\v_n}^2\|_{L^2(\p\O_{\rho_0})}=\mathcal{O}(\lambda)$ uniformly in ${\bf x}$. 

Therefore $\displaystyle \int_{\O\setminus\cup \overline{B(x_i,\rho_0)}}|\n\psi|^2\to0$ when $n\to\infty$ uniformly on ${\bf x}$ and \eqref{8.Step1ComparisonAux} holds.
\vspace{2mm}
\\{\bf Step 2.} We study the energies in $\Ring(x_i,\rho_0,\max(\delta,\lambda^2))$\vspace{2mm}

Let 
\begin{enumerate}[$\bullet$]
\item $\kappa=\max(\lambda,\sqrt\delta)$ 
\item ${\bf x}_n$ be a quasi minimizer for $J_{\rho,\v}$
\item $w_n={\rm e}^{\imath\varphi_n}$ be a minimizer of $\hat{\J}_{\rho,\v_n}({\bf x}_n)$ ($\varphi_n$ is locally defined and its gradient is globally defined in $\O\setminus\cup\overline{B(x_i,\rho)}$).
\end{enumerate}
We prove that there is $r\in(\kappa^2,\kappa)$ s.t. 
\begin{equation}\label{TGVEqua}
\frac{1}{2}\int_0^{2\pi}|\p_\theta\varphi_n(x_i^n+r{\rm e}^{\imath\theta})|^2{\rm d}\theta\leq \pi+\frac{1}{\sqrt{|\ln\kappa|}}\text{ for }i=1,...,d.
\end{equation}
This estimate is obtained via a mean value argument. We first prove that
\[
\mu_{\v_n}(\Ring(x_i^n,\kappa,\kappa^2),1)=\mu_{\1}(\Ring(x_i^n,\kappa,\kappa^2),1)+o_{\v_n}(1).
\]
Indeed we let $\o'$ be a smooth open set s.t. $\overline{\o}\subset\o'$ and $\overline{\o'}\subset B(0,1)$. We define $\alpha_\v'=\begin{cases}b^2&\text{in }\delta\Z\times\delta\Z+\lambda\delta\cdot\o'\\1&\text{otherwise}\end{cases}$. From Proposition \ref{P8.UepsCloseToaeps}, we have $\alpha_\v'\leq U_\v^2+V_\v$ with $\|V_\v\|_{L^\infty}=\mathcal{O}(\v^2)$. 

For $\rho\geq\delta$ and $x\in\R^2$, from Lemma \ref{L.LargeCircleSmallInclusionLemma}, we have $\H^1[\{\alpha'_\v=b^2\}\cap \p B(x,\rho)]\leq 16\pi^2\lambda\rho$. Therefore, using Lemma \ref{Lestimatesurdescercle} we obtain
\begin{eqnarray*}
\mu_{\1}(\Ring(x_i^n,\kappa,\kappa^2),1)+\mathcal{O}(\lambda|\ln\kappa|)&\leq (\text{Lemma \ref{Lestimatesurdescercle}})\leq&\mu_{\alpha'_{\v_n}}(\Ring(x_i^n,\kappa,\kappa^2),1)
\\&\leq\text{($\alpha_\v'\leq U_\v^2+V_\v$)}\leq&\mu_{\v_n}(\Ring(x_i^n,\kappa,\kappa^2),1)+o_{\v_n}(1)
\\&\leq\text{($U_\v^2\leq1$)}\leq&\mu_{\1}(\Ring(x_i^n,\kappa,\kappa^2),1)+o_{\v_n}(1).
\end{eqnarray*}
Since for $ s \in(\kappa^2,\kappa)$ we have $ s \geq\delta$, we obtain (because $\kappa\geq\lambda$)
\[
\mu_{\v_n}(\Ring(x_i^n,\kappa,\kappa^2),1)=\mu_{\1}(\Ring(x_i^n,\kappa,\kappa^2),1)+\mathcal{O}(\lambda|\ln\kappa|)=\pi|\ln\kappa|+o_{\v_n}(1).
\]
Therefore  from Corollary \ref{C8.AnAlmostMinConfigIsAnAlmostMinConf} and Lemma \ref{L8.AlmostminimizingtestFunction}.2, $\displaystyle\frac{1}{2}\int_{\Ring(x_i^n,\kappa,\kappa^2)}U_{\v_n}^2|\n w_n|^2=\pi|\ln\kappa|+\mathcal{O}(1)$. On the other hand, from a standard estimate, we have 
\[
\frac{1}{2}\int_0^{2\pi}|\p_\theta\varphi_n(x_i^n+ s {\rm e}^{\imath\theta})|^2{\rm d}\theta\geq\pi,\:\forall \, s \in(\kappa^2,\kappa).
\]
We deduce that
\begin{equation}\nonumber
\pi d|\ln\kappa|+\mathcal{O}(1)\geq\frac{1}{2}\int_{\cup \Ring(x_i^n,\kappa,\kappa^2)}|\n w_n|^2\geq\frac{1}{2}\int_{\kappa^2}^\kappa\frac{{\rm d} s }{ s }\sum_i\int_0^{2\pi}|\p_\theta\varphi_n(x_i^n+ s {\rm e}^{\imath\theta})|^2{\rm d}\theta.
\end{equation}
Assume that $r\in(\kappa^2,\kappa)$ s.t. \eqref{TGVEqua} holds does not exist. Then we obtain that for $ s \in(\kappa^2,\kappa)$
\[
\sum_i\frac{1}{2}\int_0^{2\pi}|\p_\theta\varphi_n(x_i^n+ s {\rm e}^{\imath\theta})|^2{\rm d}\theta>\pi d+\frac{1}{\sqrt{|\ln\kappa|}}
\]
and consequently 
\[
\frac{1}{2}\int_{\cup \Ring(x_i^n,\kappa,\kappa^2)}|\n w_n|^2\geq|\ln\kappa|\left(\pi d+\frac{1}{\sqrt{|\ln\kappa|}}\right)=\pi d|\ln\kappa|+\sqrt{|\ln\kappa|}.
\]
Clearly this lower bound contradicts $\displaystyle\frac{1}{2}\int_{\Ring(x_i^n,\kappa,\kappa^2)}U_{\v_n}^2|\n w_n|^2=\pi|\ln\kappa|+\mathcal{O}(1)$.
\vspace{2mm}
We are now in position to estimate the energy in $\Ring(x^n_i,\rho_0,r)$. Let $h_i^n:\S^1\to\S^1,\, h_i^n({\rm e}^{\imath\theta})=w_n(x_i^n+r{\rm e}^{\imath\theta})$. We have $h_i^n\times\p_\tau \left[h_i^n({\rm e}^{\imath\theta})\right]=\p_\tau\left[\varphi_n(x_i^n+r{\rm e}^{\imath\theta})\right]$. 

Thus from \eqref{TGVEqua}: $\|h_i^n\times\p_\tau h_i^n\|^2_{L^2(\S^1)}\leq2\pi +{2}/{\sqrt{|\ln\kappa|}}$. Consequently 
\[
\int_{\S^1}{|h_i^n\times\p_\tau h_i^n-1|^2}=\int_{\S^1}{\left\{|h_i^n\times\p_\tau h_i^n|^2+1-2h_i^n\times\p_\tau h_i^n\right\}}\leq{2}/{\sqrt{|\ln\kappa|}}\to0.
\]
Therefore $h_i^n\times\p_\tau h_i^n\to1$ in $L^2(\S^1)$. Consequently, up to pass to a subsequence, we have the existence of $\alpha_i\in\S^1$ s.t. $\alpha_i^{-1}h_i^n{\rm e}^{-\imath\theta}\to1$ in $H^1(\S^1)$.

From Propositions 12 and 13 in \cite{publi3} we have
\begin{eqnarray*}
\inf_{\substack{w\in H^1(\Ring(x_i^n,\rho_0,r),\S^1)\\w(x_i^n+\rho_0{\rm e}^{\imath\theta})=\alpha_i{\rm e}^{\imath\theta}\\w(x_i^n+r{\rm e}^{\imath\theta})=h_i^n({\rm e}^{\imath\theta})}}\frac{1}{2}\int_{\Ring(x_i^n,\rho_0,r)}|\n w|^2&=&\inf_{\substack{w\in H^1(\Ring(x_i^n,\rho_0,r),\S^1)\\w(x_i^n+\rho_0{\rm e}^{\imath\theta})=\alpha_i{\rm e}^{\imath\theta}\\w(x_i^n+r{\rm e}^{\imath\theta})=\alpha_i{\rm e}^{\imath\theta}}}\frac{1}{2}\int_{\Ring(x_i^n,\rho_0,r)}|\n w|^2+o_n(1)
\\&=&\pi\ln\frac{\rho_0}{r}+o_n(1).
\end{eqnarray*}
\vspace{2mm}
{\bf Step 3.} We conclude

We are going to construct a map $\tilde{w}_n\in\J_{\rho}({\bf y}_n)$, $\max|y_i-b_i|\leq \delta$ and s.t.
\begin{equation}\label{ContradictionConstructionMap}
\int_{\O\setminus\cup\overline{B(y_i,\rho)}}U_{\v_n}^2|\n \tilde{w}_n|^2+\eta_0\leq\hat\J_{\rho,\v_n}({\bf x}_n).
\end{equation}
Clearly \eqref{ContradictionConstructionMap} is in contradiction with the assumption:  $J_{\rho,\v_n}-\hat\J_{{\rho,\v_n}}({\bf x}_n)\to0$. Then this contradiction will imply that ${\bf a}=\lim {\bf x_n}$ minimizes $W_g$. 

We let ${\bf y}_n$ be s.t. $\max|y^n_i-b_i|\leq\delta$ and $x_i^n-y_i^n\in\delta\Z\times\delta\Z$ and we define 
\[
\tilde{w}_n(x)=\begin{cases}
w_{\rho_0}^{{\bf y}_n}(x)&\text{if }x\in\O\setminus\cup\overline{B(y^n_i,\rho_0)} 
\\
{\rm Cst}_{i,n}w^i(x-y_i^n+x_i^n)&\text{if } x\in \Ring(y^n_i,\rho_0,r)
\\{\rm Cst}_{i,n}w_n[x-y_i^n+x_i^n]&\text{if }x\in \Ring(y^n_i,r,\rho)
\end{cases}
\]
Here:
\begin{enumerate}[$\bullet$]
\item $w_{\rho_0}^{{\bf y}_n}$ is a minimizer of $\hat{\J}_{\rho_0,\1}({\bf y}_n)$,
\item $w^i$ is a minimizer of $\displaystyle \inf_{\substack{w\in H^1(\Ring(x_i^n,\rho_0,r),\S^1)\\w(x_i^n+\rho_0{\rm e}^{\imath\theta})=\alpha_i{\rm e}^{\imath\theta}\\w(x_i^n+r{\rm e}^{\imath\theta})=h_i^n({\rm e}^{\imath\theta})}}\frac{1}{2}\int_{\Ring(x_i^n,\rho_0,r)}|\n w|^2$
\item ${\rm Cst}_{i,n}\in\S^1$ is a constant s.t. $\tilde{w}_n\in H^1(\O\setminus\cup\overline{B(y^n_i,\rho)},\S^1)$
\item $w_n$ is the minimizer of $\hat{\J}_{\rho,\v_n}({\bf x}_n)$ used in Step 2..
\end{enumerate}

We now compare the energies of $\tilde{w}_n$ and $w_n$.
\begin{eqnarray*}
\int_{\O\setminus\cup\overline{B(y^n_i,\rho)}}U_{\v_n}^2|\n\tilde{w}_n|^2&=&\int_{\O\setminus\cup\overline{B(y^n_i,\rho_0)}}U_{\v_n}^2|\n\tilde{w}_n|^2+\int_{\cup_i\Ring(y^n_i,\rho_0,r)}U_{\v_n}^2|\n\tilde{w}_n|^2+
\\&&\phantom{aaaaaaaaaaaaaaaaqqqqqqqqqqq}+\int_{\cup_i\Ring(y^n_i,r,\rho)}U_{\v_n}^2|\n\tilde{w}_n|^2.
\end{eqnarray*}
From Step 1. (the definition of $\rho_0$ and Estimate \eqref{8.Step1ComparisonAux}), we have 
\begin{eqnarray*}
\frac{1}{2}\int_{\O\setminus\cup\overline{B(y^n_i,\rho_0)}}U_{\v_n}^2|\n\tilde{w}_n|^2&\leq&\pi d|\ln\rho_0|+W_g({\bf y}_n)+\eta_0+o_n(1)\\&\leq&\pi d|\ln\rho_0|+W_g({\bf x}_n)-10\eta_0
\\&\leq& \frac{1}{2}\int_{\O\setminus\cup\overline{B(x^n_i,\rho_0)}}U_{\v_n}^2|\n{w}_n|^2-2\eta_0.
\end{eqnarray*}
From Step 2., letting  $\alpha_\v'=\begin{cases}b^2&\text{in }\delta\Z\times\delta\Z+\lambda\delta\cdot\o'\\1&\text{otherwise}\end{cases}$, we have
\begin{eqnarray*}
\frac{1}{2}\int_{\cup_i\Ring(y^n_i,\rho_0,r)}U_{\v_n}^2|\n\tilde{w}_n|^2&=\text{(Step 2.)}=&\pi d\ln\frac{\rho_0}{r}+o_n(1)
\\
&\leq\text{(Lem. \ref{Lestimatesurdescercle} $\&$ \ref{L.LargeCircleSmallInclusionLemma})}\leq&\frac{1}{2}\int_{\cup_i\Ring(x^n_i,\rho_0,r)}\alpha'|\n{w}_n|^2+o_n(1)
\\
&\leq\text{($\alpha_\v'\leq U_\v^2+V_\v$)}\leq&\frac{1}{2}\int_{\cup_i\Ring(x^n_i,\rho_0,r)}U_{\v_n}^2|\n{w}_n|^2+o_n(1).
\end{eqnarray*}
From Lemma \ref{L8.UIsAlmostPeriodic}
\[
\int_{\cup_i\Ring(y^n_i,r,\rho)}U_{\v_n}^2|\n\tilde{w}_n|^2=\int_{\cup_i\Ring(x^n_i,r,\rho)}U_{\v_n}^2|\n{w}_n|^2+o_n(1).
\]

Therefore we obtain \eqref{ContradictionConstructionMap} and consequently Proposition \ref{P.RenormalizedBBHEnergy} holds.
\subsection{Proof of Proposition \ref{P8.CorolSecondAuxPb}}\label{S.PorrofGenDiltoihefkqsdjf}
The strategy to prove Proposition \ref{P8.CorolSecondAuxPb} is the following:
\begin{enumerate}[{ Step} 1.]
\item We let $\kappa=\max(\lambda,\delta)$. We first characterize almost minimal configurations for $I_{\kappa,\v}$ (\emph{i.e} the domain $\O$ is perforated by discs with radius $\kappa$).
\item We make the description of almost minimal points $(x_\v)_\v$ for  $\mu_\v(\Ring(\cdot,\kappa,\lambda\delta^{3/2}),1)$.
\item We estimate  $\inf_{x_0\in\R^2}\mu_\v(\Ring(x_0,\lambda\delta^{3/2},\rho),1)$ and we conclude.
\end{enumerate}
\vspace{2mm}
{\bf Step 1.} We study almost minimal configurations for $I_{\kappa,\v}$, $\kappa=\max(\lambda,\delta)$ 
\vspace{2mm}\\
We prove that $\{{\bf x},{\bf d}\}=\{(x_1^\v,d_1),...,(x_N^\v,d_N)\}$ is an almost minimal configuration for $I_{\kappa,\v}$ {\bf if and only if} $N=d$, $d_i=1$ and there is $\eta_0>0$ s.t. $\dist(x_i^\v,\p\O),|x^\v_i-x^\v_j|\geq\eta_0$.

First note that for $\eta_0>0$ and $x^\v_1,...,x^\v_d\in\O$ s.t. $\dist(x_i,\p\O),|x^\v_i-x^\v_j|\geq\eta_0$ we have easily
\begin{equation}\label{TheBounqlskfhjqsldjkhf:;;}
I_{\kappa,\v}\leq\hat{\I}_{\kappa,\v}({\bf x},{\bf d})\leq \pi d|\ln\kappa|+C(\eta_0)
\end{equation}
with $C(\eta_0)$ which is independent of $\v$.

We consider $\{{\bf x},{\bf d}\}$ which is  almost minimal for $I_{\kappa,\v}$. We argue as in the proof of Proposition \ref{P8.ToMinimizeSecondPbThePointAreFarFromBoundAndHaveDegree1} (Assertions 1 and 2, see Subsections \ref{Part1Periodic} $\&$ \ref{Part2Periodic}). We use the separation process defined Subsection \ref{S8.SeparationProcess} and the associated natural partition of $\O_\kappa:=\O\setminus\cup\overline{B(x_i^\v,\kappa)}$ (see Subsection \ref{S8.SeparationProcessGive}). 

Here the key ingredients are Lemmas \ref{Lestimatesurdescercle} $\&$ \ref{L.LargeCircleSmallInclusionLemma} (which replace the periodic structure of the pinning term). Combining both lemmas we get that if $R>r\geq\kappa$, then
\[
\mu_\v(\Ring(x_0,R,r),1)=\pi\ln\frac{R}{r}+\mathcal{O}(\lambda\ln\frac{R}{r}).
\]
The rings $\Ring(x_0,R,r)$ which occur in the partition of $\O_\kappa$ are all s.t. $C(\O)\geq R>r\geq\kappa$ and thus $\dfrac{R}{r}=\mathcal{O}(\kappa^{-1})$. Which infer that $\mathcal{O}(\lambda\ln\frac{R}{r})=o_\v(1)$ and consequently $\mu_\v(\Ring(x_0,R,r),1)=\pi\ln\frac{R}{r}+o_\v(1)$

Therefore we get: If $\{{\bf x},{\bf d}\}$ is an almost minimal configuration for $I_{\kappa,\v}$ then $N=d$, $d_i=1$ and there is $\eta_0>0$ s.t. $\dist(x_i^\v,\p\O),|x^\v_i-x^\v_j|\geq\eta_0$. This is proved by contradiction exactly as in Subsections \ref{Part1Periodic} $\&$ \ref{Part2Periodic} and using \eqref{TheBounqlskfhjqsldjkhf:;;}. 

Moreover, if $\{{\bf x},{\bf d}\}$ is an almost minimal configuration for $I_{\kappa,\v}$, then the arguments of Subsections \ref{Part1Periodic} $\&$ \ref{Part2Periodic} in conjunction with  \eqref{TheBounqlskfhjqsldjkhf:;;}, yield $|\hat{\I}_{\kappa,\v}({\bf x},{\bf d})-\pi d|\ln\kappa||\leq C(\eta_0)$. Here $\eta_0$ is obtained in the previous paragraph. Therefore  we get $I_{\kappa,\v}=\pi d|\ln\kappa|+\mathcal{O}(1)$.

Conversely, from  \eqref{TheBounqlskfhjqsldjkhf:;;}, for $\eta_0>0$ and $x^\v_1,...,x^\v_d\in\O$ s.t. $\dist(x_i^\v,\p\O),|x^\v_i-x^\v_j|\geq\eta_0$, we have $\{x^\v_1,...,x^\v_d\}$ which is almost minimal for $I_{\kappa,\v}$.
\vspace{2mm}
\\{\bf Step 2.} We study almost minimal configurations for $\mu_\v(\Ring(\cdot,\kappa,\lambda\delta^{3/2}),1)$
\vspace{2mm}

For $j\in\{1,...,\num\}$, we denote: $\o_\v^j:=\cup_{i\in\M_j^\v}\{{y_{i,j}^\v+\lambda\delta^j\o}\}$. And recall that the set of centers of connected components of $\o_\v^j$ is $\WM_j^\v:=\{y_{i,j}^\v\,|\,i\in\mathcal{M}_j^\v\}$.

Letting $x^0_\v\in\o_\v$  et $c>0$ (independent of $\v$) s.t.  $B(x^0_\v,c\lambda\delta)\subset\o^1_\v$, on the one hand  we may easily prove that
\begin{equation}\label{InClmlksjdfhqskdljfhDiff1}
\mu_\v(\Ring(x^0_\v,\delta,\lambda\delta^{3/2}),1)= \pi b^2|\ln\delta^{1/2}|+\pi|\ln\lambda|+\mathcal{O}(1),
\end{equation}
and on the other hand, applying Lemmas \ref{Lestimatesurdescercle} $\&$ \ref{L.LargeCircleSmallInclusionLemma}, we have
\begin{equation}\label{InClmlksjdfhqskdljfhDiff2}
\mu_\v(\Ring(x^0_\v,\kappa,\delta),1)=\pi[1+\mathcal{O}(\lambda)]\ln\frac{\kappa}{\delta}.
\end{equation}
Therefore, from \eqref{InClmlksjdfhqskdljfhDiff1} and \eqref{InClmlksjdfhqskdljfhDiff2}, we get
\begin{equation}\label{InClmlksjdfhqskdljfhDiff3}
\mu_\v(\Ring(x^0_\v,\kappa,\lambda\delta^{3/2}),1)=\pi \left[\frac{b^2}{2}+1+\mathcal{O}(\lambda)\right]|\ln\delta|+\pi\ln\frac{\kappa}{\lambda}+\mathcal{O}(1).
\end{equation}
We are going to prove that  this situation ($B(x^0_\v,c\lambda\delta)\subset\o^1_\v$) is the only way to get the minimal energy. More precisely we prove that for a fixed constant $C_0>0$, if we have $(x_\v)_\v\subset\O$ which is s.t. 
\begin{equation}\label{InClmlksjdfhqskdljfhDiff4}
\mu_\v(\Ring(x_\v,\kappa,\lambda\delta^{3/2}),1)\leq\inf_{x_0\in\O}\mu_\v(\Ring(x_0,\kappa,\lambda\delta^{3/2}),1)+C_0,
\end{equation}
then there is $c>0$ independent of $\v$ s.t. for sufficiently small $\v$ we have $B(x_\v,c\lambda\delta)\subset\o_\v$, \emph{i.e.} $B(x_\v,c\lambda\delta)\subset y_{i_\v,1}^\v+\lambda\delta\o$ with $y_{i_\v,1}^\v\in\WM_1^\v$.

We let $C_0>0$ and $(x_\v)_\v\subset\O$ s.t. \eqref{InClmlksjdfhqskdljfhDiff4} holds.

Up to pass to a sequence $\v_n\downarrow0$, dropping the subscript $n$ (we write $\v$ instead of $\v_n$), we may assume that one of these cases occurs
\begin{itemize}
\item[Case 0.] $\exists\,c>0$ s.t. $B(x_\v,c\lambda\delta)\subset\o_\v$,
\item[Case 1.] $x_\v\notin \cup_{j=1}^\num \cup_{i\in\mathcal{M}_{j}^\v}B(y_{i,j}^\v,\delta^j)$,
\item[Case 2.] $x_\v\in\cup_{j=2}^\num \cup_{i\in\mathcal{M}_{j}^\v}B(y_{i,j}^\v,\delta^j)$,
\item[Case 3.] $\{x_\v\in \cup_{i\in\mathcal{M}_{1}^\v}B(y_{i,1}^\v,\delta)\setminus\overline{\o_\v^1}\}$ or  $\{x_\v\in\o_\v^1\,\&\,\dist(x_\v,\p\o_\v^1)/\lambda\delta\to0\}$.
\end{itemize}
We want to prove that only Case 0. occurs if \eqref{InClmlksjdfhqskdljfhDiff4} holds.
\\{\bf Case 1.} From Lemmas \ref{Lestimatesurdescercle} $\&$ \ref{L.LargeCircleSmallInclusionLemma}, it is direct to prove that
\[
\mu_\v(\Ring(x_\v,\kappa,\lambda\delta^{3/2}),1)\geq \pi\left[1+\mathcal{O}(\lambda)\right]\ln\frac{\kappa}{\lambda\delta^{3/2}}=\pi\left[\frac{3}{2}+\mathcal{O}(\lambda)\right]|\ln\delta|+\pi\ln\frac{\kappa}{\lambda}.
\]
Using \eqref{InClmlksjdfhqskdljfhDiff3} we get
\[
\mu_\v(\Ring(x_\v,\kappa,\lambda\delta^{3/2}),1)-\inf_{x_0\in\O}\mu_\v(\Ring(x_0,\kappa,\lambda\delta^{3/2}),1)\to+\infty.
\]
Therefore, if $(x_\v)_\v$ satisfies \eqref{InClmlksjdfhqskdljfhDiff4}, then Case 1. does not occur.
\\{\bf Case 2.} We let $j_0\in\{2,...,\num\}$  be s.t. $x_\v\in \cup_{i\in\mathcal{M}_{j_0}^\v}B(y_{i,j_0},\delta^{j_0})$. We define $\kappa':=\max\{\delta^{j_0},\lambda\delta^{3/2}\}$ and we denote $y_0=y_{i,j_0}^\v\in\WM_{j_0}$ be s.t. $x_\v\in B(y_0,\delta^{j_0})$.

We first assume that $x_\v\notin\overline{\o_\v}$ and we let $\textsf{\ae}=\max\{\lambda\delta^{3/2},\dist(x_\v,\p\o_\v^{j_0})-\lambda\delta^{j_0}\}$. In order to estimate $\mu_\v(\Ring(x_\v,\kappa,\lambda\delta^{3/2}),1)$, we divide $\Ring(x_\v,\kappa,\lambda\delta^{3/2})$ into
\[
\Ring(x_\v,\kappa,\kappa'+2\lambda\delta^{j_0})\cup{\Ring(x_\v,\kappa'+2\lambda\delta^{j_0},\textsf{\ae}+2\lambda\delta^{j_0})}\cup{\Ring(x_\v,\textsf{\ae}+2\lambda\delta^{j_0},\textsf{\ae})}\cup\Ring(x_\v,\textsf{\ae},\lambda\delta^{3/2}).
\]
From Lemmas \ref{Lestimatesurdescercle} $\&$ \ref{L.LargeCircleSmallInclusionLemma} we have
\[
\mu_\v(\Ring(x_\v,\kappa,\kappa'+2\lambda\delta^{j_0}),1)\geq\pi[1+\mathcal{O}(\lambda)]\ln\frac{\kappa}{\kappa'+2\lambda\delta^{j_0}}.
\]
Note that $\dist(\Ring(x_\v,\kappa'+2\lambda\delta^{j_0},\textsf{\ae}+2\lambda\delta^{j_0}),\o_\v^{j_0})\geq\lambda\delta^{j_0}$ and if for some $j$ we have $\Ring(x_\v,\kappa'+2\lambda\delta^{j_0},\textsf{\ae}+2\lambda\delta^{j_0})\cap\o_\v^{j}\neq\emptyset$, then $\dist(x_\v,\o_\v^j)\geq\delta^{j}$ (because $x_\v\in B(y_0,\delta^{j_0})$). Therefore, using Proposition \ref{P8.UepsCloseToaeps} and Lemmas \ref{Lestimatesurdescercle} $\&$ \ref{L.LargeCircleSmallInclusionLemma} we get
\[
\mu_\v(\Ring(x_\v,\kappa'+2\lambda\delta^{j_0},\textsf{\ae}+2\lambda\delta^{j_0}),1)\geq\pi[1+\mathcal{O}(\lambda)]\ln\frac{\kappa'+2\lambda\delta^{j_0}}{\textsf{\ae}+2\lambda\delta^{j_0}}.
\]

It is obvious that 
\[
\mu_\v(\Ring(x_\v,\textsf{\ae}+2\lambda\delta^{j_0},\textsf{\ae}),1)\geq b^2\pi\ln\frac{\textsf{\ae}+2\lambda\delta^{j_0}}{\textsf{\ae}}\geq b^2\pi\ln\left(1+2\delta^{j_0-3/2}\right)=o_\v(1).
\]
By definition of $\textsf{\ae}$, from Proposition \ref{P8.UepsCloseToaeps}, we have
\[
\mu_\v(\Ring(x_\v,\textsf{\ae},\lambda\delta^{3/2}),1)\geq\pi\ln\frac{\textsf{\ae}}{\lambda\delta^{3/2}}-o_\v(1).
\]
Summing these lower bounds we have
\begin{eqnarray*}
\mu_\v(\Ring(x_\v,\kappa,\lambda\delta^{3/2}),1)&\geq&\pi[1+\mathcal{O}(\lambda)]\left[\ln\frac{\kappa}{\kappa'+2\lambda\delta^{j_0}}+\ln\frac{\kappa'+2\lambda\delta^{j_0}}{\textsf{\ae}+2\lambda\delta^{j_0}}+\ln\frac{\textsf{\ae}}{\lambda\delta^{3/2}}\right]+o_\v(1)
\\&\geq&\pi[1+\mathcal{O}(\lambda)]\ln\frac{\kappa}{\lambda\delta^{3/2}}+o_\v(1)\\&=&\pi[1+\mathcal{O}(\lambda)]\left(\frac{3}{2}|\ln\delta|+\ln\frac{\kappa}{\lambda}\right)+o_\v(1)
\end{eqnarray*}
and therefore $\mu_\v(\Ring(x_\v,\kappa,\lambda\delta^{3/2}),1)-\inf_{x_0\in\O}\mu_\v(\Ring(x_0,\kappa,\lambda\delta^{3/2}),1)\to+\infty$ (because $0\leq\ln(\kappa/\lambda)\leq|\ln\delta|$ and from \eqref{InClmlksjdfhqskdljfhDiff3}).

We now assume that $x_\v\in\o_\v$. Because $j_0\geq2$ and $x_\v\in B(y_0,\lambda\delta^{j_0})$, we have $B(y_0,2\lambda\delta^{j_0})\cap\Ring(x_\v,\kappa,\lambda\delta^{3/2})=\emptyset$. Therefore, from the dilution of the inclusion, if there is $\tilde{\o}_\v$, a connected component of $\o_\v^j$ s.t. $\Ring(x_\v,\kappa,\lambda\delta^{3/2})\cap\tilde{\o}_\v$, then $\dist(x_\v,\tilde{\o}_\v)\geq\delta^j/3$. Consequently, from Lemmas \ref{Lestimatesurdescercle} $\&$ \ref{L.LargeCircleSmallInclusionLemma}, we have
\[
\mu_\v(\Ring(x_\v,\kappa,\lambda\delta^{3/2}),1)\geq \pi[1+\mathcal{O}(\lambda)]\ln\frac{\kappa}{\lambda\delta^{3/2}}=\pi[1+\mathcal{O}(\lambda)]\left(\frac{3}{2}|\ln\delta|+\ln\frac{\kappa}{\lambda}\right).
\]
From \eqref{InClmlksjdfhqskdljfhDiff3}, we obtain that $\mu_\v(\Ring(x_\v,\kappa,\lambda\delta^{3/2}),1)-\inf_{x_0\in\O}\mu_\v(\Ring(x_0,\kappa,\lambda\delta^{3/2}),1)\to+\infty$.

We deduce that if $(x_\v)_\v$ satisfies \eqref{InClmlksjdfhqskdljfhDiff4}, then Case 2. does not occur.
\\{\bf Case 3.} We denote $y_0:=y_{i,1}^\v\in\WM_1^\v$ be s.t. $x_\v\in B(y_0,\delta)$.

On the one hand, if $\kappa\leq 10^{-2}\delta$, then we have $\mu_\v(\Ring(x_\v,\kappa,\delta),1),\mu_\v(\Ring(y_0,\kappa,\delta),1)\leq2\pi\ln10$. 

On the other hand, if $\kappa> 10^{-2}\delta$, then we have $\Ring(y_0,\kappa,10\delta)\subset \Ring(x_\v,10\kappa,10^{-1}\delta)$ and thus (using Proposition \ref{P7.MyrtoRingDegDir}) we get
\begin{eqnarray*}
\mu_\v(\Ring(x_\v,\kappa,\delta),1)&\geq& \mu_\v(\Ring(x_\v,10\kappa,10^{-1}\delta),1)-(2\pi\ln10+C_b)
\\&\geq&\mu_\v(\Ring(y_0,\kappa,10\delta),1)-(2\pi\ln10+C_b)
\\&\geq&\mu_\v(\Ring(y_0,\kappa,10\delta),1)-(3\pi\ln10+2C_b).
\end{eqnarray*}
Moreover, following the argument of Subsection \ref{SubsectionProofjfhsldkjfhsldkjfh}, we have (because $\dfrac{\lambda\delta^{3/2}}{\lambda\delta}\to0$)
\[
\mu_\v(\Ring(x_\v,\delta,\lambda\delta^{3/2}),1)- \mu_\v(\Ring(y_0,\delta,\lambda\delta^{3/2}),1)\to+\infty.
\]
Therefore we have the existence of $H_\v\to+\infty$ as $\v\to0$ s.t.
\begin{eqnarray*}
\mu_\v(\Ring(x_\v,\kappa,\lambda\delta^{3/2}),1)&\geq&\mu_\v(\Ring(x_\v,\kappa,\delta),1)+\mu_\v(\Ring(x_\v,\delta,\lambda\delta^{3/2}),1)
\\&\geq&\mu_\v(\Ring(y_0,\kappa,\delta),1)+\mu_\v(\Ring(y_0,\delta,\lambda\delta^{3/2}),1)+H_\v
\\\text{(Prop. \ref{P8.DirectPropAnnProb}.3)}&\geq&\mu_\v(\Ring(y_0,\kappa,\lambda\delta^{3/2}),1)+H_\v-2C_b.
\end{eqnarray*}
Consequently $\mu_\v(\Ring(x_\v,\kappa,\lambda\delta^{3/2}),1)-\inf_{x_0\in\O}\mu_\v(\Ring(x_0,\kappa,\lambda\delta^{3/2}),1)\to+\infty$ and  since $(x_\v)_\v$ satisfies \eqref{InClmlksjdfhqskdljfhDiff4}, Case 3 does not occur.
\vspace{2mm}
\\{\bf Step 3.} We study $\inf_{x_0\in\R^2}\mu_\v(\Ring(x_0,\lambda\delta^{3/2},\rho),1)$ and we conclude
\vspace{2mm}

It is obvious that $\inf_{x_0\in\R^2}\mu_\v(\Ring(x_0,\lambda\delta^{3/2},\rho),1)=\pi b^2\ln\dfrac{\lambda\delta^{3/2}}{\rho}+o_\v(1)$.
Now we are in position to conclude. On the one hand, from the previous steps, for  $\eta_0,c>0$ and a configuration of points/degrees $\{{\bf x}_\v,{\bf 1}\}=\{(x_1^\v,1),...,(x_d^\v,1)\}$ s.t. $|x_i^{\v}-x_j^{\v}|,\dist(x_i^{\v},\p\O)\geq\eta_0$ and $B(x_i^{\rho,\v},c\lambda\delta)\subset\o_\v^1$ for all $i\neq j$, $i,j\in\{1,...,N\}$, we have $\hat{\I}_{\rho,\v}({\bf x}_\v)=I_{\rho,\v}+\mathcal{O}(1)$.

On the other hand, for $\v_n\downarrow0$, if either there is $i\in\{1,...,N\}$ s.t. $d_i>1$ or $\dist(x_i^{\rho,\v},\p\O)\to0$ or there are $i\neq j$ s.t. $|x_i-x_j|\to0$, then the configuration of points/degrees cannot be almost minimal for $I_{\delta,\v_n}$ and thus it cannot be almost minimal for $I_{\rho,\v_n}$. 

Moreover, if there is $i$ s.t. $x_i^{\v_n}\notin\o_{\v_n}^1$ or $\dist(x_i^{\v_n},\p\o_{\v_n}^1)/(\lambda\delta)\to0$, then $(x_i^{\v_n})_n$ cannot be an almost minimal configuration for $\mu_\v(\Ring(\cdot,\kappa,\lambda\delta^{3/2}),1)$. And thus $\{{\bf x},{\bf d}\}$ cannot be an almost minimal configuration for $I_{\rho,\v_n}$.

Therefore Assertions 1. and 2. of Proposition \ref{P8.CorolSecondAuxPb} holds.
 
The rest of the proposition is obtained exactly as Corollary \ref{C8.AnAlmostMinConfigIsAnAlmostMinConf}.
\section{Proof of Proposition \ref{P8.asymptotiquedegnonzero}}\label{S8.ProofOfResultMyrto} 

We use the {\it{unfolding operator}} (see \cite{CDG1}, definition 2.1). We define, for $\O_0\subset\R^2$ an open set, $p\in(1,\infty)$ and $\delta>0$:
\[
\begin{array}{cccc}
\mathcal{T}_\delta:&L^p(\O_0)&\to&L^p(\O_0\times \tilde{Y})
\\
&\phi&\mapsto&\mathcal{T}_\delta(\phi)(x,y)=\left\{\begin{array}{cl}\phi\left(\delta\left[\displaystyle\frac{x}{\delta}\right]+\delta y\right)&\text{for }(x,y)\in \tilde{\O}_\delta^{\rm incl}\times \tilde{Y}\\0&\text{for }(x,y)\in\Lambda_\delta\times \tilde{Y}\end{array}\right.
\end{array}
\]
and
\[
\tilde{Y}=(0,1)\times(0,1),\hspace{10mm} \tilde{\O}^{\rm incl}_\delta:=\bigcup_{\substack{{\tilde{Y}}^K_\delta\subset\O_0\\\tilde{Y}^K_\delta=\delta(K+\tilde{Y}),\,K\in\Z^2}}\overline{\tilde{Y}^K_\delta},
\]
\[
\Lambda_\delta:=\O_0\setminus\tilde{\O}^{\rm incl}_\delta\hspace{5mm}\text{ and }\hspace{5mm}\left[\frac{x}{\delta}\right]:=\left(\left[\frac{x_1}{\delta}\right],\left[\frac{x_2}{\delta}\right]\right)\in\Z^2.
\]
Here, for $s\in\R$, $\displaystyle [s]$ is the integer part of $s$.

We will use the following results:
\begin{equation}\label{8.Unfoldingproperty1}
\text{$\mathcal{T}_\delta$ is linear and continuous, of norm  at most $1$ (\cite{CDG1}, Proposition 2.5)},
\end{equation}
\begin{equation}\label{8.Unfoldingproperty2}
\mathcal{T}_\delta(\phi\psi)=\mathcal{T}_\delta(\phi)\mathcal{T}_\delta(\psi) 
\text{ (\cite{CDG1}, equation (2.2))},
\end{equation}
\begin{equation}\label{8.Unfoldingproperty4} 
\text{$\delta\mathcal{T}_\delta(\n\phi)(x,y)=\n_y\mathcal{T}_\delta(\phi)(x,y)$ for $\phi\in W^{1,p}(\O_0)$ (\cite{CDG1}, equation (3.1)),}
\end{equation}
\begin{equation}\label{8.Unfoldingproperty3}
\text{for }\phi\in L^1(\O_0),\,\text{ we have }\int_{\tilde{\O}^{\rm incl}_\delta}{\phi}=\int_{\O_0\times \tilde{Y}}{\mathcal{T}_\delta(\phi_\delta)}\text{ (\cite{CDG1}, Proposition. 2.5 (i))}.
\end{equation}
If $\phi_\delta\in H^1(\O_0)$ is such that $\phi_\delta\weak \phi_0$ in $H^1$, then, up to subsequence, there exists $\hat{\phi}\in L^2(\O_0,H^1_{\rm per}(\tilde{Y}))$ s.t.:
\begin{equation}\label{8.Unfoldingproperty5}
\mathcal{T}_\delta(\phi_\delta)\to \phi_0\text{ and }\mathcal{T}_\delta(\n\phi_\delta)\weak \n \phi_0+\n_y\hat{\phi}\text{ in $L^2(\O_0\times \tilde{Y})$ (\cite{CDG1}, Theorem 3.5)}.
\end{equation}
Here $H^1_{\rm per}(\tilde{Y})$ stands for the set of functions $\phi\in H^1(\tilde{Y})$ s.t. the extending  of $\phi$ by $\tilde{Y}$-periodicity is in $H^1_{\rm loc}(\R^2)$ (see \cite{CD1}, section 3.4).

In order to define properly the homogenized matrix $\mathcal{A}$  we recall a classical result (see Theorem 4.27 in \cite{CD1}).
\begin{prop}\label{P8.CiornescuDonato1}
Let $H_0\in L^\infty(\tilde{Y},[b^2,1])$. For all $f\in (H^1_{\rm per}(\tilde{Y}))'$ s.t. $f$ annihilates the constants there exists a unique solution $h\in H^1_{\rm per}(\tilde{Y})$ of
\[
\Div(H_0\n_y h)=f\text{ and }\mathcal{M}_{\tilde{Y}}(h)=\int_{\tilde{Y}} h=0.
\]
\end{prop}
Using the previous theorem we denote $\chi_j\in H^1_{\text{per}}(\tilde{Y})$ the unique solution of
\begin{equation}\label{8.chijequ}
\Div(H_0\n_y \chi_j)=\p_{y_j}(H_0)\text{ and }\mathcal{M}_{\tilde{Y}}(\chi_j)=0.
\end{equation}
With these auxiliary functions, we can give an explicit expression of $\mathcal{A}$ the homogenized matrix of $H_0(\frac{\cdot}{\delta}){\rm Id}_{\R^2}$ (see Theorem 6.1 in \cite{CD1}): 
\begin{equation}\nonumber
\mathcal{A}=\int_{\tilde{Y}}{H_0\left(\begin{array}{cc}1-\p_{y_1}\chi_1&-\p_{y_1}\chi_2\\-\p_{y_2}\chi_1&1-\p_{y_2}\chi_2\end{array}\right)}=\int_{\tilde{Y}} H_0({\rm Id}_{\R^2}-\n_y\chi),\,\chi=(\chi_1,\chi_2).
\end{equation}
For the convenience of the reader we restate, in larger detail, Proposition \ref{P8.asymptotiquedegnonzero}.

\hspace{2mm}\\
{\bf Proposition.} \emph{
Let $\O_0\subset\R^2$ be a smooth bounded open set and let $v_n\in H^2(\O_0,\C)$ be s.t.
\begin{enumerate}
\item $|v_n|\leq1$ and $\displaystyle\int_{\O_0}(1-|v_n|^2)^2\to0$,
\item $v_n\weak v_*$ in $H^1(\O_0)$ and $v_*\in H^1(\O_0,\S^1)$,
\item there is $H_n \in W^{1,\infty}(\O_0,[b^2,1])$ and $\delta_n\downarrow0$ s.t. $\mathcal{T}_{\delta_n}(H_n)\to H_0$ in $L^2(\O_0\times \tilde{Y})$ with $H_0$ independent of $x\in\O_0$,
\item $-{\rm div}(H_n\n v_n)=v_nf_n(x)$, $f_n\in L^\infty(\O_0,\R)$.
\end{enumerate}
Then $v_*$ is a solution of
\begin{equation}\nonumber
-\Div(\mathcal{A}\n v_*)=(\mathcal{A}\n v_*\cdot\n v_*)v_*.
\end{equation}
Here $\mathcal{A}$ is the homogenized matrix of $H_0(\frac{\cdot}{\delta}){\rm Id}_{\R^2}$ given by
\begin{equation}\nonumber
\mathcal{A}=\int_{\tilde{Y}}{H_0\left(\begin{array}{cc}1-\p_{y_1}\chi_1&-\p_{y_1}\chi_2\\-\p_{y_2}\chi_1&1-\p_{y_2}\chi_2\end{array}\right)}.
\end{equation}
}
\begin{proof}
In order to keep notations simple, we write, in what follows, $\delta$ rather than $\delta_n$.

Since $f_n$ is real valued, we have that ${\rm div}(H_n\n v_n)\times v_n=0$. From \eqref{8.Unfoldingproperty1} and \eqref{8.Unfoldingproperty2}, we obtain
\begin{equation}\label{8.FirstEstimateProofMyrto}
{\rm div}_y\left[\mathcal{T}_\delta(H_n)(x,y)\mathcal{T}_\delta (\n v_n)(x,y)\right]\times \mathcal{T}_\delta (v_n)(x,y)=0\text{ in }\O_0\times \tilde{Y}.
\end{equation}
Note that from the assumptions and \eqref{8.Unfoldingproperty1},\eqref{8.Unfoldingproperty5}, passing to a subsequence, there is $\hat{w}\in L^2(\O_0,H^1_{\rm per}(\tilde{Y}))$ s.t.
\[
\mathcal{T}_\delta ( v_n)(x,y)\to v_*(x), \,\mathcal{T}_\delta (\n v_n)(x,y)\weak\n v_*(x)+\n_y \hat{v}(x,y)\text{ in }L^2(\O_0\times \tilde{Y})
\]
and
\[
\mathcal{T}_\delta (H_n)(x,y)\to H_0(y)\text{ in }L^2(\O_0\times \tilde{Y}).
\]
Thus we obtain the convergence:
\[
{\rm div}_y\left[\mathcal{T}_\delta(H_n)(x,y)\mathcal{T}_\delta (\n v_n)(x,y)\right]\times \mathcal{T}_\delta (v_n)(x,y)\weak{\rm div}_y\left[H_0(\n v_*+\n_y\hat{v})\right]\times v_*\text{ in }L^2(\O_0\times H^{-1}(\tilde{Y})).
\]
Consequently,
\[
{\rm div}_y\left[H_0(\n v_*+\n_y\hat{v})\right]\times v_*=0.
\]
Since $v_*$ is independent of $y\in \tilde{Y}$, the previous assertion is equivalent to
\[
-{\rm div}_y\left[H_0\n_y( \hat{v}\times v_*)\right]=(\n_y H_0\cdot \n v_*)\times v_*,
\]
which in turn is equivalent to
\[
-{\rm div}_y\left[H_0\n_y( \hat{v}\times v_*)\right]=\sum_i\p_{y_i} H_0(\p_i v_*\times v_*).
\]
Hence, from Proposition \ref{P8.CiornescuDonato1} and \eqref{8.chijequ}, we obtain
\begin{equation}\label{8.EquationGivenByChi}
\hat{v}\times v_*=-\sum_i \chi_i(\p_i v_*\times v_*)=-\chi\cdot(\n v_*\times v_*),\,\chi=\left(\chi_1,\chi_2\right).
\end{equation}
Let $\psi\in \mathcal{D}(\O_0)$ and $n$ sufficiently large s.t. ${\rm Supp}(\psi)\subset\tilde{\O}^{\rm incl}_\delta$. Since $-\Div\left[H_n\n v_n\times v_n\right]=0$, we have
\[
\int_{\tilde{\O}^{\rm incl}_\delta}H_n\n v_n\times v_n\cdot\n\psi=0.
\]
This identity combined with \eqref{8.Unfoldingproperty3} implies that
\[
\int_{\O_0\times \tilde{Y}}\mathcal{T}_\delta[H_n(\n v_n\times v_n)\cdot\n\psi]=0.
\]
Therefore, using \eqref{8.Unfoldingproperty4}  and \eqref{8.Unfoldingproperty5}, we obtain:
\begin{eqnarray*}
0=\int_{\O_0\times \tilde{Y}}\mathcal{T}_\delta\left[H_n(\n v_n\times v_n)\cdot\n\psi\right]&=&\int_{\O_0\times \tilde{Y}}\mathcal{T}_\delta(H_n)\mathcal{T}_\delta(\n v_n)\times \mathcal{T}_\delta(v_n)\cdot\mathcal{T}_\delta(\n\psi)
\\&\underset{n\to\infty}{\to}&\int_{\O_0\times \tilde{Y}}H_0\left[\n v_*\times v_*+\n_y(\hat{v}\times v_*)\right]\cdot\n\psi.
\end{eqnarray*}
Finally, for all $\psi\in\mathcal{D}(\O_0)$, using \eqref{8.EquationGivenByChi}, we have
\begin{eqnarray*}
0=\int_{\O_0\times \tilde{Y}}H_0\n v_*\times v_*\left[{\rm Id}_{\R^2}-\n_y\chi\right]\cdot\n\psi&=&\int_{\O_0}\left(\left\{\int_{\tilde{Y}}H_0\left[{\rm Id}_{\R^2}-\n_y\chi\right]\right\}\n v_*\times v_*\right)\n\psi
\\&=&-\int_{\O_0}-\Div\left(\mathcal{A}\n v_*\times v_*\right)\psi.
\end{eqnarray*}
Here $\displaystyle \mathcal{A}=\int_{\tilde{Y}}H_0\left({\rm Id}_{\R^2}-\n_y\chi\right)$.

Thus $-\Div\left(\mathcal{A}\n v_*\times v_*\right)=0$. Note that, since $H_0$ and $\chi$ are independent of $x$, $\mathcal{A}$ is a constant matrix. This fact combined with the equation $-\Div\left(\mathcal{A}\n v_*\times v_*\right)=0$ implies that $v_*$ satisfies 
\begin{equation}\label{7.EquaVstarslkjgh}
-\Div(\mathcal{A}\n v_*)=(\mathcal{A}\n v_*\cdot\n v_*)v_*.
\end{equation}
 Indeed, we can always consider $\f_*$ which is locally defined in $\O_0$ and whose gradient is globally defined and in $L^2(\O_0,\R^2)$ s.t. $v_*={\rm e}^{\imath\f_*}$.

Since $v_*\times\n v_*=\n\f_*$ we obtain that $\Div(\mathcal{A}\n\f_*)=0$. Identity \eqref{7.EquaVstarslkjgh} follows from the equation of $\varphi_*$ and the fact that $|\n\f_*|^2=|\n v_*|^2$.


\end{proof}


\bibliography{BiblioMain}
\end{document}